\documentclass[a4paper]{amsart}

\usepackage[leqno]{amsmath}
\usepackage[T1]{fontenc}
\usepackage[utf8]{inputenc}
\usepackage{lmodern}
\usepackage[foot]{amsaddr}
\usepackage{amsthm}
\usepackage{amssymb}
\usepackage{graphicx}
\usepackage{calc} 
\usepackage{mathtools}
\usepackage{enumitem}
\usepackage{ifpdf} 
\usepackage{ifthen}
\usepackage{braket}
\usepackage{xcolor}
\usepackage[all]{xy}
\usepackage{dsfont}
\usepackage{todonotes}
\usepackage{subcaption}
\ifpdf
  \usepackage[pdftex]{hyperref}
  \hypersetup{plainpages=false,unicode=true,pdffitwindow=true}
\else
  \usepackage[dvips]{hyperref}
\fi
\usepackage[makeroom]{cancel}  
\usepackage[%
backend=biber,%
maxbibnames=9,%
style=alphabetic,%
natbib=true,%
url=false,%
doi=true,%
isbn=false,%
eprint=true,%
firstinits=true]{biblatex}%
\numberwithin{equation}{section}

\theoremstyle{plain}
\newtheorem{theorem}{Theorem}[section]
\newtheorem{lemma}[theorem]{Lemma}

\theoremstyle{definition}
\newtheorem{definition}[theorem]{Definition}
\newtheorem{example}[theorem]{Example}
\newtheorem{assumption}[theorem]{Assumption}
\newtheorem{remark}[theorem]{Remark}

\newcommand{\E}{\mathbb{E}}
\newcommand{\Pb}{\mathbb{P}}
\newcommand{\R}{\mathbb{R}}

\newcommand{\f}[2]{\frac{#1}{#2}}
\newcommand{\pa}{\partial}
\newcommand{\half}{\f{1}{2}}

\newcommand{\dd}{\mathrm{d}} 

\newcommand{\abs}[1]{\left\lvert#1\right\rvert} 
\newcommand{\indic}[1]{\mathds{1}_{#1}} 
\newcommand{\ip}[2]{\left\langle #1\,,#2\right\rangle} 
\newcommand{\floor}[1]{\left\lfloor#1\right\rfloor} 

\DeclareMathOperator*{\argmin}{arg\, min}

\title[SOC under discrete time partial observations]{Continuous time Stochastic optimal control under discrete time partial observations}

\author[Ch.Bayer]{Christian Bayer$^1$}
\address{$^1$Weierstrass Institute for Applied Analysis and Stochastics (WIAS), Berlin, Germany.}
\author[B.Djehiche]{Boualem Djehiche$^2$}
\address{$^2$KTH, School of Engineering Sciences (SCI), Mathematics (Dept.), Mathematical Statistics.}
\author[E.Rezvanova]{Eliza Rezvanova$^3$}
\address{$^3$King Abdullah University of Science and Technology (KAUST), Computer, Electrical and Mathematical Sciences \& Engineering Division (CEMSE), Thuwal, Saudi Arabia.}
\author[R.Tempone]{Ra\'{u}l F.~Tempone$^3,^4,^5$}
\address{$^4$Chair of Mathematics for Uncertainty Quantification, RWTH Aachen University, Aachen, Germany.}
\address{$^5$Alexander von Humboldt Professor in Mathematics for Uncertainty Quantification, RWTH Aachen University, Aachen  Germany.}

\begin{document}

\begin{abstract}
This work addresses stochastic optimal control problems where the unknown state evolves in continuous time while partial, noisy, and possibly controllable measurements are only available in discrete time. We develop a framework for controlling such systems, focusing on the measure-valued process of the system's state and the control actions that depend on noisy and incomplete data. Our approach uses a stochastic optimal control framework with a probability measure-valued state, which accommodates noisy measurements and integrates them into control decisions through a Bayesian update mechanism. We characterize the control optimality in terms of a sequence of interlaced Hamilton Jacobi Bellman (HJB) equations coupled with controlled impulse steps at the measurement times. For the case of Gaussian-controlled processes, we derive an equivalent HJB equation whose state variable is finite-dimensional, namely  the state's mean and covariance. We demonstrate the effectiveness of our methods through numerical examples. These include control under perfect observations, control under no observations, and control under noisy observations. Our numerical results highlight significant differences in the control strategies and their performance, emphasizing the challenges and computational demands of dealing with uncertainty in state observation.

\end{abstract}

\keywords{Markov Decision Process, Stochastic optimal control,
Filtering, Partially Observed state,
Bayesian updates,
Measure--valued process}

\subjclass[2020]{60H10, 60H35, 93E20, 49L20, 93E11, 60G35}
%
\maketitle

\section{Introduction}
\label{sec:introduction}

This work studies \emph{Stochastic Optimal Control} (SOC) problems where the state evolves in continuous time but observations are available only in discrete time. The focus is on probability measure-valued processes and control actions dependent on noisy and incomplete data.
SOC is a vital area of study in both theoretical and applied mathematics, finding applications in finance, engineering, and various sciences. This field deals with the challenge of making optimal decisions in systems whose dynamics are driven by random processes.

In optimal control, we aim to determine a control strategy for a dynamical system that minimizes a given cumulative running cost function over time, plus a terminal cost, within a finite time horizon.
As general references, we refer to books \cite{FlemingSoner1993} (providing a comprehensive introduction to the theory of SOC, focusing on Hamiltonian systems and Hamilton-Jacobi-Bellman (HJB) equations) and \cite{BardiCapuzzoDolcetta1997} (exploring the connection between Optimal Control and the viscosity solutions of HJB equations).

In many applications which motivate our work,  the system's state evolves continuously over time, but the current state of the controlled system is not perfectly known. 
This discrepancy between the continuous evolution of the state and 
imperfect observations presents a significant challenge for control strategies.
The books \cite{Bensoussan1992,Mao2006,Krishnamurthy16} address control problems where the system state is only partially observable, using a filtering framework and applying it to problems of control under uncertainty.

In this work, we concentrate on the important special case that
\begin{itemize}
\item the state process evolves in continuous time on a finite time interval, but is observed in discrete time;
\item the control on the state process acts in continuous time, incurring running as well as terminal costs to be minimized;
\item the (noisy, partial) observations incur costs, as well, which are taken into account in the minimization problem.
\end{itemize}

As discussed below, the literature on SOC under partial observation mainly concentrates on cases where both control and observations happen in continuous or discrete time.
Moreover, the observation process is usually not considered to incur costs.
The mixed case studied in this work is highly relevant in many applications where the ``observation'' corresponds to an actual, physical measurement, especially one requiring manual intervention.

As a practical application of such stochastic control problems, consider the treatment of a disease guided by regular diagnostic tests. The disease evolves continuously, and the medical doctor must decide when and how to treat the patient and when and which kinds of medical tests to prescribe. Each step incurs costs that should be accounted for when optimizing the patient's health. Other applications include automated trading systems in finance, where the market state evolves continuously, but traders only have discrete and noisy observations of market indicators. The goal is to optimize trading strategies to maximize profit or minimize risk. In autonomous vehicle navigation within robotics, the vehicle's position and environment evolve continuously, but sensors provide discrete, incomplete, and noisy measurements. The control strategy aims to navigate the vehicle safely and efficiently. In epidemiology, disease spreads continuously, but health officials only obtain periodic and potentially noisy data through tests of a relatively small number of individuals. The objective is to control the spread by optimizing costly interventions like vaccination or quarantine. Lastly, in industrial process control, the state of a process evolves continuously, but observations from sensors are discrete and noisy. The goal is to control the process to ensure product quality while minimizing control and data acquisition costs.

\subsection*{Literature review}
\label{sec:literature-review}

We start by describing the main frameworks for analyzing classical SOC problems.
In the 1960s, Ronald A. Howard \cite{Howard1960} popularized the term ``Markov Decision Processes'' (MDPs) and developed the policy iteration method, which is a fundamental technique for solving MDPs.
Throughout the 1970s and 1980s, MDPs addressed more complex settings, including continuous-time processes and infinite-horizon problems, leading to the development of methods like value iteration and Q-learning. The integration of MDPs with machine learning, particularly reinforcement learning (RL), in the late 1980s and 1990s, see \cite{BertsekasTsitsiklis1996,Szepesvari2010,sutton_barto18}, among others, has led to significant advancements. This includes the development of algorithms like Temporal Difference (TD) learning and the popularization of Q-learning and deep reinforcement learning.
Today, MDPs are central to many applications in artificial intelligence, robotics, economics, and operations research.
The Kalman filter \cite{Kalman1960} became an essential tool for dealing with linear systems with Gaussian noise, while the need for handling non-linear systems led to the development of the Extended Kalman Filter \cite{WelchBishop1995}, Ensemble Kalman Filter \cite{Evensen1994} and, later, the Particle Filter in the 1990s \cite{DoucetFreitasGordon2001}.
The concept of separation principle in control theory, which suggests that control and filtering can be separated in some cases, was a significant result for linear systems but proved challenging in non-linear settings, see \cite{AndersonMoore1979}. This book develops the separation principle in the context of linear systems and discusses its extension to non-linear systems through various filtering techniques.

The main alternative approach is based on the Hamilton-Jacobi-Bellman (HJB) equation. It describes the optimal cost function's evolution in dynamic programming terms. This approach was extensively developed during the 1970s and 1980s.
The application of viscosity solutions to HJB equations, cf. \cite{Lions1982,CrandallLions1983,crandall@1984} by Michael Crandall and Pierre-Louis Lions in the 1980s provided critical mathematical tools for dealing with the challenges posed by the non-linearity and high dimensionality of realistic control problems.
Modern research focuses on bridging the gap between theoretical optimality and practical computability, especially in high-dimensional spaces where traditional methods are computationally infeasible.
Applications now span complex systems in finance, engineering, and networked systems, where uncertainty and partial observability are key concerns.

The research in MDPs and stochastic control of \emph{partially observed} systems continues to be a vibrant field, driven by both theoretical interests and practical applications.

To fix notations, let us assume that we are controlling a system $X_t$ in continuous time $t \in [0,T]$, which is given as the solution of a stochastic differential equation (SDE), symbolically
\begin{equation}
    \label{eq:control_SDE}
    \dd X_t = b(X_t; \alpha) \dd t + \sigma(X_t; \alpha) \dd W_t,
\end{equation}
driven by $m$-dimensional Brownian motion $W$, taking values in $\R^d$.
Here, $\alpha$ denotes the control, taking values in a suitable set -- and being progressively measurable w.r.t.~a suitable filtration. 
Suppose now that rather than $X_t$, we observe a process $Y_t$ satisfying
\[
    \dd Y_t = h(X_t) \dd t + \zeta \dd B_t,
\]
driven by another Brownian motion $B$.
(In this context, $X$ is often referred to as the \emph{signal} process and $Y$ as the \emph{observation} process.)
Given the observations $(Y_s)_{s \in [0,t]}$ up to time $t$, we can first compute the \emph{conditional distribution} $\mu_t$ of the state process $X_t$ at time $t$, i.e., we solve the \emph{filtering problem}.
Note that $\mu_t$ is a measure-valued stochastic process, which is adapted to the filtration generated by the observation process $(Y_t)$.
The control problem can now be re-expressed as a control problem for the conditional distribution, i.e., a measure-valued stochastic optimal control problem.
However, the analysis and numerics for such stochastic optimal control problems is much less understood compared to the standard, finite-dimensional situation.

An important technical tool required for deriving dynamic programming principles or Hamilton-Jacobi-Bellman equations for measure-valued stochastic processes is an appropriate It\^{o} formula.
There has been renewed interest in this problem, mainly coming from mean-field games and control of McKean--Vlasov equations.
A very general It\^{o} formula for measure-valued semi-martingales has recently been derived in \cite{guo_pham_wei23} and references therein.
More related works in the context of mean-field optimal control or optimal stopping, we refer to \cite{talbi_touzi_zhang23a,talbi_touzi_zhang23b,germain_pham_warin22}.

The SOC problem with partial state observation is better understood when controls are \emph{relaxed} i.e., $\alpha_t$ is replaced by a measure, see, for instance, \cite{fleming_nisio84,elkaroui_du_jeanblanc88} for existence results for relaxed optimal controls.
The classical work by Fleming \cite{Fleming1980} introduced measure-valued processes for partially observed control problems, providing a theoretical foundation for analyzing stochastic control problems where the state is not fully observable. Fleming's insights are essential for our analysis of measure-valued processes.
Recent works have made significant strides in approximating SOC problems under partial observation. Tan and Yang \cite{TanYang2023} discuss discrete-time approximation of continuous-time stochastic control problems under continuous time, partial observation. Their methods' convergence properties support our work's theoretical foundation, especially in the context of the measure-valued control framework we discuss. This recent work deals with approximating schemes for filtered problems, which may be relevant to our problem, which is not a standard filtering-control problem and only uses discrete time, partial observations.


The situation changes quite drastically when we consider SOC problems formulated within the field of MDPs.
Indeed, there is a classical and rich literature on so-called \emph{partially observed Markov decision processes} (POMDPs), which, as the name suggests, is concerned with MDPs where only partial, noisy observation of the controlled state process are available.
We refer to \cite{Krishnamurthy16} for a monograph and \cite{baeuerle_rieder17} for an interesting recent application to an economics problem.
As discussed earlier, the strategy is to lift the POMDP by considering the conditional distribution of the unobserved full state.
The resulting control problem for measure-valued (hence, in general, infinite-dimensional) processes is the seen to be a standard MDP, and can be analysed by standard methods.
Note that, in contrast to the stochastic optimal control literature, the MDP literature is almost exclusively concerned with discrete time problems, and this also extends to the partially-observed case.

The problem of SOC under noisy observation is also related to the problem of reinforcement learning \cite{sutton_barto18}.
Indeed, reinforcement learning operates under even less information, since not even knowledge of the driving dynamics of the controlled system is given, but rather has to be learned while controlling the system.
Of course, stationarity of the system is usually assumed.
Reinforcement learning is, again, generically formulated in discrete time, even though some attempts of continuous time extensions have recently been made, see, e.g., \cite{wang_etal_20}.

\subsection*{Our contribution}
\label{sec:our-contribution}

In our work, we consider a different problem, arising from the need to develop more efficient and reliable control strategies for systems with discrete, partial and noisy observations.
Consider a continuous time SOC \eqref{eq:control_SDE} with \emph{partial} and \emph{noisy} observations $Y_{t_i}$ available at (fixed) discrete times $0 < t_1 < t_2 < \cdots t_n < T$.
As before, decisions have to be based on the conditional distribution $\mu_t$ of the (unobserved) state process $X_t$ given all the observations already available to us, i.e., $\set{Y_{t_i} : t_i \le t}$.
The dynamics of $\mu_t$ can now be described as follows:
\begin{enumerate}
\item Between observation times, $\mu_t$ follows the Fokker-Planck equation associated to the process \eqref{eq:control_SDE} -- a \emph{deterministic} dynamics.
\item At an observation time $t = t_i$, we \emph{update} the conditional distribution with the (random) new information $Y_{t_i}$, leading to a random jump $\mu_{t_i} = K_{\epsilon}(\mu_{t_i^-}, Y_{t_i}) \mu_{t_i^-}$, where $K_{\epsilon}$ denotes the Radon--Nikodym derivative of the updated distribution w.r.t.~the distribution prior to the update, and $\epsilon$ indicates the level of the noise in our observation.
\end{enumerate}
We refer to \eqref{eq:dynamics-controlled} for the precise dynamics in the controlled case.
Note that in a statistical sense, the second step can be interpreted as a \emph{Bayesian update} of the \emph{prior distribution} $\mu_{t_i^-}$ to a \emph{posterior distribution} $\mu_{t_i}$ at time $t_i$ taking into account our data $Y_{t_i}$.

As in the case of continuous time observation, we replace the control problem in the unobserved state $X$ by the corresponding control problem in its conditional distribution $\mu_t$.
The dynamics of $\mu_t$ is, however, quite different from the continuous-time case, at least on a formal level.
Rather than solving a Zakai SPDE, we need to solve a deterministic PDE with finitely many stochastic jumps. Of course, this problem is still infinite-dimensional.

The above setup invites to include a second control into our problem.
As already indicated above, measurements are usually noisy (say, with standard deviation $\epsilon$), and may not convey information about the full state $X_t$, anyway, -- think of noisy observations of a function of $X_t$, e.g., just a single component.
However, in many cases, \emph{more precise measurement} methods may be available, albeit at a higher \emph{cost}.
Hence, in addition to the control $\alpha$ guiding the state process $X_t$, we may consider a second control $\beta$ for the measurement method, as well as an associated cost term.
In the statistical literature, this is also known as \emph{optimal experimental design}, see, for instance, \cite{smucker2018optimal}.
Mathematically, this means that we obtain Bayesian updates $\mu_{t_i} = K_{\beta_{t_i}}(\mu_{t_i^-}, Y_{t_i}) \mu_{t_i^-}$ which are directly influenced by the control, not only indirectly via $\mu$ and $Y$.
Of course, this principle may also be applied to the observation times themselves, which could, more generally, be chosen by the controller.

From a computational perspective, we are still faced with an inherently infinite-dimensional SOC, since the conditional distribution $\mu_t$ takes values in the set of probability measures on the underlying state space.
We propose to solve the stochastic optimal control problem numerically under the assumption that $\mu_t$ can be accurately characterized by the expectations $\int \varphi_j \dd \mu_t$ of finitely many test functions $\varphi_j$, $j \in \mathcal{J}$ -- either exactly, or in an approximate sense.
In this case, the HJB equation for our SOC can be reduced to an HJB equation in $\abs{\mathcal{J}}$ space variables and one time variable.

As an example, consider a drift-controlled Ornstein--Uhlenbeck process, in which case all conditional distributions $\mu_t$ are Gaussian, and, hence, can be characterized by their means $m_t$ and co-variances $\sigma_t^2$.
More generally, assume that the conditional distribution $\mu_t$ can be approximated by Gaussians $\mathcal{N}(m_t, \sigma_t^2)$.
We can still solve the corresponding HJB equation, projecting both the dynamics of $\mu_t$ between observation times as well as the Bayesian updates to the set of Gaussian distributions as we go.
This method can be interpreted as a generalization  of the Kalman filter.

\subsection*{Outline of this work}
\label{sec:outline}

We start with a motivating example of a SOC problem in continuous time with discrete time noisy observations in Section~\ref{sec:guiding-example}. In the following Section~\ref{sec:general-theory} we provide a formal setup for the problem, prove that the dynamic programming principle holds and derive the HJB equations under suitable regularity conditions. We then discuss the specific example of a drift-controlled one-dimensional Ornstein--Uhlenbeck process under observations with additive, independent Gaussian noise and quadratic costs, see Section~\ref{sec:optimal-control-an}. We observe that the problem cannot be reduced to solving a system of Riccati ODEs as usual, but we derive the associated HJB equation in time, mean and variance of the underlying Gaussian distribution $\mu_t$.
We then propose an appropriate finite-difference solver for the HJB equation, analyze its behavior and provide numerical examples in Section~\ref{sec:num-approach}. Finally, in Section~\ref{sec:kalman-filters-with} we generalize the HJB equation to the multi-dimensional (approximately) Gaussian case, and make the link to the Kalman filter.


\section{A  motivating example}
\label{sec:guiding-example}

Consider an example based on a controlled Ornstein-Uhlenbeck process, with variance controlled Gaussian noisy measurements at times \( t_i \), \( i=1, \ldots, n \). More precisely, for \( t \ge 0 \), we consider the following controlled SDE representing the unobserved dynamics of the real-valued process \( X \):
\begin{equation}
  \label{eq:OU-SDE}
  dX_t = (-\theta X_t + \alpha_t) \, dt + b \, dW_t,
\end{equation}
where \( \alpha_t \in \mathbb{R} \) is the control and \( X_0 \sim \mu_0 \) is the initial condition.

Let \( \mu^\alpha_t \) denote the random evolution of the conditional law of the unobserved controlled process \( X^\alpha \) under the control process \( \alpha_t \). The natural conditioning event corresponds to all observations made up to time \( t \), i.e., \( t_i \le t \).

Our goal is to minimize the expected cost given by  a running cost, a final cost, and a cost associated to each of the measurements, namely
\begin{equation}
  \label{eq:quadratic-cost0}
  J(\alpha) \coloneqq \mathbb{E} \left[ \int_0^T \left( \int_{\mathbb{R}} x^2 \mu^\alpha_t(\mathrm{d}x) + C \alpha_t^2 \right) dt + \int_{\mathbb{R}} x^2 \mu^\alpha_T(\mathrm{d}x) +  \sum_{i=1}^n \frac{1}{\beta_{i} }\right].
\end{equation}
The last sum term in the above corresponds to the cost of the measurements. More precisely, at each time \( t_i \), we observe a noisy version of $X_{t_i}$,
\[
Y_i \coloneqq X_{t_i} + \beta_i Z_i,
\]
where \( Z_i \sim \mathcal{N}(0, 1) \) are independent of each other and all other sources of randomness, and \( X_{t_i} \sim \mu^\alpha_{t_i^-} \). We assume that the noise level \( \beta_i > 0 \) of the measurement at time $t_i$ is also a control parameter. 
Therefore, at each time \( t_i \), we have to update the conditional distribution \( \mu^\alpha \) according to the Bayesian update:
\begin{equation}
  \label{eq:bayes-update}
  \mu^\alpha_{t_i}(\mathrm{d}y) \propto K_{\beta_i}(y; \mu^\alpha_{t_i^-}, Y_i) \mu^\alpha_{t_i^-}(\mathrm{d}y),
\end{equation}

where \( K_{\beta}(y; \mu^\alpha_{t_i^-}, Y_i) \) denotes the Gaussian likelihood corresponding to the measurement \( Y_i \). It is intuitive to see that smaller values of $\beta_i$ will provide better measurements and thus reduce the values of the first two terms in the objective functional \eqref{eq:quadratic-cost0}. However, these smaller $\beta_i$ values will make the last term in \eqref{eq:quadratic-cost0} larger, indicating that there is a non-trivial tradeoff to address when solving this SOC problem. 

Motivated by this problem, we will state in the next section our SOC problem entirely in terms of the evolution of the conditional distribution, which is an interlaced sequence of controlled time evolutions, governed by the control $\alpha_t$ during intervals $(t_i,t_{i+1})$, with discontinuous jumps at times $t_i$ given by the Bayesian updates, which are controlled by the $\beta_i$. This evolution of $\mu_t$ will be then Markovian and will naturally take us into a sequence of interlaced HJB equations, connected by conditional expectations matching conditions. See Section \ref{sec:general-theory} for more details.

\section{General theory}
\label{sec:general-theory}

For a fixed time horizon $T>0$, consider a large enough filtered probability space $\left(\Omega, \mathcal{F}, \mathbb{F}, \Pb \right)$.
We also denote by $\mathcal{P}(\R^d)$ the set of probability measures on $\left( \R^d, \mathcal{B}(\R^d) \right)$, 
 furnished with the topology of weak convergence. 
We construct a stochastic optimal control problem for a probability-measure-valued process $\mu_t = \mu_t(\dd x; \omega)$ generalizing the example discussed in Section~\ref{sec:guiding-example}.
In particular, we consider a continuous-time version of the problem, where the underlying, unobserved process is a diffusion process.

\subsection{Formulation of the problem}
\label{sec:formulation-problem}

To fix notation, let us first look at the dynamics of the uncontrolled process.
To start out, we fix observation times $0 <  t_1 < \cdots < t_n < T$, which are (for simplicity) assumed to be deterministic.
Likewise, for simplicity we assume that there are no observations at times $0$ and $T$, and we introduce $t_0 \coloneqq 0$, $t_{n+1} \coloneqq T$.
We also introduce the notation $\floor{t} \coloneqq \max\Set{i | t_i < t}$, with the convention that $\floor{0} \coloneqq 0$.
Let $\mathcal{G}$ denote the infinitesimal generator of a $d$-dimensional diffusion process 
defined for $f\in C_b^2(\R^d)$ 
by
\begin{equation}
  \label{eq:generator-G-uncontrolled}
  \mathcal{G}f(y) \coloneqq \sum_{i=1}^d b_i(y) \partial_i f(y) + \half \sum_{i,j=1}^d a_{ij}(y) \partial_{ij}^2 f(y).
\end{equation}
Between observation times, the process $\mu_t$ satisfies the Fokker--Planck equation associated with $\mathcal{G}$, i.e., using the adjoint operator we write
\begin{subequations}
  \label{eq:dynamics-uncontrolled}
\begin{equation}
  \label{eq:FP-uncontrolled}
  \dd \mu_t = \mathcal{G}^\ast \mu_t \dd t, \quad t_i \le t < t_{i+1},\,\, i=1,\ldots, n.
\end{equation}
We note that $\mathcal{G}^\ast \mu_t$ can be seen as a differential operator acting on the density of $\mu_t$ -- tacitly assuming the existence of such a density, as well as identifying the measure with its density. More generally, we can understand $\mathcal{G}^\ast$ as an operator acting on (signed) measures defined in a weak sense. 

At each observation time $t_i$, we obtain a noisy observation $Y_i$, and update the conditional measure $\mu$ according to the \emph{Bayes rule}
\begin{equation}
  \label{eq:observation-update-uncontrolled}
  \mu_{t_i}(\dd x) = K_\varepsilon(x; \mu_{t_i^-}, Y_i) \mu_{t_i^-}(\dd x),\,\, i=1,\ldots, n.
\end{equation}
\end{subequations}
where the kernel $K_\varepsilon$ depends on the noise level $\varepsilon\ge0$ as well as the nature of the measurement procedure.
The following example illustrates a typical choice of measurement procedure and the corresponding choice of the kernel.
\begin{example}[A sequence of measurements]
  \label{ex:gaussian-noisy-observation}
Let the measurement at $t_i$ is given by $Y_i = \widehat{X}_i + \varepsilon Z_i$ for $\widehat{X}_i \sim \mu_{t_i^-}$ and $Z_i \sim \mathcal{N}(0,1)$ independent random variables, where we assume, for simplicity, that $d=1$.
  More specifically, let $(U_i)_{i=1}^n$, $(Z_i)_{i=1}^n$ denote two i.i.d.~sequences of standard uniform and normal r.v.s, respectively. Then  we define
  \begin{equation*}
    \widehat{X}_i \coloneqq F[\mu_{t_i^-}]^{-1}(U_i), \quad Y_i \coloneqq \widehat{X}_i + \varepsilon Z_i,
  \end{equation*}
  where $F$ maps probability measures to their c.d.f.s.
  With $\rho_\varepsilon$ denoting the density of $\mathcal{N}(0, \varepsilon^2)$, the posterior distribution is given by~\eqref{eq:observation-update-uncontrolled} with
  \begin{equation*}
    K_\varepsilon(x; \mu, y) \coloneqq \f{L_\varepsilon(x;y)}{\int_{\R} L_\varepsilon(x;y) \mu(\dd x)}, \quad L_{\varepsilon}(x;y) \coloneqq \rho_\varepsilon(y-x).
  \end{equation*}
\end{example}

As anticipated in Section~\ref{sec:guiding-example}, 
we ensure 
 that the dynamics for $\mu_t$ satisfies the Markov property.
For this reason, we impose
\begin{assumption}[Likelihood structure]
  \label{ass:Markovian-update}
  The observation $Y_i$ at time $t_i$ is a deterministic function of the observation time $t_i$, the probability measure $\mu_{t_i^-}$, the noise level $\varepsilon$, and a random variable $\Gamma_i$ belonging to an sequence of independent r.v.s $(\Gamma_i)_{i=1}^n$ which are independent of all other sources of randomness.
\end{assumption}
See the above Example~\ref{ex:gaussian-noisy-observation} for an illustrative example of an update rule satisfying Assumption~\ref{ass:Markovian-update} -- with $\Gamma_i = (U_i,Z_i)$.
To reflect Assumption~\ref{ass:Markovian-update} in the notation, we may also write $Y_i = Y^{\varepsilon,\mu_{t_i^-}}_i = Y^\varepsilon_i$, depending on the context.

Note that the $\mathcal{P}(\R^d)$-valued Markov process $\mu_t$ defined above has the property that its dynamics between measurement times is deterministic. The sole stochastic effects enter through the measurement values at the  measurement times.
At the observation points, it jumps in a random manner according to the Bayes update \eqref{eq:observation-update-uncontrolled}.

As such, it is easy to see that $\mu_t$ satisfies an "It\^{o}-formula''. Given $\Phi: \mathcal{P}(\R^d) \to \R$ ``nice enough'', we associate to it its \emph{flat derivative}, see, for instance, \cite[Page 18f]{dawson1993measure} or \cite[Appendix F]{kolokoltsov2010nonlinear}.

\begin{definition}[Flat derivative]\label{def:flatD}
  A function $\Phi:\,\, \mathcal{P}(\R^d)\longrightarrow \R$ is said to be differentiable at $\mu$ with derivative $\frac{\delta\Phi}{\delta \mu}: \mathcal{P}(\R^d) \times \R^d \to \R$ if for any $\nu\in \mathcal{P}(\R^d)$,
  \begin{equation}\label{gat}
    \Phi(\mu)-\Phi(\nu)=\int_0^1\langle \mu-\nu,\frac{\delta\Phi}{\delta \mu}(\nu+\lambda(\mu-\nu),\cdot)\rangle \, \dd \lambda.
  \end{equation} 
  Moreover, assuming that $\Phi$ is actually even defined in a neighborhood of $\mu$ in the space of signed measures, we can equivalently require the limit 
  \begin{equation}
    \frac{\delta\Phi}{\delta\mu}(\mu,x):=\underset{\theta\downarrow 0}{\lim} \frac{\Phi(\mu+\theta\delta_x)-\Phi(\mu)}{\theta}=\left.\frac{\dd}{\dd\theta}\right|_{\theta=0}\Phi(\mu+\theta \delta_x)
  \end{equation}
  to exist for every $x\in\R^d$.

  The function $\Phi$ is continuously differentiable if furthermore the function $\R^d\ni x\mapsto\frac{\delta\Phi}{\delta\mu}(\mu,x)$ is continuous in $\R^d$.
\end{definition}

Observe, that according to Definition \ref{def:flatD}, if we have $\Phi(\mu) = F(\ip{\mu}{f})$ for a differentiable function $F:\R\to\R$ and a bounded function $f:\R^d \to \R$, then the flat derivative $\f{\delta \Phi}{\delta \mu}$ always exists and satisfies
\begin{equation}
  \label{eq:example-gateaux-derivative}
  \f{\delta \Phi}{\delta \mu}(\mu,x) = F^\prime(\ip{\mu}{f}) f(x).
\end{equation}


In this context, 
in the next lemma we state  a chain rule for the following class of function.
\begin{definition}[Class $\mathcal{S}^{1,1}(\mathcal{P}(\R^d))$]
We say that a function $\Phi\in \mathcal{S}^{1,1}(\mathcal{P}(\R^d))$ if there is a continuous version of the flat derivative $\f{\delta \Phi}{\delta \mu}(\mu,x)$ such that
\begin{itemize}
\item the mapping $(\mu,x)\mapsto \f{\delta \Phi}{\delta \mu}(\mu,x)$ is jointly continuous w.r.t. $(\mu,x)$,
\item the mapping $x\mapsto \f{\delta \Phi}{\delta \mu}(\mu,x)$ is twice continuously differentiable with bounded first and second order derivatives. 
\end{itemize}
\end{definition}

\begin{lemma} [Chain rule] 
  \label{lem:chain-rule}
  Assume $(\mu_t)_{t \in [0,T]}$ solves \eqref{eq:dynamics-uncontrolled} and that the function $\Phi(t,\mu)$ from $[0,T]\times \mathcal{P}(\R^d)$ to $\R$ is differentiable w.r.t.~the time variable $t$ and is in $\mathcal{S}^{1,1}(\mathcal{P}(\R^d))$ w.r.t.~$\mu$.
  
Then, we have
\begin{multline}\label{Ito}
  \Phi(t,\mu_t)=\Phi(0,\mu_0)+\int_0^t \left[\frac{\partial\Phi}{\partial s}(s,\mu_s)+\langle \mathcal{G}^*\mu_s,\frac{\delta\Phi}{\delta \mu}(s,\mu_s,\cdot)\rangle \right] \dd s \\
+ \sum_{t_i \le t} \left[\Phi\left(t_i,K_\epsilon(\cdot;\mu_{t_i^-},Y_i) \mu_{t_i^-}\right)-\Phi(t_i,\mu_{t_i^-})\right].
\end{multline}
\end{lemma}
\begin{proof}[Sketch of proof, following \cite{guo_pham_wei23}.]
  The proof of the chain rule lemma is first checked for the class of functions $\Phi(\mu):=F(\langle \mu,f\rangle)$ for a differentiable function $F:\,\,\R\longrightarrow \R$ and a bounded function $f:\,\,\R^d\longrightarrow \R$. Then, the result is naturally extended to the class of functions 
  $$
  \Phi(\mu):=F(\langle \mu,f_1\rangle,\langle \mu,f_2\rangle,\dots,\langle \mu,f_n\rangle)
  $$ for any fixed $n\in \mathbb{N}$ and polynomial $F:\,\,\R^n\longrightarrow \R$ and polynomials $f_1,f_2,\ldots, f_n:\,\ \R^d\longrightarrow \R$. Finally, apply the Stone-Weierstrass (density) theorem to conclude that the chain rule is valid for any differentiable function $\Phi(\mu)$. 
\end{proof}



For the controlled dynamics, we consider the situation of the generator $\mathcal{G}$ depending on a control $\alpha$.
In general, we assume that both the drift $b$ and the diffusion $\sigma$ depend on the control $\alpha$.
In addition, we \emph{may} also assume $\varepsilon$ to be a control, which we then also denote by $\beta$ -- i.e., we consider controls $\alpha$ of the dynamics and $\beta$ of the measurement we take.
Hence, we consider $u \coloneqq (\alpha, \beta) \in U \subset 
\R^{k_1} \times \R^{k_2}$.
Consider the filtration $(\mathcal{F}_t)_{t \in [0,T]}$ generated by the observations, i.e.,
\begin{equation*}
  \mathcal{F}_t = \sigma\left( \set{Y_i: t_i \le t} \right).
\end{equation*}
Note that the filtration is constant between observation times and ``jumps'' at the observation times $t_i$ -- we denote $\mathcal{F}_{t_i^-} \coloneqq \mathcal{F}_{t_{i-1}}$.
We recall that the controls $\alpha$ and $\beta$ play very different roles:
\begin{itemize}
\item $\alpha$ acts on the dynamics of the unobserved process in continuous time, and is adapted to the filtration $(\mathcal{F}_t)_{t \in [0,T]}$;
\item $\beta$ only acts as an impulse control at the observation times $t_i$, but can only use information available just before $t_i$ -- hence, it is predictable w.r.t. the filtration $(\mathcal{F}_t)_{t \in [0,T]}$ and not just adapted in general.
\end{itemize}

\emph{Admissible controls} are processes in the following class:
\begin{definition}
  \label{def:admissible-controls}
  Let $\mathcal{U}[0,T]$ denote the set of all controls $u = (\alpha,\beta)$ such that
  \begin{enumerate}
  \item $\alpha$ is adapted w.r.t.~the filtration $(\mathcal{F}_t)_{t \in [0,T]}$.
  \item $\beta$ is a piece-wise constant process, which only jumps at observation times $t_i$, where it is left-continuous. As a process, it is \emph{predictable} w.r.t.~the filtration $(\mathcal{F}_t)_{t \in [0,T]}$.
  \item For any $t$, we have $u_t = (\alpha_t, \beta_t) \in U$.
  \end{enumerate}
  We also introduce the notation $\mathcal{U}[t,T]$, defined in the obvious way.
\end{definition}
Note that $\beta$ only acts at the observation times $t_1, \ldots, t_n$. For this reason, abusing notation, we denote $\beta_i \coloneqq \beta_{t_i}$ -- keeping in mind that $\beta_i \in \mathcal{F}_{t_i^-}$.
Furthermore, $\alpha$ being adapted to $(\mathcal{F}_t)_{t \in [0,T]}$ does, of course, not imply that $\alpha$ is piecewise-constant as well. It only implies that the dynamics of $\alpha_t$ between observation times has to be deterministic -- conditional on all the observation made before time $t$.

We denote by $\mathcal{G}(\alpha)$ the \emph{controlled generator} and by $K_\beta$ the \emph{controlled update}.
The dynamics of the \emph{controlled process} can, hence, be described by
\begin{subequations}
  \label{eq:dynamics-controlled}
  \begin{align}
    \dd \mu_t^u &= \mathcal{G}^\ast(\alpha) \mu_t^u \dd t, \quad t_i \le t < t_{i+1},\\
    \mu_{t_{i+1}}^u(\dd x) &=  K_{\beta_{i+1}}\biggl(x; \mu^u_{t_{i+1}^-}, Y^{\beta_{i+1}, \mu^u_{t_{i+1}^-}}_{i+1}\biggr) \mu^u_{t_{i+1}^-}(\dd x),
  \end{align}
\end{subequations}
for $0\le i <n$ and $\mu^u_0 = \mu_0 \in \mathcal{P}(\R^d)$.
The cost functional associated to the stochastic optimal control problem is given by
\begin{equation}
  \label{eq:cost-functional}
  J(u) \coloneqq \E\left[ \int_0^T \ell\left( t, \mu^u_t, \alpha_t \right) \dd t + \sum_{i=1}^n h(t_i, \mu^u_{t_i^-}, \beta_{i}) + g(\mu^u_T) \right],
\end{equation}
where $\ell$ is the running cost associated to the control of the underlying process (generally only depending on $\alpha$), whereas $h$ is the cost associated to the design of the measurement (generally only depending on $\beta$).
Finally, $g$ is the terminal cost, which in most cases is linear on the measure.

The optimal control problem is to minimize~\eqref{eq:cost-functional} over all admissible controls $u$, subject to the dynamics \eqref{eq:dynamics-controlled}.
An optimizing control is denoted by $u^\ast$ and the corresponding optimal path by $\mu_t^\ast$.
Under the above assumptions, we expect optimal controls to be of feedback form, i.e., $u_t = \theta(t, \mu_t)$.

\subsection{Dynamic programming principle}
\label{sec:dynam-progr-princ}

For given $\mu \in \mathcal{P}(\R^d)$ and $0 \le s < T$, consider the (weak) solution of the controlled system
\begin{subequations}
  \label{eq:controlled-dynamics-s-T}
  \begin{align}
    \dd \mu_t^u &= \mathcal{G}^\ast(\alpha) \mu_t^u \dd t, \quad s \wedge t_i < t < t_{i+1},\\
    \mu_{t_{i+1}}^u(\dd x) &= K_{\beta_{i+1}}\biggl(x; \mu^u_{t_{i+1}^-}, Y^{\beta_{i+1}, \mu^u_{t_{i+1}^-}}_{i+1}\biggr) \mu^u_{t_{i+1}^-}(\dd x),
  \end{align} 
\end{subequations}
(for all $i$ s.t.~$t_i \ge s$) with $\mu^u_s = \mu$, denoted by $\mu_t^u(s,\mu)$ for $s \le t \le T$.
By uniqueness of solutions, we have the \emph{flow property} for $u \in \mathcal{U}[s,T]$ and $s \le t \le v \le T$:
\begin{equation}
  \label{eq:controlled-dynamics-flow}
  \mu^u_v(s,\mu) = \mu^u_v(t, \mu^u_t(s,\mu)).
\end{equation}
The flow property implies that the cost function associated to the control $u$ satisfies
\begin{equation*}\begin{array}{lll}
  J(s, \mu; u|_{[s,T]}) &\coloneqq \E\left[ \int_s^T \ell(t, \mu_t^u(s, \mu), u_t) \dd t + \sum_{i=\floor{s}+1}^n h(t_i, \mu^u_{t_i^-}(s, \mu), \beta_i) + g(\mu^u_T) \right]\\
  &= \E\left[ \int_s^t \ell(r, \mu_r^u(s, \mu), u_r) \dd r + \sum_{i=\floor{s}+1}^{\floor{t}} h(t_i, \mu^u_{t_i^-}(s, \mu), \beta_i)\right. \\ & \qquad\qquad\qquad\qquad \qquad\qquad\qquad\qquad\qquad\qquad\left. + J\left(t, \mu^u_t(s, \mu); u|_{[t,T]} \right) \right].
  \end{array}
\end{equation*}

\begin{definition}
  \label{def:value-function}
  The \emph{value function} $V:[0,T] \times \mathcal{P}(\R^d) \to \R$ is defined by $V(T, \mu) \coloneqq g(\mu)$ and
  \begin{equation*}
    V(s,\mu) \coloneqq \inf_{u \in \mathcal{U}[s,T]} J(s, \mu; u), \quad 0 \le s <T, \quad \mu \in \mathcal{P}(\R^d).
  \end{equation*}
\end{definition}

The value function satisfies the \emph{dynamic programming principle}:
\begin{theorem}[Dynamic Programming Principle]
  \label{thr:dynamic-programming}
  For any $0 \le s < T$, $\mu \in \mathcal{P}(\R^d)$, and $s \le t \le T$ we have
  \begin{equation}\label{eq:DPP}
    V(s, \mu) = \inf_{u \in \mathcal{U}[s,t]} \E\left[ \int_s^t \ell(r, \mu_r^u(s, \mu), u_r) \dd r + \sum_{i=\floor{s}+1}^{\floor{t}} h(t_i, \mu^u_{t_i^-}(s, \mu), \beta_i) + V\left( t, \mu^u_t(s, \mu) \right) \right].
  \end{equation}
  In particular, at jump times $t_i$ -- interpreting $s = t_i^-$ and $t = t_i$ -- we have
  \begin{equation}\label{eq:DPP-jump}
    V(t_i^-, \mu) = \inf_{(\alpha,\beta) \in U} \left\{ h(t_i, \mu, \beta) + \E\left[ V(t_i, K_\beta(\cdot; \mu, Y^{\beta,\mu}_i) \mu) \right] \right\},
  \end{equation}
  where the expectation is effectively taken over $Y^{\beta,\mu}_i$.
\end{theorem}
We note that the objective function in the optimization problem at the observation time $t_i$ stated above does not, in fact, depend on $\alpha$ at all. Nonetheless, minimization over $(\alpha,\beta)$ is used to reflect the fact that the constraint $(\alpha,\beta) \in U$ acts on both components of the control.
\begin{proof}[Proof of Theorem~\eqref{thr:dynamic-programming}]
  The proof is standard. We recall it for convenience. Denote the right-hand side of \eqref{eq:DPP} by $W(s,\mu)$.
  
  For any $\epsilon>0$, there exists a control $u^{\epsilon}\in \mathcal{U}[s,T]$ such that
  \begin{align*}
    V(s,\mu) + \epsilon &\ge J(s,\mu;u^{\epsilon}) \\
               &= \E\Biggl[ \int_s^t \ell(r, \mu_r^{u^\epsilon}(s, \mu), u^\epsilon_r) \dd r + \sum_{i=\floor{s}+1}^{\floor{t}} h(t_i, \mu^{u^\epsilon}_{t_i^-}(s, \mu), \beta_i) \\
    & \quad\quad + J\left(t, \mu^{u^\epsilon}_t(s, \mu); {u^\epsilon}|_{[t,T]} \right) \Biggr] \\
               &\ge  \E\Biggl[ \int_s^t \ell(r, \mu_r^{u^\epsilon}(s, \mu), u^\epsilon_r) \dd r + \sum_{i=\floor{s}+1}^{\floor{t}} h(t_i, \mu^{u^\epsilon}_{t_i^-}(s, \mu), \beta_i) + V(t, \mu_t^{u^\epsilon}(s,\mu) \Biggr] \\
               &\ge W(s,\mu).
  \end{align*}
  
  To obtain the reverse inequality,  we have for any $v \in \mathcal{U}[s,T]$
  \begin{multline*}
    V(s,\mu) \le J(s,\mu;v)=\\
    \E\left[ \int_s^t \ell(r, \mu_r^{v}(s, \mu), v_r) \dd r + \sum_{i=\floor{s}+1}^{\floor{t}} h(t_i, \mu^{v}_{t_i^-}(s, \mu), \beta_i) +J(t,\mu^{v}_t(s,\mu);v)\right].
  \end{multline*}
  In particular, given arbitrary controls $u = (\alpha, \beta), u^\prime = (\alpha^\prime, \beta^\prime) \in \mathcal{U}[s,T]$, by choosing
  \begin{equation*}
    \nu(r)=\begin{cases} u(r), & r\in [s,t], \\
      u^\prime(r), & r\in (t,T],        
    \end{cases}
  \end{equation*}
  we have
  \begin{multline*}
    \int_s^t \ell(r, \mu_r^{v}(s, \mu), v_r) \dd r + \sum_{i=\floor{s}+1}^{\floor{t}} h(t_i, \mu^{v}_{t_i^-}(s, \mu), \beta_i) = \\
    \int_s^t \ell(r, \mu_r^{u}(s, \mu), u_r) \dd r + \sum_{i=\floor{s}+1}^{\floor{t}} h(t_i, \mu^{u}_{t_i^-}(s, \mu), \beta_i),
  \end{multline*}
  and by the flow property \eqref{eq:controlled-dynamics-flow}, we obtain
  \[
    J(t,\mu^{v}_t(s,\mu);\nu) = J(t,\mu^{u}_t(s,\mu;t);u^\prime).
  \]
  Therefore, by taking the infimum over $u^\prime \in \mathcal{U}[s,T]$ we obtain 
  \begin{equation*}
    V(s,\mu) \le \E\left[\int_s^t \ell(r, \mu_r^{u}(s, \mu), u_r) \dd r + \sum_{i=\floor{s}+1}^{\floor{t}} h(t_i, \mu^{u}_{t_i^-}(s, \mu), \beta_i) + V(t,\mu^{u}_t(s,\mu))\right].
  \end{equation*}
  Since $u$ is an arbitrary admissible control, we finally obtain $V(s,\mu)\le W(s,\mu)$.
\end{proof}

\subsection{The HJB equation}
\label{sec:optimal-control-via}

We are now ready to derive the Hamilton--Jacobi--Bellman (HJB) equation for our control problem.
We first introduce the \emph{Hamiltonian} for the control problem outside the observation dates.
For $t \in [0,T]$ and $\mu \in \mathcal{P}(\R^d)$ and 
$p \in C_b(\R^d)$
 we set
\begin{equation}
  \label{eq:Hamiltonian}
  \mathcal{H}(t, \mu, p) \coloneqq \inf_{(\alpha,\beta) \in U} \left\{ \ip{\mathcal{G}^\ast(\alpha) \mu}{p} + \ell(t, \mu, \alpha) \right\}.
\end{equation}

\begin{theorem}[HJB equation]
  \label{thr:HJB-equation}
  Provided that the value function $V(t,\mu)$ is differentiable w.r.t. the time variable $t$ and is in $\mathcal{S}^{1,1}(\mathcal{P}(\R^d))$ w.r.t. $\mu$, it satisfies the HJB equation
  \begin{subequations}
    \label{eq:HJB-equation}
    \begin{align}
      \label{eq:HJB-between-observations}
      \frac{\partial V}{\partial t} (t,\mu) &+ \mathcal{H}\left(t, \mu, \frac{\delta V}{\delta \mu}(t, \mu, \cdot) \right) = 0, \quad t_{i} \le t < t_{i+1},\ i=1, \ldots, n,\\
      \label{eq:HJB-observations}
      V(t_i^-, \mu) &= \inf_{(\alpha,\beta) \in U} \left\{ h(t_i, \mu, \beta) + E\left[ V(t_i, K_\beta(\cdot; \mu, Y^{\beta,\mu}_i) \mu) \right] \right\},\quad i=1, \ldots, n,\\
      \label{eq:HJB-terminal}
      V(T,\mu) &= g(\mu).
    \end{align}
  \end{subequations}
  Fix a function $u^\ast:[0,T] \times \mathcal{P}(\R^d) \to U$ by
  \begin{equation}
    \label{eq:optimal-control-from-value}
    u^\ast(t, \mu) \in
    \begin{cases}
      \argmin_{(\alpha,\beta) \in U} \left\{ \ip{\mathcal{G}^\ast(\alpha) \mu}{\frac{\delta V}{\delta\mu}(t, \mu,\cdot)} + \ell(t, \mu, \alpha) \right\}, & t \notin \set{t_1, \ldots, t_n},\\
      \argmin_{(\alpha,\beta) \in U}\left\{ h(t_i, \mu, \beta) + E\left[ V(t_i, K_\beta(\cdot; \mu, Y^{\beta,\mu}_i) \mu) \right] \right\}, & t = t_i^-,\ i \in \set{1, \ldots, n},
    \end{cases}
  \end{equation}
  provided the set of minimizers in \eqref{eq:optimal-control-from-value} is not empty.
  Additionally, assume that $\mu^\ast_t$ denotes an optimal path.
  Then an \emph{optimal control} is defined by $u^\ast_t \coloneqq u^\ast(t, \mu^\ast_t)$.
\end{theorem}

Note that the minimization problem for the $\beta$-component of $u^\ast$ is not unique for $t$ not an observation time, which allows us to choose a piecewise-constant, predictable version, enforcing the admissibility conditions.  
We write $u^\ast = (\alpha^\ast,\beta^\ast)$ and, continuing our abuse of notation, $\beta^\ast_i = \beta^\ast_{t_i^-}$.

\begin{proof}[Proof of Theorem \ref{thr:HJB-equation}.]
First we note that at observation times $t_i$, \eqref{eq:DPP-jump} is exactly \eqref{eq:HJB-observations}. We only need to derive \eqref{eq:HJB-between-observations}.
Consider in \eqref{eq:DPP} a constant control $u:=a=(\bar{\alpha},\bar{\beta})$ for some arbitrary $a\in U$. For any $t, \theta>0$ such that,  for some $i=1,\ldots,n$, 
$[t , t+\theta]\subset [t_i, t_{i+1})$, we have
\begin{equation*}\begin{array}{ll}
V(t,\mu)\le \E\left[\int_t^{t+\theta}\ell(s,\mu_s^{a}(t,\mu),a)\,ds+V(t+\theta,\mu^{a}_{t+\theta}(t,\mu))\right]. 
\end{array}
\end{equation*}
Since $V$ satisfies the smoothness assumptions of the chain rule \eqref{Ito}, we obtain
\begin{multline*}
\E\left[\int_t^{t+\theta}\frac{\partial V}{\partial s}(s,\mu_s^{a}(t,\mu))+\langle \mathcal{G}^*(a)\mu_s^{a}(t,\mu),\frac{\delta V}{\delta\mu}(s,\mu_s^{a}(t,\mu),\cdot)\rangle + \ell(s,\mu_s^{a}(t,\mu),a)\, \dd s \right]\ge 0.
\end{multline*}
By the mean-value theorem, dividing by $\theta$ and then sending it to $0$, yields
\[
\frac{\partial V}{\partial t}(t,\mu)+\langle \mathcal{G}^*(a)\mu,\frac{\delta V}{\delta\mu}(t, \mu,\cdot)\rangle + \ell(t,\mu,a) \ge 0, \quad t_{i}\le t<t_{i+1},\quad i=1,\ldots,n.
\]
Since this inequality is true for all $a\in U$, we obtain
\begin{equation}\label{HJB-p}
\frac{\partial V}{\partial t}(t,\mu)+ \mathcal{H}\left(t, \mu, \frac{\delta V}{\delta \mu}(t, \mu, \cdot) \right)\ge 0, \quad t_{i}\le t<t_{i+1},\quad i=1,\ldots,n.
\end{equation}
Now, suppose that $u^*=(\alpha^*,\beta^*)$ is an optimal control. Then the value function satisfies
\begin{equation*}
V(t,\mu)=\E\left[\int_t^{t+\theta}\ell(s,\mu_s^{u^*}(t,\mu),u^*_s)\, \dd s+\sum_{i=\floor{t}+1}^{\floor{t+\theta}} h(t_i, \mu^{u^*}_{t_i^-}(s, \mu), \beta^*_i)+V(t+\theta,\mu^{u^*}_{t+\theta}(t,\mu))\right]. 
\end{equation*}
In particular, at the jump times, we have
\begin{equation}\label{DPP-jump-1}
V(t_i^-, \mu_{t_i^-}^{u^*})= h(t_i, \mu_{t_i^-}^{u^*}, \beta^*_i) + E\left[ V(t_i, K_{\beta^*_i}(\cdot; \mu_{t_i^-}^{u^*}, Y^{\beta_i,\mu_{t_i^-}^{u^*}}_i) \mu_{t_i^-}^{u^*}) \right],\,\, i=1, \ldots, n.
\end{equation}
Thus, by using a similar argument as above between jump times,  we obtain 
\[
\frac{\partial V}{\partial t}(t,\mu)+\left[\langle \mathcal{G}^*(u^*)\mu,\frac{\delta V}{\delta\mu}(s, \mu)(\cdot)\rangle + \ell(s,\mu,u^*(s)) \right]= 0, \quad t_i\le t<t_{i+1},\,\, i=1,\ldots,n,
\]
which, in view of \eqref{HJB-p}, suggests the value function should satisfy
\begin{equation*}
\frac{\partial V}{\partial t}(t,\mu)+\underset{a\in U}{\inf\,} \left[\langle \mathcal{G}^*(a)\mu,\frac{\delta V}{\delta\mu}(s, \mu)(\cdot)\rangle + \ell(s,\mu,a)  \right] =0, \quad t_i\le t<t_{i+1},\,\, i=1,\ldots,n.
\end{equation*}
\end{proof}

Next, we derive a verification theorem for our SOC problem.
\begin{theorem}[Verification theorem]
  \label{thr:HJB-verification}
  Let $W(t,\mu)$ be a solution to the HJB equation \eqref{eq:HJB-between-observations}, \eqref{eq:HJB-observations} and \eqref{eq:HJB-terminal}. 
  Then, 
  \begin{equation*}
    W(s,\mu) = \inf_{u \in \mathcal{U}[s,T]} J(s, \mu; u), \quad 0 \le s <T, \quad \mu \in \mathcal{P}(\R^d).
  \end{equation*}
 Furthermore, assume the function $\widehat{u}:[0,T] \times \mathcal{P}(\R^d) \to U$ satisfies
  \begin{equation*}
    \widehat{u}(t, \mu) \in
    \begin{cases}
      \argmin_{(\alpha,\beta) \in U} \left\{ \ip{\mathcal{G}^\ast(\alpha) \mu}{p} + \ell(t, \mu, \alpha) \right\}, & t \notin \set{t_1, \ldots, t_n},\\
      \argmin_{(\alpha,\beta) \in U}\left\{ h(t_i, \mu, \beta) + E\left[ W(t_i, K_\beta(\cdot; \mu, Y^{\beta,\mu}_i) \mu) \right] \right\}, & t = t_i^-,\ i \in \set{1, \ldots, n}.
    \end{cases}
  \end{equation*}
  Then, the feedback control $u^*$ given by $u^*_t:=\widehat{u}(t,\mu)$ is optimal.
 \end{theorem}
\begin{proof}
In view of \eqref{eq:HJB-between-observations}, \eqref{eq:HJB-observations} and \eqref{eq:HJB-terminal} it follows that, for each $a\in U$ and $(t,\mu)\in[0,T]\times \mathcal{P}(\R^d)$,
\begin{gather*}
  \frac{\partial W}{\partial t}(t,\mu)+\langle \mathcal{G}^*(a)\mu,\frac{\delta W}{\delta\mu}(s, \mu,\cdot)\rangle + \ell(s,\mu,a)\ge 0, \quad t_i\le t<t_{i+1},\,\, i=1,\ldots,n,\\
  W(t_i^-, \mu)= \inf_{(\alpha,\beta) \in U} \left\{ h(t_i, \mu, \beta) + E\left[ W(t_i, K_\beta(\cdot; \mu, Y^{\beta,\mu}_i) \mu) \right] \right\},\quad i=1, \ldots, n,
\end{gather*}
and $W(T,\mu)=g(\mu)$. Let us now replace $t,a, \mu$ by $s$ and $u_r=(\alpha_r,\beta_r), \mu_r^{u},\,s\le r\le T$. Upon conditioning on $\mu_s^u=\mu$, we get 
\[
W(s,\mu)\le \E\left[ \int_s^T \ell(t, \mu_t^u(s, \mu), u_t) \dd t + \sum_{i=\floor{s}+1}^n h(t_i, \mu^u_{t_i^-}(s, \mu), \beta_{t_i}) + g(\mu^u_T)\,\Big|\,\mu^u_s=\mu \right],
\]
which entails that 
\begin{equation*}
    W(s,\mu) = \inf_{u \in \mathcal{U}[s,T]} J(s, \mu; u), \quad 0 \le s <T, \quad \mu \in \mathcal{P}(\R^d).
  \end{equation*}
  
 Now, for $u^*=\widehat{u}$, the last inequality becomes an equality. Therefore, $W(t,\mu)=J(s, \mu; u^*)$ i.e., $u^*
$ is optimal.
\end{proof} 
 
\begin{remark}[Solving the HJB equation]
  \label{rem:solving-HJB}
  The HJB equation~\eqref{eq:HJB-equation} is a piecewise backward PDE.
  Given the solution $V(t_i^-, \cdot)$ at time $t_i$, $i=1, \ldots, n+1$ (interpreted as $V(t_{n+1}^-, \mu) = g(\mu)$ at $t_{n+1} = T$), we obtain the solution on $[t_{i-1}, t_i)$ by solving the PDE~\eqref{eq:HJB-between-observations} backward in time.
  Then, we define $V(t_{i-1}^-, \mu)$ by~\eqref{eq:HJB-observations}, and continue as before.
\end{remark}

\begin{remark}[Parameterized HJB equation]
  \label{rem:parameterized-HJB}
  Suppose that we are given a subset $\mathcal{M} \subset \mathcal{P}(\R^d)$, which is invariant under $\mathcal{G}^\ast(\alpha)$ as well as under the Bayesian update for any optimal control $u$, and that we start in $\mu_0 \in \mathcal{M}$.
  Let us further assume that probability measures in $\mathcal{M}$ are uniquely characterized by the expectations of functions $\Set{\varphi_j | j \in \mathcal{J}}$ indexed by a (finite or infinite) index set $\mathcal{J}$.

  In this case, we can recast the value function as a function $$V(t, \mu) = U\left(t, (\ip{\mu}{\varphi_j})_{j \in \mathcal{J}} \right),$$ 
  with $\mu \in \mathcal{M}$, for some function $U(t,z)$, $z \in \R^{\abs{\mathcal{J}}}$.
  Note that we thus have
  \begin{equation}
    \label{eq:change-of-variables-moments}
    \frac{\delta V}{\delta \mu}(t,\mu,\cdot) = \sum_{j\in\mathcal{J}} \frac{\partial U}{\partial z_j} \left(t, (\ip{\mu}{\varphi_j})_{j \in \mathcal{J}} \right) \varphi_j(\cdot).
  \end{equation}
  Hence, the HJB equation~\eqref{eq:HJB-equation} can be recasted into an HJB equation for $U(t,z)$.
\end{remark}

\begin{example}
  \label{ex:parameterized-HJB}[Parameterized HJB equation]
  For a general example of Remark~\ref{rem:parameterized-HJB}, consider a compact set $K \subset \R^d$ which is invariant under the controlled dynamics described by $\mathcal{G}^\alpha$. Let $\mathcal{M} \subset \mathcal{P}(\R^d)$ denote the set of probability measures supported by $K$, and further assume that the Bayesian update $K_\beta$ leaves $\mathcal{M}$ invariant.
  Hence, if we start in $\mu_0 \in \mathcal{M}$, we stay in $\mathcal{M}$ for any choice of control.

  Recall that any $\nu \in \mathcal{M}$ is characterized by its moments $\int \varphi_j(x) \nu(\dd x)$, $\varphi_j(x) \coloneqq x_1^{j_1} \cdots x_d^{j_d}$, $j \in \mathcal{J} = \mathbb{N}_0^d$. Hence, we are formally in the situation of Remark~\ref{rem:parameterized-HJB} and we can introduce moment coordinates in this space of measures, reducing the problem from a SOC formulation with probability measures valued states to the simpler case of a countable sequence valued state, that one may have to truncate for constructive approximation.
\end{example}



\section{Optimal control of an Ornstein -- Uhlenbeck process}
\label{sec:optimal-control-an}

Consider the linear--quadratic example based on a controlled Ornstein--Uhlenbeck process, with Gaussian noisy observations at times $t_i$, $i=1, \ldots, n$.
More precisely, we consider the controlled generator
\begin{equation}
  \label{eq:OU-generator}
  \mathcal{G}(\alpha) f(x) = (-\theta x + \alpha) \partial_x f(x) + \frac{1}{2} b^2 \partial_{xx} f(x),
\end{equation}
based on a control $\alpha \in \mathbb{R}$, implying that
\begin{equation*}
  \mathcal{G}(\alpha)^\ast p(x) = \theta p(x) + (\theta x-\alpha) \partial_x p(x) + \frac{1}{2} \partial_{xx} p(x).
\end{equation*}

We denote by $\mu^u_t$ the random evolution of conditional law of the unobserved controlled process $X^u$ with a control process $u_t = (\alpha_t,\beta_t)$.
 
We try to minimize quadratic costs, i.e.,
\begin{equation}\label{eq:quadratic-cost}
  \ell(t, \mu, \alpha) \coloneqq \int_{\mathbb{R}} x^2 \mu(\mathrm{d}x) + C \alpha^2, \quad g(\mu) \coloneqq \int_{\mathbb{R}} x^2 \mu(\mathrm{d}x),
\end{equation}
with all other cost terms vanishing.

At time $t_i$ we observe $Y^\beta_i \coloneqq \widehat{X}^u_i + \beta_i Z_i$ with $Z_i \sim \mathcal{N}(0,1)$ -- independent of each other and all other sources of randomness -- and $\widehat{X}^u_i \sim \mu^u_{t_i^-}$, see Example~\ref{ex:gaussian-noisy-observation}.
Hence, we update $\mu^u$ according to the specification
\begin{equation}\label{eq:bayes-update}
  \mu^u_{t_i}(\mathrm{d}y) = K_{\beta_i}(y; \mu^u_{t_i^-}, Y^\beta_i) \mu^u_{t_i^-}(\mathrm{d} y).
\end{equation}
For now, we assume that the noise level $\beta$ is fixed, and not a control parameter.
To make this clear notation-wise, we write $\varepsilon = \beta$.

Observe that the above updating rule preserves Gaussian random variables. Indeed, we have
\begin{lemma}\label{lem:gauss}
  Suppose that $\mu = \mathcal{N}(m, \sigma^2)$.
  Then $\nu$ defined by $\nu(\mathrm{d}x) \coloneqq K_\varepsilon(x; \mu, y) \mu(\mathrm{d}x)$ is equal to $\mathcal{N}\left( m + \frac{\sigma^2}{\sigma^2+\varepsilon^2} (y - m), \, \frac{\sigma^2\varepsilon^2}{\sigma^2+\varepsilon^2} \right)$.
\end{lemma}
\begin{proof}
  Use the formula for the conditional distribution of $(X|Y=y)$ for jointly Gaussian $(X,Y)$ and apply it with $X = \widehat{X}^u_i$ and $Y = Y^\varepsilon_i$.
\end{proof}

It is well--known, see \cite{guo_pham_wei23}, that the value function to the stochastic optimal control problem -- with no observation of the trajectories -- is given as
\begin{equation}
  \label{eq:OU-value-no-observation}
  V(t, \mu) = \zeta(t) \int_{\R} x^2 \mu(\dd x) + \eta(t) \left( \int_{\R} x \mu(\dd x) \right)^2 + \xi(t)
\end{equation}
in terms of the \emph{Riccati equations}
\begin{subequations}
  \label{eq:Riccati}
  \begin{gather}
    \label{eq:Riccati-a}
    \dot{\zeta}(t) - 2\theta \zeta(t) + 1 = 0,\\
    \label{eq:Riccati-b}
    \dot{\eta}(t) - \frac{C-1}{C^2}\left( \zeta(t) + \eta(t) \right)^2 - 2 \theta \eta(t) = 0,\\
    \label{eq:Riccati-c}
    \dot{\xi}(t) + b^2\zeta(t) = 0,
  \end{gather}
\end{subequations}
with terminal conditions $\zeta(T) = 1$, $\eta(T) = \xi(T) = 0$.

Now let us consider the case with noisy observations (in the sense of Example~\ref{ex:gaussian-noisy-observation}) at times $t_1, \ldots, t_n$.
It is easy to see that optimal trajectories of $\mu_t^u$ remain in the class of (random) normal distributions\footnote{I.e., normal distributions with possibly random mean or variance.}, provided that the initial distribution $\mu_0$ is within this class.
Nonetheless, the Bayesian update step~\eqref{eq:HJB-observations} destroys the form~\eqref{eq:OU-value-no-observation}.
\begin{lemma}
  \label{lem:OU-value-violation}
  Assume that $\mu = \mathcal{N}(m, \sigma^2)$ for some $m \in \R$, $\sigma^2>0$, and $V(t,\mu) = \alpha \int_{\R} x^2 \mu(\dd x) + \varepsilon \left( \int_{\R} x \mu(\dd x) \right)^2 + \gamma = \alpha (m^2 + \sigma^2) + \varepsilon m^2 + \gamma$. For $Y \sim \mathcal{N}(m, \sigma^2+\varepsilon^2)$, we have
  \begin{equation*}
    E\left[ V(t, K_\varepsilon(\cdot; \mu, Y) \mu) \right] = \alpha (m^2 + \sigma^2) + \varepsilon \left( m^2 + \frac{\sigma^4}{\sigma^2+\varepsilon^2} \right) + \gamma,
  \end{equation*}
  which is not of the form~\eqref{eq:OU-value-no-observation} for a Gaussian distribution unless $\varepsilon = 0$.
\end{lemma}
\begin{proof}
  By Lemma~\ref{lem:gauss}, we know that
  \begin{equation*}
    \nu \coloneqq K_\varepsilon(\cdot; \mu, Y) = \mathcal{N}\left( m + \frac{\sigma^2}{\sigma^2+\varepsilon^2} (Y - m), \, \frac{\sigma^2\varepsilon^2}{\sigma^2+\varepsilon^2} \right).
  \end{equation*}
  Hence,
  \begin{multline*}
    V(t, \nu) = \alpha \left( \frac{\sigma^2\varepsilon^2}{\sigma^2+\varepsilon^2} + \left[ m + \frac{\sigma^2}{\sigma^2+\varepsilon^2} (Y - m) \right]^2 \right) \\+ \varepsilon \left( m + \frac{\sigma^2}{\sigma^2+\varepsilon^2} (Y - m) \right)^2 + \gamma,
  \end{multline*}
  and we conclude by taking expectations.
\end{proof}

Following the parametric approach suggested in Remark~\ref{rem:parameterized-HJB}, we note that normal distributions are uniquely determined by their first and second moments, or, more appropriately, by mean and variance.
In order to avoid confusion with the measure $\mu$ and when taking derivatives, we use the variables $m$ for the mean and $z$ for the \emph{variance} of our measures.
\begin{lemma}[HJB, controlled O-U case]
  \label{lem:parameterized-HJB-OU}
  Let $\widehat{U} = \widehat{U}(t, m, z)$ denote the value function $V(t,\mu)$ in terms of mean $m$ and variance $z$ of the Gaussian measure $\mu$. The function $\widehat{U}$ satisfies the HJB equation
  \begin{subequations}
    \label{eq:parameterized-HJB-OU}
    \begin{align}
      \label{eq:parameterized-HJB-OU-1}
      \frac{\pa \widehat{U}}{\pa t}(t, m, z) + z+m^2 - \theta m &\frac{\pa \widehat{U}}{\pa m} (t, m, z) + (b^2 - 2\theta z) \frac{\pa \widehat{U}}{\pa z}(t, m, z)\\
      \nonumber  &- \frac{1}{4C} \left[ \frac{\pa \widehat{U}}{\pa m}(t, m, z) \right]^2 = 0,\quad t_i \le t < t_{i+1},~i=1,...,n\\
      \label{eq:parameterized-HJB-OU-2}
      \widehat{U}(t_i^-, m, z) = \int_{\R} \widehat{U}\biggl( t_i, m+ &\frac{z}{\sqrt{z + \varepsilon^2}}z, \frac{z \varepsilon^2}{z + \varepsilon^2} \biggr) \phi(z) \dd z,\\
      \label{eq:parameterized-HJB-OU-3}
      \widehat{U}(T,m,z) = m^2 + z,
    \end{align}
  \end{subequations}
  where $\phi$ denotes the density of a standard normal.
\end{lemma}
Since the function $\widehat{U}$ is defined on $\mathbb{R}^+ \times\mathbb{R} \times \mathbb{R}^+$,  we naturally wonder if we need to impose a value of $\widehat{U}$ on the boundary $z=0$. Recall that equation (\ref{eq:parameterized-HJB-OU-1}) - (\ref{eq:parameterized-HJB-OU-3}) is a terminal value problem, which is solved backwards in time. Using the method of characteristics for first order nonlinear PDEs, it can be shown that for any point in the domain $\mathbb{R} \times\mathbb{R}^+$ at terminal time the information propagates backwards via the characteristic lines $z(\tau)$ and crosses the boundary $z=0$ from inside the domain to the outside, and no information is entering the domain from the $z<0$ (see Section \ref{sec:num-approach}). Therefore, we do not need to impose a boundary condition at $z=0$.
\begin{proof}[Proof of Lemma~\ref{lem:parameterized-HJB-OU}.]
Following Remark \ref{rem:parameterized-HJB},  since we are working with Gaussian distributions in one dimension, it is enough to choose $\varphi_1(x) \coloneqq x$, $\varphi_2(x) \coloneqq x^2$, as well as $U(t,m,s) \coloneqq V(t, \mu)$, for $m \coloneqq \int_{\R} \varphi_1(x) \mu(\dd x), \, s \coloneqq \int_{\R} \varphi_2(x) \mu(\dd x)$. By~\eqref{eq:change-of-variables-moments}, we have
  \begin{equation*}
    \frac{\delta V}{\delta \mu}(t, \mu) = \frac{\pa U}{\pa m}(t,m,s) \varphi_1 + \frac{\pa U}{\pa s}(t,m,s) \varphi_2.
  \end{equation*}
  Note that
  \begin{align*}
    \mathcal{H}\left(t, \mu, \frac{\delta V}{\delta \mu}(t,\mu)\right) &= \inf_{\alpha \in \R} \left\{ \ip{\mu}{\mathcal{G}(\alpha)\frac{\delta V}{\delta \mu}(t,\mu)} + \ip{\mu}{\varphi_2} + C\alpha^2 \right\},\\
                         &= \inf_{\alpha \in \R}\biggl\{ \int_{\R} \left[ \frac{\pa U}{\pa m}(t,m,s) (-\theta x + \alpha) + 2 \frac{\pa U}{\pa s}(t,m,s) x(-\theta x + \alpha) + b^2 \frac{\pa U}{\pa s}(t,m,s)  \right] \mu(\dd x) +\\
                         &\quad\quad +\int_{\R} x^2 \mu(\dd x) + C \alpha^2 \biggr\}\\
                         &= \inf_{\alpha \in \R} \biggl\{ C \alpha^2 + \left[ \int_{\R} \left( \frac{\pa U}{\pa m}(t,m,s) + 2 x \frac{\pa U}{\pa s}(t,m,s) \right) \mu(\dd x) \right] \alpha +\\
                         &\quad\quad + \int_{\R} \left[x^2 - \theta x \frac{\pa U}{\pa m}(t,m,s) + (b^2- 2\theta x^2) \frac{\pa U}{\pa s}(t,m,s)  \right] \mu(\dd x) \biggr\}\\
    &= \inf_{\alpha \in \R} \biggl\{ C \alpha^2 + \left( \frac{\pa U}{\pa m}(t,m,s) + 2 m \frac{\pa U}{\pa s}(t,m,s) \right) \alpha +\\
                         &\quad\quad + \left[ s - \theta m \frac{\pa U}{\pa m}(t,m,s)  + (b^2- 2\theta s) \frac{\pa U}{\pa s}(t,m,s) \right] \biggr\}\\
                         &=  s - \theta m \frac{\pa U}{\pa m}(t,m,s) + (b^2- 2\theta s) \frac{\pa U}{\pa s}(t,m,s) \\
    &\quad\quad- \f{1}{4C} \left( \frac{\pa U}{\pa m}(t,m,s) + 2 m \frac{\pa U}{\pa s}(t,m,s) \right)^2.
  \end{align*}
  Changing variables $s \to z = s - m^2$, we get
  \begin{equation*}
    \frac{\pa U}{\pa m}(t,m,s) = \frac{\pa \widehat{U}}{\pa m}(t,m,z) - 2 m \frac{\pa \widehat{U}}{\pa z}(t,m,z),
  \end{equation*}
  whereas derivatives w.r.t.~$s$ can just be replaced by derivatives w.r.t.~$z$. This change of variables implies that
  \begin{multline*}
    \mathcal{H}\left(t, \mu, \frac{\delta V}{\delta \mu}(t,\mu)\right) = m^2 + z - \theta m \frac{\pa \widehat{U}}{\pa m}(t,m,z) + (b^2- 2\theta z) \frac{\pa \widehat{U}}{\pa s}(t,m,z) \\
    - \f{1}{4C} \left( \frac{\pa \widehat{U}}{\pa m}(t,m,z) \right)^2,
  \end{multline*}
  from which we immediately get~\eqref{eq:parameterized-HJB-OU-1}.

  By Lemma~\ref{lem:gauss}, for $\mu = \mathcal{N}(m, z)$,
  \begin{equation*}
    V(t_i, K_\varepsilon( \cdot; \mu, y)) = U\left( t_i, m + \frac{z}{z + \varepsilon^2} (y-m), \frac{z \varepsilon^2}{z + \varepsilon^2} \right).
  \end{equation*}
  By assumption, $Y_i^\alpha = \widehat{X}^\alpha_i + \varepsilon Z_i \sim \mathcal{N}(m, z + \varepsilon^2)$, for $Z_i \sim \mathcal{N}(0,1)$, and provided that $\mu = \mathcal{N}(m, z)$. As $h \equiv 0$ and $\varepsilon$ is considered fix (i.e., not controlled), \eqref{eq:HJB-observations} implies~\eqref{eq:parameterized-HJB-OU-2}.
\end{proof}



\section{Numerical approach}
\label{sec:num-approach}


We illustrate our method on three numerical examples: a classic linear–quadratic control problem, control of measurement noise, and control under a high-cost region constraint.

\subsection{Example 1: Linear-quadratic control under uncertain observation}\label{subsec:NA-LQ_control}
Recall from Lemma \ref{lem:parameterized-HJB-OU} the problem under consideration: find $\widehat{U} = \widehat{U}(t, m, z)$ which satisfies the HJB equation
\begin{subequations}
  \label{eq:parameterized-HJB-OU_num}
  \begin{align}
    \label{eq:parameterized-HJB-OU-1_num}
    \frac{\pa \widehat{U}}{\pa t}(t, m, z) + z + m^2 - \theta m &\frac{\pa \widehat{U}}{\pa m} (t, m, z) + (b^2 - 2\theta z) \frac{\pa \widehat{U}}{\pa z}(t, m, z)\\
    \nonumber  &- \frac{1}{4C} \left[ \frac{\pa \widehat{U}}{\pa m}(t, m, z) \right]^2 = 0,\quad t_n \le t < t_{n+1},\quad n=0,...,N\\
    \label{eq:parameterized-HJB-OU-2_num}
    \widehat{U}(t_n^-, m, z) = \int_{\R} \widehat{U}\biggl( t_n, m+ &\frac{z}{\sqrt{z + \varepsilon^2}}w, \frac{z \varepsilon^2}{z + \varepsilon^2} \biggr) \phi(w) \dd w,\\ 
    \label{eq:parameterized-HJB-OU-3_num}
    \widehat{U}(T,m,z) = m^2 + &z,
  \end{align}
\end{subequations}
where $\phi$ denotes the density of a standard normal.

In this section we present the numerical approach to solving this problem. Observe that this equation is defined on an unbounded domain. We show how we truncate the domain to be able to treat this problem numerically. We introduce the numerical scheme used to solve the HJB equation (\ref{eq:parameterized-HJB-OU-1_num}) on time interval $t \in [t_n,t_{n+1})$, and we show how to transform the result at time $t_i$ according to (\ref{eq:parameterized-HJB-OU-2_num}). Finally, we present the results obtained testing this approach on a synthetic problem. 

\subsubsection{Choice of the computational domain and boundary conditions}\label{subsubsec:NA-LQ_control_domain}

As mentioned, the original problem is defined on an unbounded domain with $m \in (-\infty,\infty)$ and $z \in [0,\infty)$. However, to solve the equation numerically, one must choose the truncated domain. Furthermore, one should either prescribe the boundary conditions at the boundaries of the chosen domain, or propose the numerical scheme, which uses only the interior points to compute current values and is shown to be stable. The upwind difference scheme has the desired property as long as information propagates from the initial conditions by lines flowing outside the domain. To check if this is the case, one can compute the characteristic lines of the equation following section 3.2 of \cite{evans@2010}. In our case, the characteristic lines of (\ref{eq:parameterized-HJB-OU-1_num}) are flowing from outside the domain, which is unwanted. However, if we restate the problem in a reversed time flow as follows
  \begin{subequations}
    \label{eq:parameterized-HJB-OU_num_rev_time}
    \begin{align}
      \label{eq:parameterized-HJB-OU-1_num_rev_time}
      \frac{\partial \widehat{U}}{\partial t}(t, m, z) + H\Big(m, z, \frac{\partial \widehat{U}}{\partial m},\frac{\partial \widehat{U}}{\pa z}\Big) = 0, \quad t_i \le t < t_{i+1},\\
      H\Big(m, z, \frac{\partial \widehat{U}}{\partial m},\frac{\partial \widehat{U}}{\pa z}\Big) := H_1\Big(m, z, \frac{\pa \widehat{U}}{\pa m}\Big) + H_2\Big(m, z, \frac{\partial \widehat{U}}{\pa z}\Big) \\
      H_1\Big(m, z,\frac{\pa \widehat{U}}{\pa m}\Big) := - m^2 + \theta m \frac{\pa \widehat{U}}{\pa m} (t, m, z) + \frac{1}{4C} \left[ \frac{\pa \widehat{U}}{\pa m}(t, m, z) \right]^2,\\
      H_2\Big(m, z,\frac{\pa \widehat{U}}{\pa z}\Big) := 
      - z  - (b^2 - 2\theta z) \frac{\pa \widehat{U}}{\pa z}(t, m, z), 
    \end{align}
  \end{subequations}
we can choose a domain, such that the characteristic lines cross the boundary from inside the domain. See, for example, Figure \ref{Char}, which shows the characteristic lines for $m \in [-1,1]$ and $z \in [0,1]$ for the synthetic problem with parameters given in subsection \ref{subsection:NA-LQ_control_simulations}. The lines for $m$ are symmetric around $m(0)=0$; for $z$ they are symmetric around $z = b^2/(2\theta)$. Therefore, as long as the computational domain contains $m=0$ and $z=0.5$, the numerical solution is stable. The details on the derivation of the characteristics are presented in Appendix \ref{sec:Details on characteristic}. 
Additionally, Saldi, Basar, and Raginsky \cite{SaldiBasarRaginsky2018} explored the asymptotic optimality of finite model approximations for partially observed Markov decision processes (POMDPs). Although relevant for a time discrete setting, their approximation results for POMDPs may provide theoretical support for our domain truncation numerical approach as well, connecting our work to established results in the approximation of partially observed control problems.

\begin{figure}[h!]
	\centering
	\includegraphics[width=0.45\textwidth]{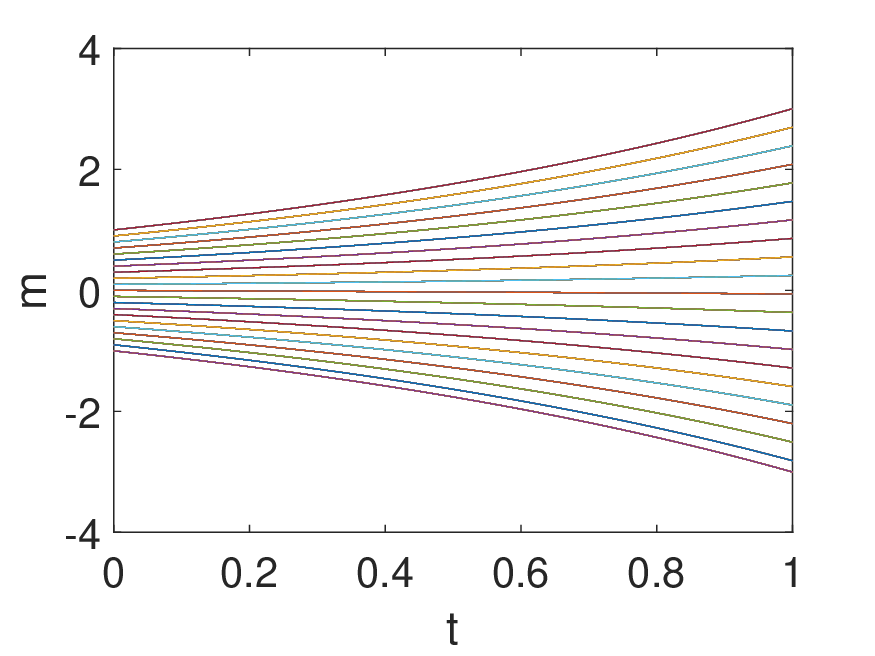}
    \includegraphics[width=0.45\textwidth]{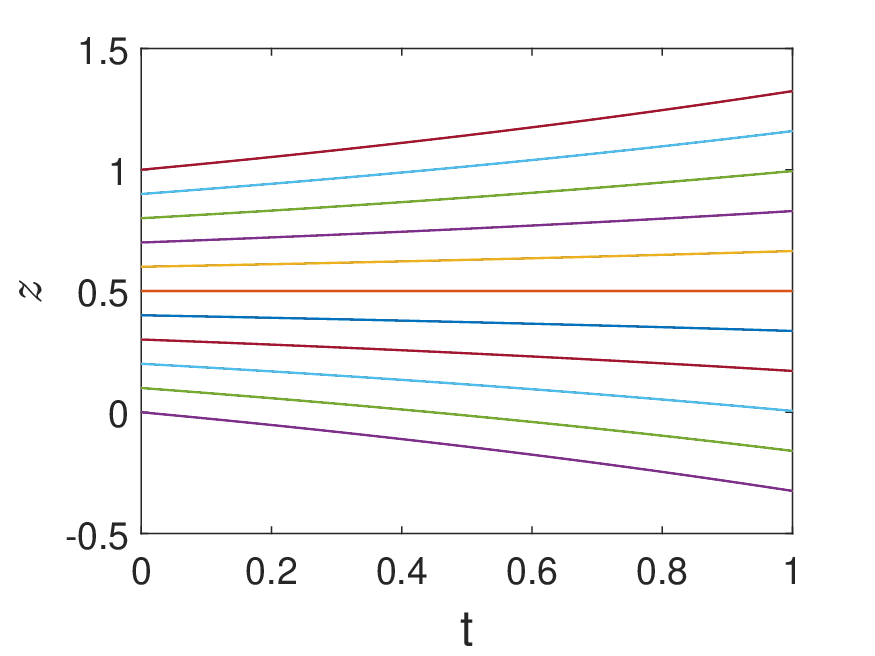}
	\caption{\label{Char} Characteristic lines shown for the domain of interest $m_i \in [-1,1]$, $z_j \in [0,1]$. For $m$, the lines are symmetric w.r.t. zero; for $z$ they are symmetric around $z = b^2/(2 \theta)=0.5$.}
\end{figure}

\subsubsection{Numerical scheme for HJB equation}\label{subsubsec:NA-LQ_control_num_scheme}

In order to construct a convergent scheme we will follow the guidelines outlined in section 3.1 of \cite{falcone@2016} and \cite{crandall@1984}. We introduce the grid
\begin{gather}
\mathcal{G} := \{(m_i,z_j,t_n): m_i = i\Delta m,~ z_j = j \Delta z,~ t_n = n\Delta t,\\
~~~\forall i \in [0,...,I], ~j \in [0,...,J], ~n \in [0,...,N]\}
\end{gather}
with uniform step sizes $\Delta m = \frac{m_I-m_0}{I}$, $\Delta z = \frac{z_J-z_0}{J}$, $\Delta t = \frac{t_N-t_0}{N}$. The boundary points $m_0$, $m_I$, $z_0$ and $z_J$ are chosen to ensure the stability of the scheme as described in the previous subsection. 

Let $U^{n}_{i,j} \approx \widehat{U}(t_n, m_i, z_j)$ denote the approximation of the solution to (\ref{eq:parameterized-HJB-OU-1_num_rev_time}). To approximate the spatial derivatives at time $t_n$ we will use left and right finite differences denoted by $D^{\pm}(U^n_{i,j})$ and defined as follows:
\begin{gather}
D^-_m(U^n_{i,j}) = \frac{U^n_{i,j}-U^n_{i-1,j}}{\Delta m},~~~~~D^+_m(U^n_{i,j}) = \frac{U^n_{i+1,j}-U^n_{i,j}}{\Delta m}, \nonumber \\
D^-_{z}(U^n_{i,j}) = \frac{U^n_{i,j}-U^n_{i,j-1}}{\Delta z},~~~~~D^+_{z}(U^n_{i,j}) = \frac{U^n_{i,j+1}-U^n_{i,j}}{\Delta z}.
\end{gather}

We aim to construct a scheme in a differenced form
\begin{align}
U^{n+1}_{i,j} = G(U^{n}_{i-1,j},U^{n}_{i,j},&U^{n}_{i+1,j},U^{n}_{i,j-1},U^{n}_{i,j+1}), \label{scheme}\\
G(U^{n}_{i-1,j},U^{n}_{i,j},U^{n}_{i+1,j},&U^{n}_{i,j-1},U^{n}_{i,j+1}) := \\ &U^n_{i,j} 
- \Delta t \mathcal{H}(m,D^-_m(U^n_{i,j}),D^+_m(U^n_{i,j});z,D^-_{z}(U^n_{i,j}),D^+_{z}(U^n_{i,j})),
\end{align}
where the function $\mathcal{H}$ is called the numerical Hamiltonian. Theorem 1 in \cite{crandall@1984} guarantees the convergence, provided that the scheme $G$ is monotone and consistent. 

\textit{Monotonicity}: The scheme is said to be monotone, if the function $G$ is monotone in each of its arguments as long as $|D^{\pm}_m(U^n_{i,j})|, |D^{\pm}_{z}(U^n_{i,j})| \leq L$. To achieve that, the information about the speed of propagation of the solution $\partial H/\partial m$, $\partial H / \partial z$ is used. For instance, 
$\frac{\partial H_1}{\partial (\partial \widehat{U}/\partial m)} \geq 0,~\text{if}~\frac{\partial \widehat{U}}{\partial m} \geq -2C \theta m$, and $\frac{\partial H_2}{\partial (\partial \widehat{U}/\partial z)} \geq 0,~\text{if}~b^2 - 2 \theta z \leq 0$, thus the numerical Hamiltonians can be defined as follows:
\begin{align} \label{Ham_num_m}
&\mathcal{H}_1(m,z,D^-_m,D^+_m) = \\
&\begin{cases}
H_1(m,z,D^-_m),& \text{if}~D^-_m,D^+_m \geq -2 C \theta m, \\ 
H_1(m,z,D^+_m)+H_1(m,z,D^-_m)-H_1(m,z,-2C \theta m),& \text{if}~D^-_m \geq -2 C \theta m, ~D^+_m \leq -2 C \theta m\\
H_1(m,z,-2C \theta m),& \text{if}~D^-_m \leq -2C \theta m, ~D^+_m \geq -2 C\theta m\\
H_1(m,z,D^+_m),& \text{if}~D^-_m,D^+_m \leq -2 C\theta m
\end{cases} \nonumber
\end{align}
\begin{gather}\label{Ham_num_s}
\mathcal{H}_2(m,z,D^-_{z},D^+_{z}) = \begin{cases}
H_2(m,z,D^-_{z}),& \text{if}~b^2 - 2 \theta z < 0, \\
H_2(m,z,D^+_{z}),& \text{if}~b^2 - 2 \theta z \geq 0
\end{cases},\\
\mathcal{H}(m,D^-_m(U^n_{i,j}),D^+_m(U^n_{i,j});z,D^-_{z}(U^n_{i,j}),D^+_{z}(U^n_{i,j})) = \mathcal{H}_1(m,z,D^-_m,D^+_m) + \mathcal{H}_2(m,z,D^-_{z},D^+_{z}) \label{Ham_num}
\end{gather}

With Hamiltonian defined in (\ref{Ham_num}) the scheme $G$ is monotone if
\begin{gather}\label{monotone_condition}
1 - 2\frac{\Delta t}{\Delta m} \Big|\frac{\pa H_1}{\pa (\pa \hat{U}/\pa m)}(m,z,\alpha)\Big| - \frac{\Delta t}{\Delta z}\Big|\frac{\pa H_2}{\pa (\pa \hat{U}/\pa z)}(m,z,\gamma)\Big| \geq 0
\end{gather}
as long as $|\alpha|$, $|\gamma| \leq L$, which is an a priori bound on the derivatives of the value function. The derivation of (\ref{monotone_condition}) and the discussion of the bounds $L$ are given in Appendix \ref{sec:Derivation of the monotonicity condition}.

\textit{Consistency}. It is easy to check that the consistency of the proposed scheme is trivially satisfied
\begin{gather}
\mathcal{H}(m,\alpha, \alpha; z,\beta,\beta) = H_1(m,\alpha) + H_2(z,\beta)= H(m,z,\alpha,\beta).
\end{gather}

\subsubsection{Computation of $\hat{U}$ at time $t_n$.}
Once we obtained the value of $\widehat{U}(t_n,m,z)$ at time $t_n$, the observation is made and this information is used to correct the value function $\widehat{U}(t_n^-,m,z)$ according to (\ref{eq:parameterized-HJB-OU-2_num}):
\begin{gather}
      \widehat{U}(t_n^-, m, z) = \int_{\mathbb{R}} \widehat{U}\Big( t_n, m+ \frac{z}{\sqrt{z + \varepsilon^2}}w, \frac{z \varepsilon^2}{z + \varepsilon^2} \Big) \phi(w) dw,
\end{gather}
where $\phi(w) = \frac{1}{\sqrt{2\pi}}e^{-\frac{w^2}{2}}$. With a change of variable $-\frac{z}{\sqrt{z + \varepsilon^2}}w=\tau$ we can rewrite the integral as
\begin{gather} 
      \widehat{U}(t_n^-, m, z) = \int_{\mathbb{R}} \widehat{U}\Big( t_n, m - \tau, \frac{z \varepsilon^2}{z + \varepsilon^2} \Big) \hat{\phi}(\tau) d\tau, \label{eq:int_conv}
\end{gather}
where
\begin{gather} \label{int_trans}
      \hat{\phi}(\tau) = \frac{\sqrt{z  + \varepsilon^2}}{z}\frac{1}{\sqrt{2\pi}} e^{-\frac{\tau^2 (z + \varepsilon^2)}{2 z^2}}
\end{gather}
The integral (\ref{eq:int_conv}) is a convolution for each given $z$ and can be computed using Fourier transform.
Note that the points $(\cdot,\cdot,\frac{z\varepsilon^2}{z+\varepsilon^2})$ are off grid $\mathcal{G}$, therefore the value function at these points can be interpolated from the nearby points.

\subsubsection{Control under perfect observation}\label{subsubsec:NA-LQ_control_perfect_info}

We now assume that we are in the classical case, namely that the process $X(t)$ can be fully observed at all times with no noise.
Consider again the linear-quadratic example based on a controlled Ornstein-Uhlenbeck process described in Section \ref{sec:optimal-control-an}
\begin{gather}
	dX(t) = (-\theta X(t) + \alpha(t))dt + b dW(t),~t>0 \\
	X(0) = X_0
\end{gather}
with the cost function to minimize given by
\begin{gather}
	\mathbb{E} \Big[\int_0^T X(t)^2dt + C \int_0^T \alpha(t)^2dt\Big] \rightarrow \min
\end{gather}

Define the value function
\begin{gather} \label{value_func}
	u(t,x) := \underset{\alpha}{\min}~\mathbb{E} \Big[\int_t^T X(t)^2dt + C \int_t^T \alpha(t)^2dt \Big| X(t) =x\Big]
\end{gather}
Then $u$ solves the HJB equation
\begin{gather}\label{eq: sim_HJB}
	\partial_tu + H(x,t,\partial_xu,\partial_x^2u) = 0,~t<T \\
	u(T,\cdot) = 0,
\end{gather}
with 
\begin{align}\label{eq: sim_Ham}
	H(x,t,\partial_xu,\partial_x^2u) &= \underset{\alpha}{\min} \{(-\theta x + \alpha)\partial_xu + b^2 \partial^2_xu/2 + x^2 + C \alpha^2\} \\
	&= (-\theta x -\partial_xu/(2C))\partial_xu + b^2 \partial^2_xu/2 + x^2 + \frac{(\partial_xu)^2}{4C}.
\end{align}

One can check that $u(t,x) = f(t)x^2 + g(t)$ solves the HJB equation (\ref{eq: sim_HJB})-(\ref{eq: sim_Ham}) if $f(t)$ solves the following Ricatti ODE
\begin{gather} \label{eq: sim_f}
	1 - 2\theta f - f^2/C + f' = 0,\\
	f(T) = 0,
\end{gather}
and $g(t) = \sigma^2 \int_t^T f(s)ds$.
Equation (\ref{eq: sim_f}) can be integrated numerically. Thus, computing $u(t,x)$ provides a lower bound on the value function for the case with noisy measurements.

\subsubsection{Simulations.}\label{subsection:NA-LQ_control_simulations}

In this section we present the numerical results obtained for a test problem with the following parameters:
\begin{gather}
	\theta = 0.25, b = 0.5, C = 1, \varepsilon = 0.9
\end{gather}
We truncated the domain with $m \in [-1,1]$, $z \in [0,1]$, and $T \in [0,1]$, as described in section \ref{subsubsec:NA-LQ_control_domain}. The step sizes $\Delta t = 0.0125$, $\Delta m = 0.1$, $\Delta z = 0.1$ were chosen to satisfy condition \ref{monotone_condition}. 

The tests were performed for three cases: 
\begin{itemize}
	\item Control under no observations
	\item Control under noisy observations (observations happen with time step $\Delta t_{obs}=20\Delta t$)
	\item Control under perfect observation (no noise, full observation in continuous time)
\end{itemize}

Figure \ref{Case1} (A) shows the projection of the value function onto axes of $m$ and $z$ at time $t=0$, in the case of no observations. Observe that the cost is smaller for smaller values of $z$, because small $z$ means more accurate information about the state. Figure \ref{Case1} (B) shows the slices of the value function onto axes of $t$ and $m$ with $z=1$; and (C) onto axes of $t$ and $z$ with $m=1$.

\begin{figure}[h!]
	\centering
	\begin{subfigure}{0.33\textwidth}
		\includegraphics[width=1\textwidth]{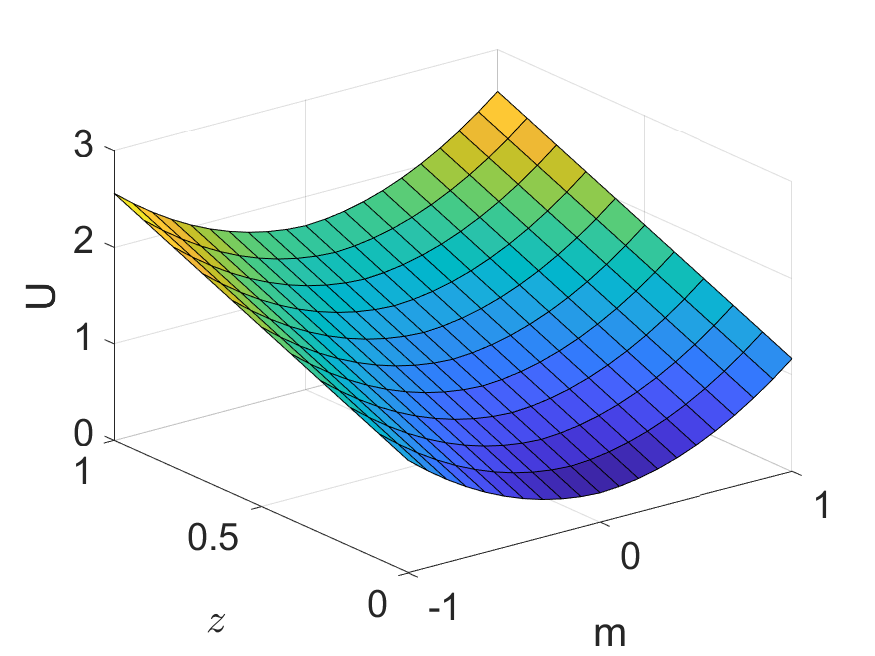}
		\caption{}
		\label{fig:NE_ValFunc_NoObs_a}
	\end{subfigure}
	\begin{subfigure}{0.32\textwidth}
		\includegraphics[width=1\textwidth]{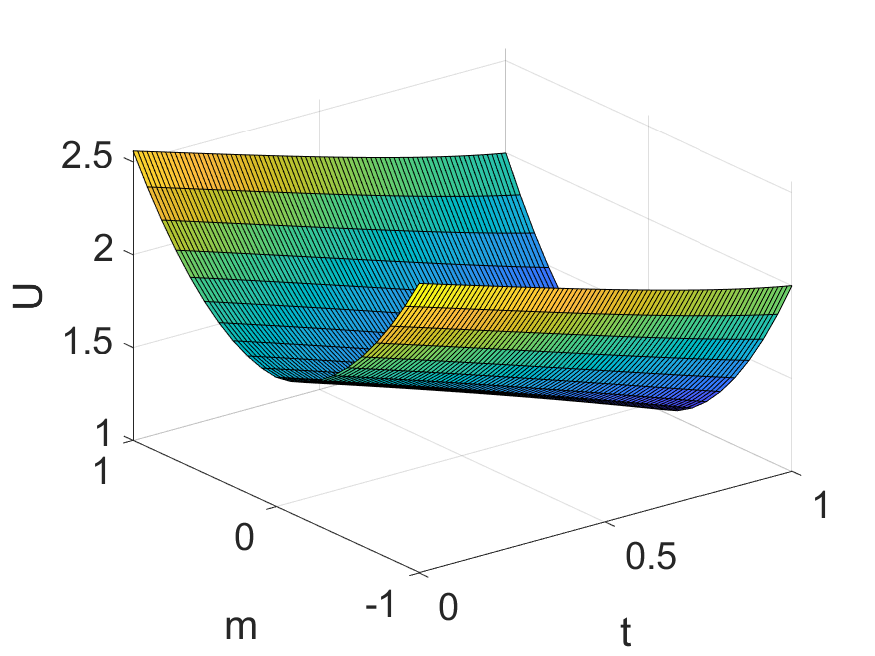}
		\caption{}
		\label{fig:NE_ValFunc_NoObs_b}
	\end{subfigure}
	\begin{subfigure}{0.33\textwidth}
		\includegraphics[width=1\textwidth]{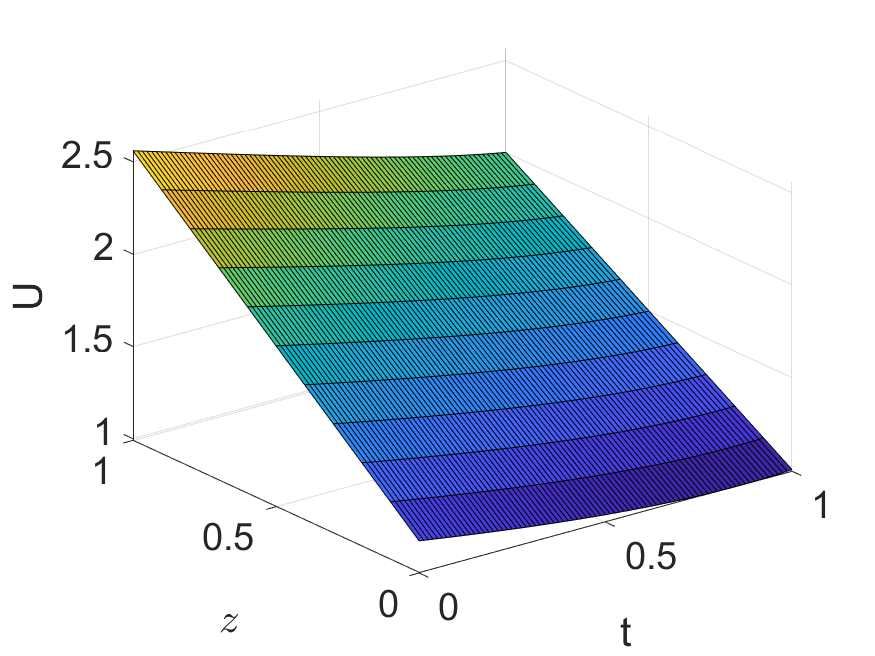}
		\caption{}
		\label{fig:NE_ValFunc_NoObs_c}
	\end{subfigure}
	\caption{\label{Case1} The value function of linear-quadratic control problem under no observations. (A) The value function at initial time; (B) the slice of value function w.r.t. $m$ and $t$, fixed $z = 1$; (C) the slice of value function w.r.t. $z$ and $t$, fixed $m = 1$.}
\end{figure}

Figure \ref{Case2} shows the same slices of the value function in the case with noisy observations made with time step $\Delta t_{obs}=20\Delta t$. Overall we see that the cost is smaller compared to the case with no observations (Figure \ref{Case1}). Furthermore, according to (\ref{eq:parameterized-HJB-OU-2_num}), at each observation point $t^-_n$ and all $z$ the value function is updated from the value function at points with significantly smaller variance $\frac{z\varepsilon^2}{z+\varepsilon^2}$ at time $t_n$, where the value function is smaller. That explains the jumps at each observation point which can be seen on Figure \ref{Case2} (B) and (C). The information gained due to observation improves our knowledge and facilitates better decision making.
\begin{figure}[h!]
	\centering
	\begin{subfigure}{0.33\textwidth}
		\includegraphics[width=1\textwidth]{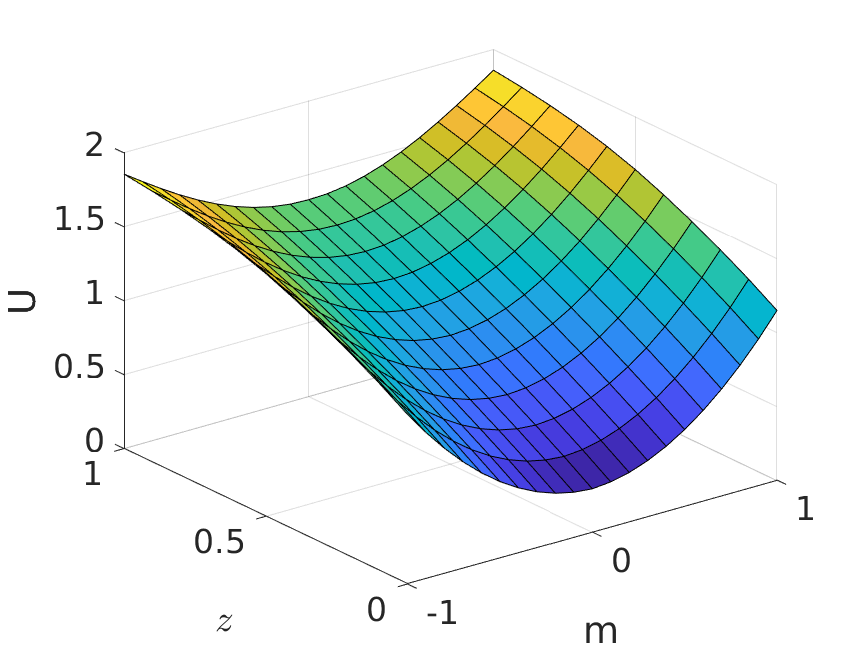}
		\caption{}
		\label{fig:NE_ValFunc_a}
	\end{subfigure}
	\begin{subfigure}{0.32\textwidth}
		\includegraphics[width=1\textwidth]{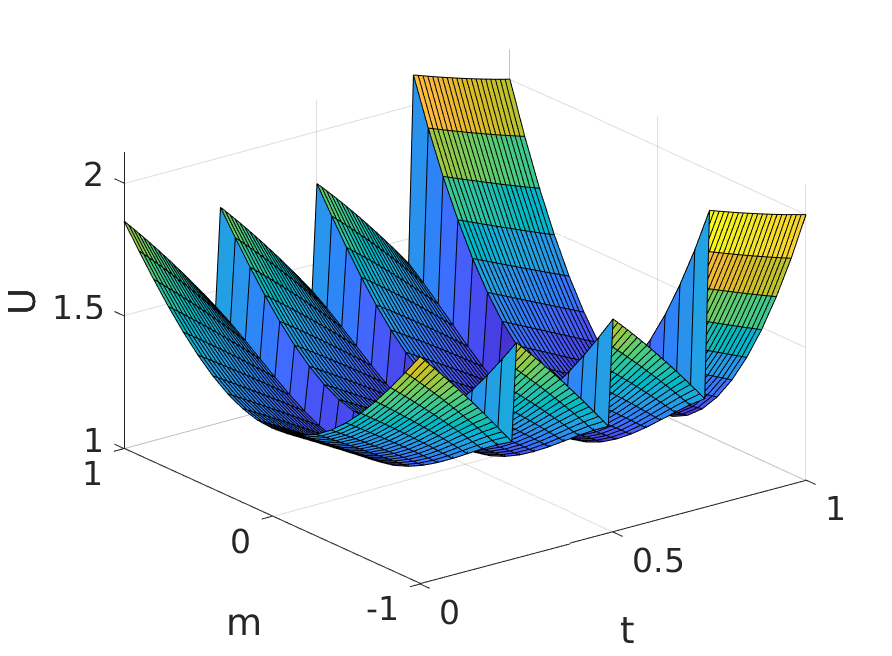}
		\caption{}
		\label{fig:NE_ValFunc_b}
	\end{subfigure}
	\begin{subfigure}{0.33\textwidth}
		\includegraphics[width=1\textwidth]{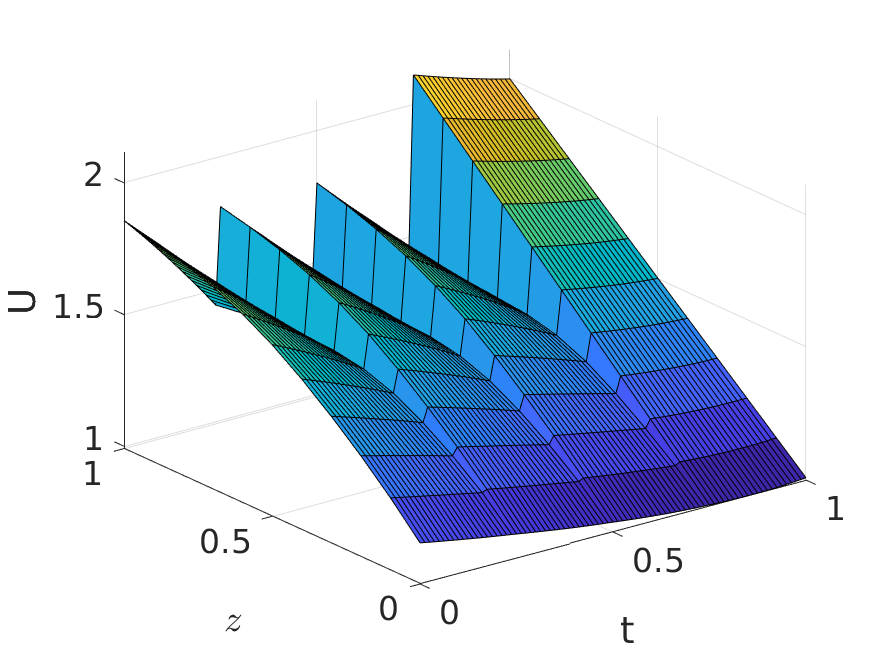}
		\caption{}
		\label{fig:NE_ValFunc_c}
	\end{subfigure}
	\caption{\label{Case2} The value function of linear-quadratic control problem under noisy observations. (A) The value function at initial time; (B) the slice of value function w.r.t. $m$ and $t$, fixed $z = 1$; (C) the slice of value function w.r.t. $z$ and $t$, fixed $m = 1$.}
\end{figure}

Figure \ref{Case2_OptPath} shows the optimal paths of the mean $m^*(t)$ and variance $z^*(t)$. The optimal paths are computed according to the dynamics (\ref{eq:optimal_path}) between the observation points and updated at the observation points using the simulated noisy observations of $X(t_{n+1})$, namely $Y(t_{n+1})~\sim~\mathcal{N}(m^*(t^-_{n+1}),z^*(t^-_{n+1})+\varepsilon^2)$ as follows:
\begin{align} 
	\begin{cases}
		dm^*(t)= \Big(-\theta m^*(t) + \frac{\partial \hat{U}}{\partial m}(t,m^*(t),z^*(t))\Big)dt,&t_n \leq t < t_{n+1},~n=0,...,N \\
		m^*(t_{n+1}) = m^*(t^-_{n+1}) + \frac{z^*(t^-_{n+1})}{z^*(t^-_{n+1}) + \varepsilon^2}(Y(t_{n+1})-m^*(t^-_{n+1})),&n=0,...,N \\
		m^*(t_0) = m_0,\\
		dz^*(t)= \Big(-2\theta z^*(t) + b^2\Big) dt,&t_n \leq t < t_{n+1},~n=0,...,N\\
		z^*(t_{n+1})= \frac{z^*(t^-_{n+1})\varepsilon^2}{z^*(t^-_{n+1}) + \varepsilon^2},&n=0,...,N \\
		z^*(t_0) = z_0
	\end{cases}\label{eq:optimal_path}
\end{align}
The measurements $Y(t_{n+1})$ are independent in the sense that $Y(t_{n+1}) = m^*(t^-_{n+1}) + \sqrt{z^*(t^-_{n+1}) + \varepsilon^2}~W_{n+1}$ with $\{W_{n+1}\}_{n \geq 0}$ i.i.d. $\sim \mathcal{N}(0,1)$. Note also that $m^*$ is updated using the obtained observation, while $z^*$ evolves deterministically. The probability density function, $\rho(x)$, over the optimal path becomes more peaked around the mean from one observation point to the next as we collect information through Bayesian updates, as seen in Figure \ref{Case2_OptPath}(C).

\begin{figure}[h!]
	\centering
	\begin{subfigure}{0.33\textwidth}
		\includegraphics[width=1\textwidth]{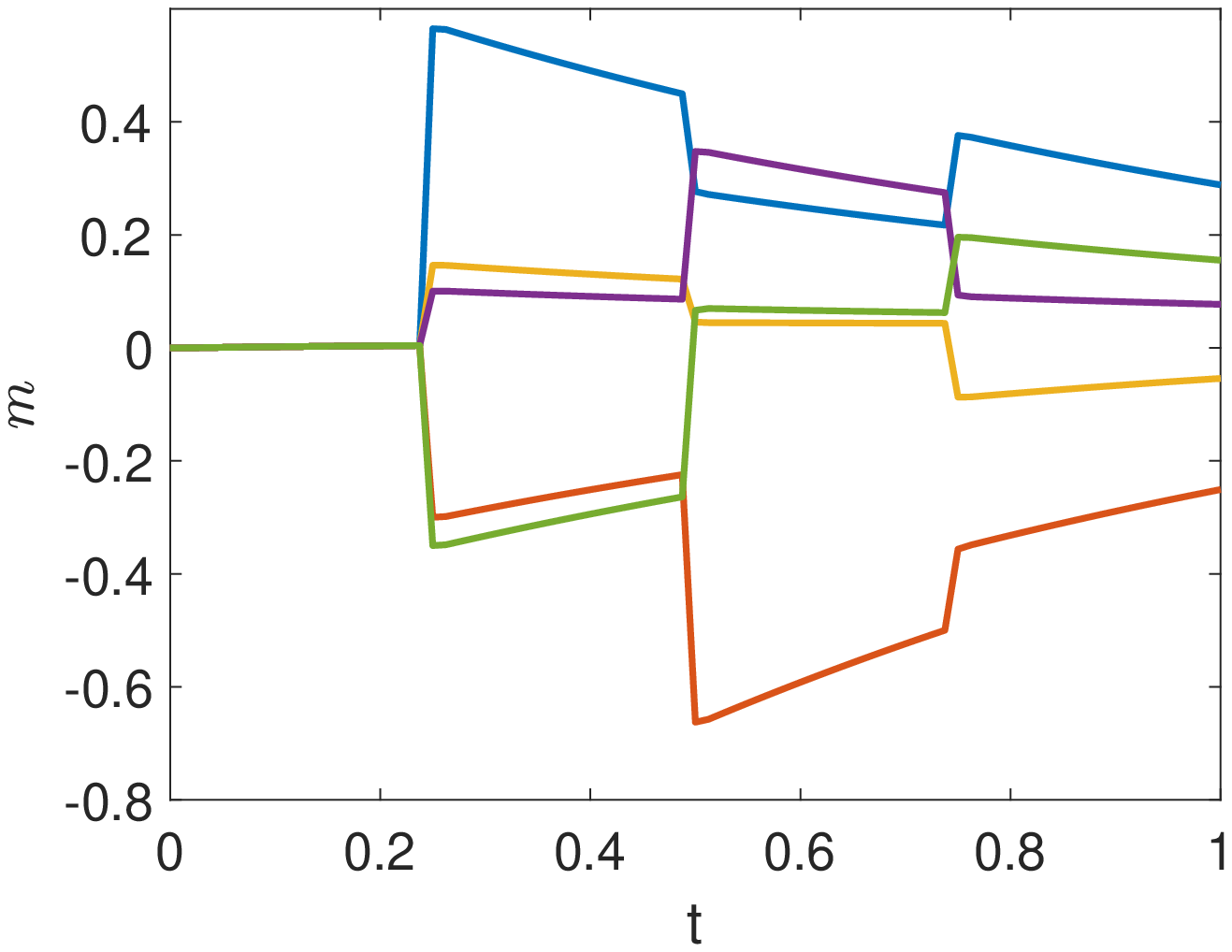}
		\caption{}
		\label{fig:NE_OptPath_a}
	\end{subfigure}
	\begin{subfigure}{0.32\textwidth}
		\includegraphics[width=1\textwidth]{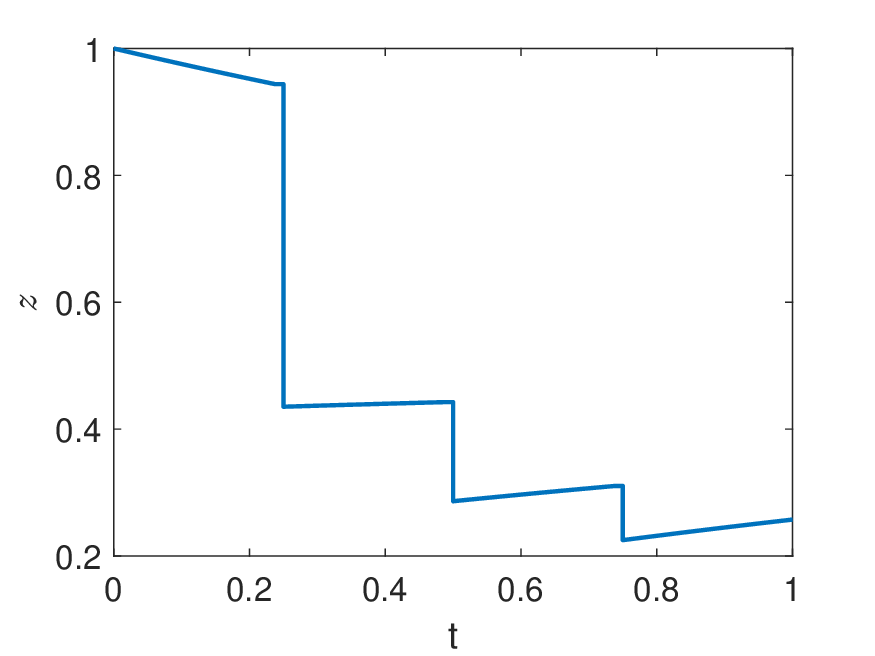}
		\caption{}
		\label{fig:NE_OptPath_b}
	\end{subfigure}
	\begin{subfigure}{0.33\textwidth}
		\includegraphics[width=1\textwidth]{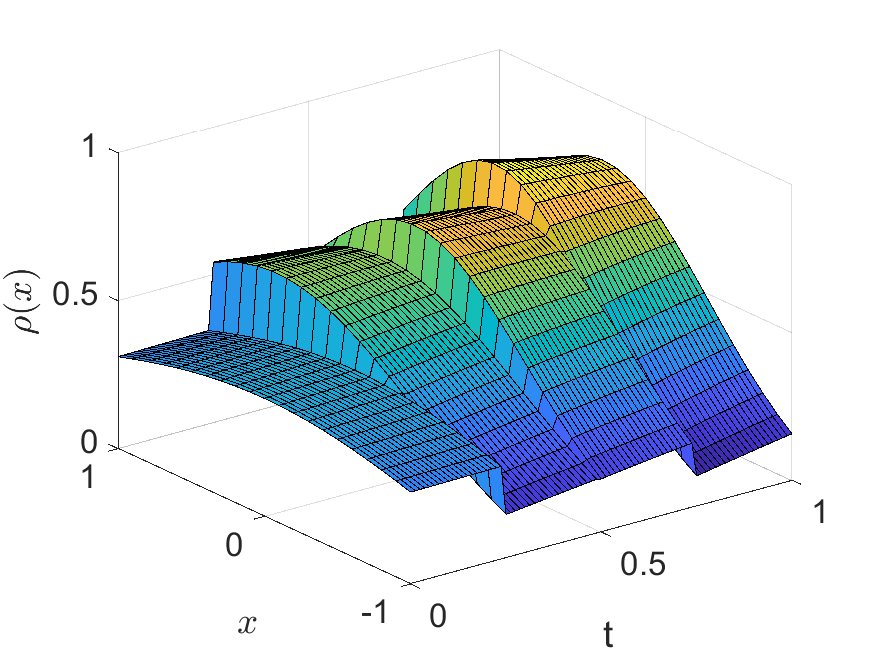}
		\caption{}
		\label{fig:NE_OptPath_c}
	\end{subfigure}
	\caption{\label{Case2_OptPath} The optimal paths of the mean $m^*(t)$ and variance $z^*(t)$ of linear-quadratic control under noisy observations problem. (A) The optimal paths of the mean $m^*(t)$ for several simulated scenarios; (B) the optimal path of the variance $z^*(t)$; (C) the probability density function, $\rho(x)$, over one of the optimal paths (the blue line on plot (A)).}
\end{figure}

Figure \ref{Case3} illustrates the value function's slices w.r.t. $m$ and $t$ with fixed $z = 1$ for the three scenarios: (A) - perfect observations; (B) - noisy observations; (C) - no observations. Evidently, the case with perfect observations provides a lower bound for the solution in the presence of noisy observations, while the scenario with no observations gives the upper bound for the solution.

\begin{figure}[h!]
	\centering
	\begin{subfigure}{0.32\textwidth}
		\includegraphics[width=1\textwidth]{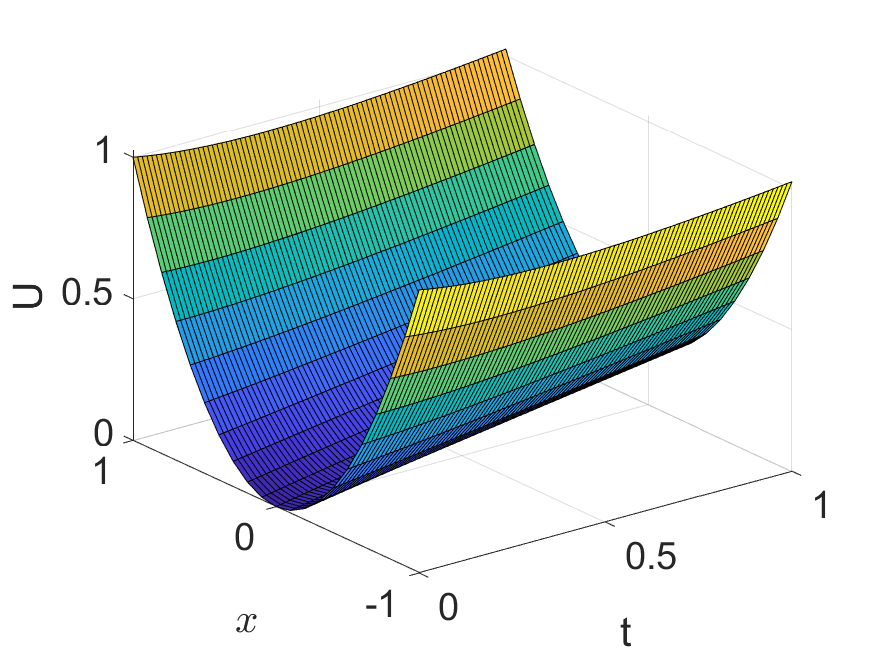}
		\caption{}
		\label{fig:NE_ValFunc_NoObs_a}
	\end{subfigure}
	\begin{subfigure}{0.32\textwidth}
		\includegraphics[width=1\textwidth]{pictures/U_t_m}
		\caption{}
		\label{fig:NE_ValFunc_FullInfo}
	\end{subfigure}
	\begin{subfigure}{0.32\textwidth}
		\includegraphics[width=1\textwidth]{pictures/U_m_t_no_obs_time_rev}
		\caption{}
		\label{fig:NE_ValFunc_NoObs_b}
	\end{subfigure}
	\caption{\label{Case3} The comparison of the value functions at time $t=0$ in three cases. (A) The value function of of linear-quadratic problem with full information; this case corresponds to the variance $z=0$ in cases B and C with partially available information, and $x$ is the exact known state which we compare to the mean $m$ in B and C; (B) the slice of the value function of linear-quadratic problem with observations w.r.t. $m$ and $t$, fixed $z = 1$; (C) the slice of the value function of linear-quadratic problem with no observations w.r.t. $m$ and $t$, fixed $z = 1$.}
\end{figure}

\subsection{Example 2: Optimal control of measurement noise}\label{subsec:NA-measurement_noise_control}

In this example, we extend our formulation to additionally control the noise of the observations $\pmb{\beta}$.
 
\subsubsection{Experiment set up}. Rather than fixing the measurement noise at $\varepsilon$, we allow each of $N$ observations 
\begin{gather}
Y(t_n) = X(t_n)+\beta_nZ_n,\quad\quad Z_n \sim \mathcal{N}(0,1),
\quad n=1,\dots,N,
\end{gather}
to use an adaptive noise level $\beta_n\in(0,\beta_{\max}]$.  Reducing $\beta_n$ improves accuracy but incurs an experimental cost $\frac{c_n}{\beta_n}$.

As discussed in Section~\ref{sec:optimal-control-an}, this creates a nontrivial trade-off between estimation error and measurement expense. The HJB associated with such value function is now given by
\begin{subequations}
	\label{eq:parameterized-HJB-OU_num_beta}
	\begin{align}
		\label{eq:parameterized-HJB-OU-1_num_beta}
		\frac{\partial \widehat{U}}{\partial t}(t, m, z) &+ z + m^2 - \theta m \frac{\partial \widehat{U}}{\partial m} (t, m, z) + (b^2 - 2\theta z) \frac{\partial \widehat{U}}{\partial z}(t, m, z) \\ \nonumber	 
		&- \frac{1}{4C} \left[ \frac{\partial \widehat{U}}{\partial m}(t, m, z) \right]^2 = 0,\quad t_n \le t < t_{n+1},\quad n=0,\dots,N\\
		\label{eq:parameterized-HJB-OU-2_num_beta}
		\widehat{U}(t_n^-, m, z) &= \underset{\beta_n \in (0,\beta_{\max}]}{\min}\int_{\mathbb{R}} \widehat{U}\biggl( t_n, m+ \frac{z}{\sqrt{z + \beta_n^2}}w, \frac{z \beta_n^2}{z + \beta_n^2} \biggr) \phi(w) d w + \frac{c_n}{\beta_n},\\ 
		\widehat{U}(T,m,z) &= m^2 + z,
	\end{align}
\end{subequations}
where at each observation step $n$ we optimize components of $\pmb{\beta} = \{\beta_1,...,\beta_N\}$ to achieve the smallest cost.

To quantify the benefit of noise control, we compare against the fixed-noise case $\beta_n \equiv \varepsilon$.  To compute the value function in a fixed-noise case, we replaces the minimization in \eqref{eq:parameterized-HJB-OU-2_num_beta} by the single update
\begin{gather}
	\widehat{U}(t_n^-, m, z) = \int_{\mathbb{R}} \widehat{U}\biggl( t_n, m+ \frac{z}{\sqrt{z + \varepsilon^2}}w, \frac{z \varepsilon^2}{z + \varepsilon^2} \biggr) \phi(w) dw + \frac{c_n}{\varepsilon},
\end{gather}

\subsubsection{Simulation and results}. We consider a simulation over $t \in [0,1]$ with $N=3$ measurements at
\begin{gather}
t_{\rm obs} = [0.25, 0.5, 0.75],
\quad
c_n = [0.05, 0.01, 0.001].
\end{gather}
In principle, the measurement noise level $\beta_n$ is unbounded, since it is modeled as Gaussian. However, in practice, real‐world measurement systems have finite precision. To reflect this, we introduce an effective upper bound $\beta_{\max}$ based on the natural variability of the OU process. Over the short time interval $T-t_0$, the conditional variance of the OU process grows approximately as $z_0^2(T-t_0)$. Applying the three‐sigma rule which captures about 99.7\% of Gaussian outcomes, we set
\begin{gather}
	\beta_{\max} = 3\sqrt{\mathbb{V}[X_T - X_{t_0}]} \approx 3z_0\sqrt{T - t_0},
\end{gather}
so that the maximum measurement noise remains on the same high‐probability scale as the process noise. For the numerical experiments, we truncate the computational domain to $m \in [-7,7]$, $z \in [0,1]$, ensuring that it comfortably covers this range of high‐probability fluctuations. Finally, we choose $\Delta t = 0.0125$, $\Delta m = 0.1$, $\Delta z = 0.1$ to ensure the monotonicity condition (Condition~\ref{monotone_condition}) is satisfied.

Figure \ref{fig:NE_ValFuncBeta} shows slices of the value function for the case with controlled noise $\pmb{\beta}$. In Figure \ref{fig:NE_CostsBeta} we compare the expected cost‐to‐go: the controlled noise strategy always outperforms any fixed‐noise baseline $\varepsilon$. Figure \ref{fig:NE_OptPathsBeta} shows the typical optimal paths of the mean $m^*(t)$, the variance $z^*(t)$, and controls $\alpha^*(t)$ and $\pmb{\beta}^*$. The average optimal noise levels are given by $\pmb{\beta}^*=[0.49, 0.52, 0.75]$ (estimated over $10^4$ simulations).
Figure \ref{fig:NE_BetaFunction} shows the optimal experiment noises $\beta_1^*$, $\beta_2^*$ and $\beta_3^*$ as functions of $m$ and $z$. In the high uncertainty region (large $z$) of Figure \ref{fig:NE_BetaFunction_t1}, $\beta_1^*$ is driven to its minimum — intuitively, when uncertainty is large we invest in a very precise (low-noise) measurement. As we start the simulation with large initial variance $z_0=1$, we observe that a precise measurement with noise $\beta_1^*=0.49$ is chosen at first observation time $t_1=0.25$. This reduces the variance $z^*(t)$ significantly, and by the time we reach the third observation at $t_3=0.75$, $\beta_3^*$ rises, reflecting the decreasing value of further measurement noise reduction.  

\begin{figure}[h!]
	\centering
	\begin{subfigure}{0.32\textwidth}
		\includegraphics[width=1\textwidth]{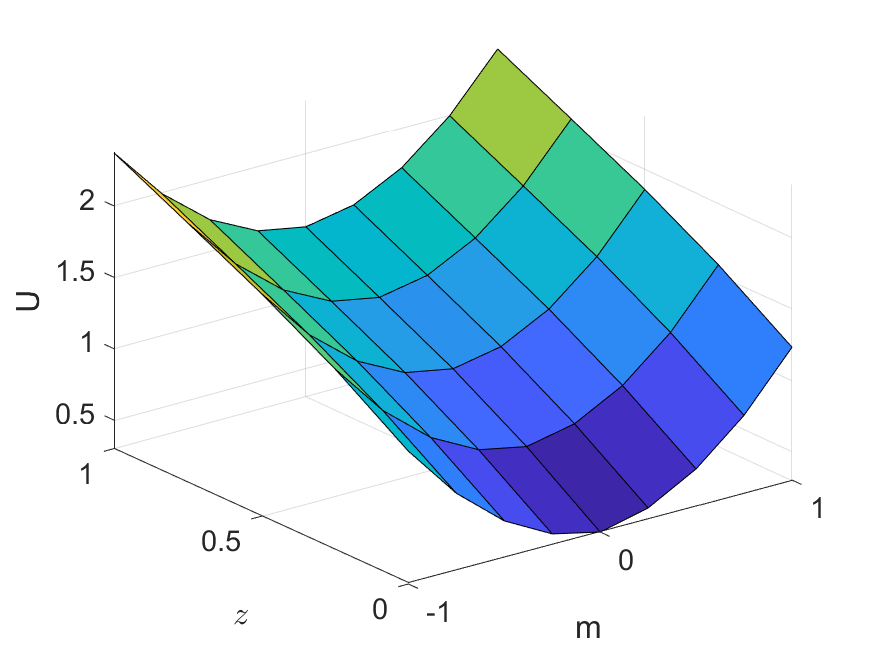}
		\caption{}
		\label{fig:NE_ValFuncBeta_ut0}
	\end{subfigure}
	\begin{subfigure}{0.32\textwidth}
		\includegraphics[width=1\textwidth]{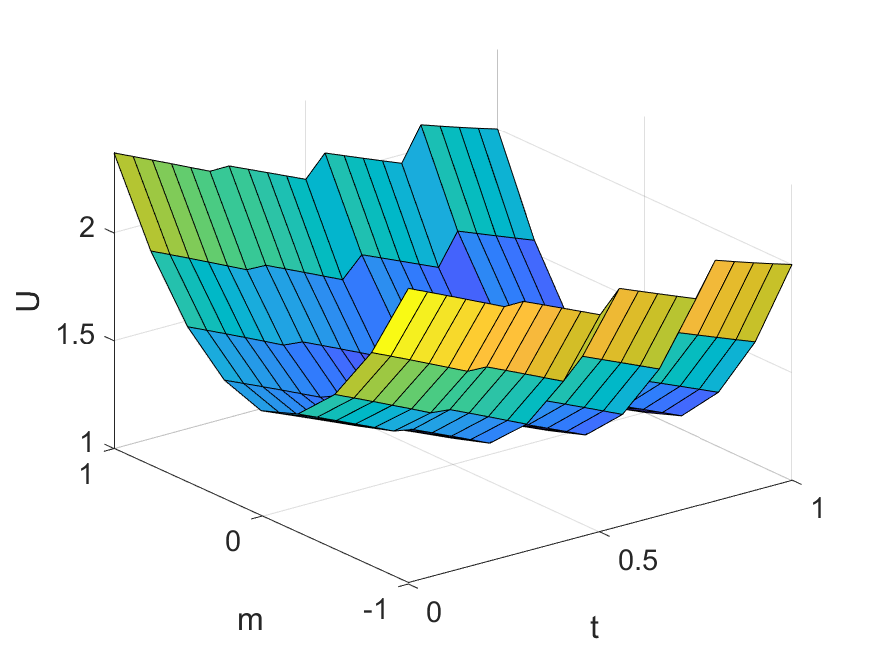}
		\caption{}
		\label{fig:NE_ValFuncBeta_utm}
	\end{subfigure}
	\begin{subfigure}{0.32\textwidth}
		\includegraphics[width=1\textwidth]{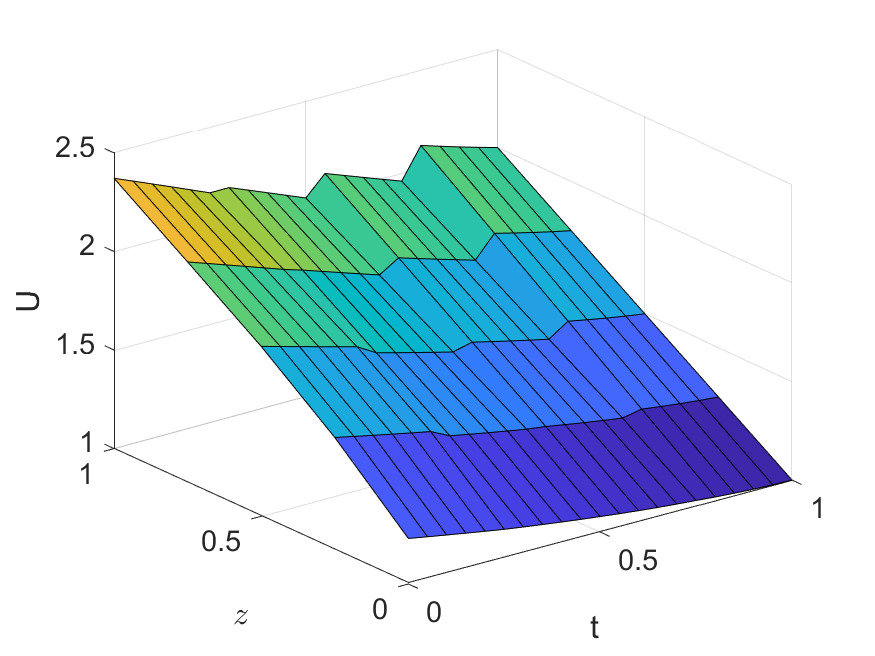}
		\caption{}
		\label{fig:NE_ValFuncBeta_utz}
	\end{subfigure}
	\caption{\label{fig:NE_ValFuncBeta} The slices of the value function for optimal control of measurement noise $\pmb{\beta}$ problem. (A) $(m,z)$-slice at time $t_0=0$; (B)  $(t,m)$-slice at fixed $z=1$; (C) $(t,z)$-slice at fixed $m=1$.} 
\end{figure}
\begin{figure}[h!]
	\centering
	\includegraphics[width=0.5\textwidth]{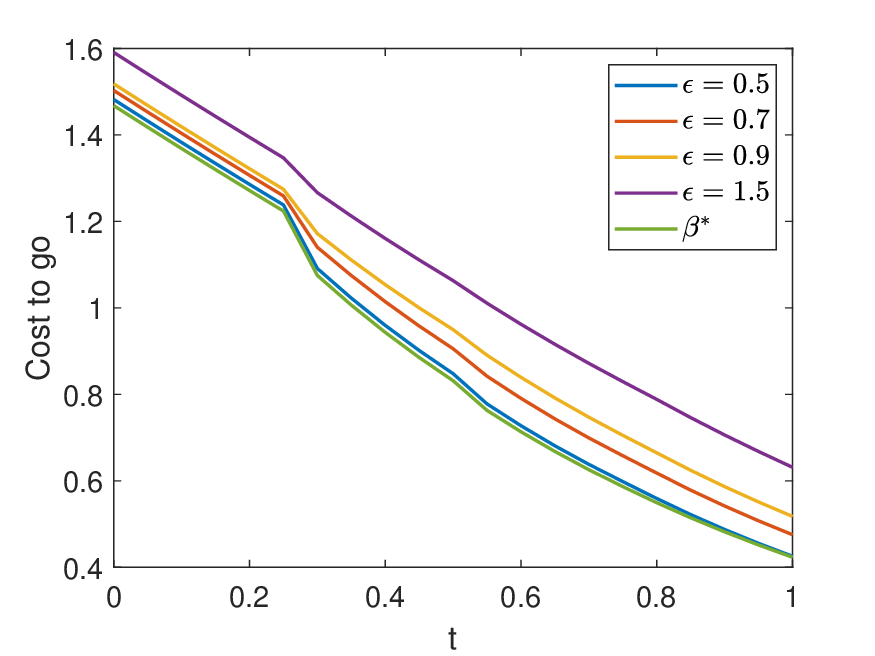}
	\caption{\label{fig:NE_CostsBeta} Comparison of the expected cost to go for cases with controlled noise level $\pmb{\beta}^*$ and fixed noises $\varepsilon$ (averaged over $10^4$ simulations).}
\end{figure}
\begin{figure}[h!]
	\centering
	\begin{subfigure}{0.48\textwidth}
		\includegraphics[width=1\textwidth]{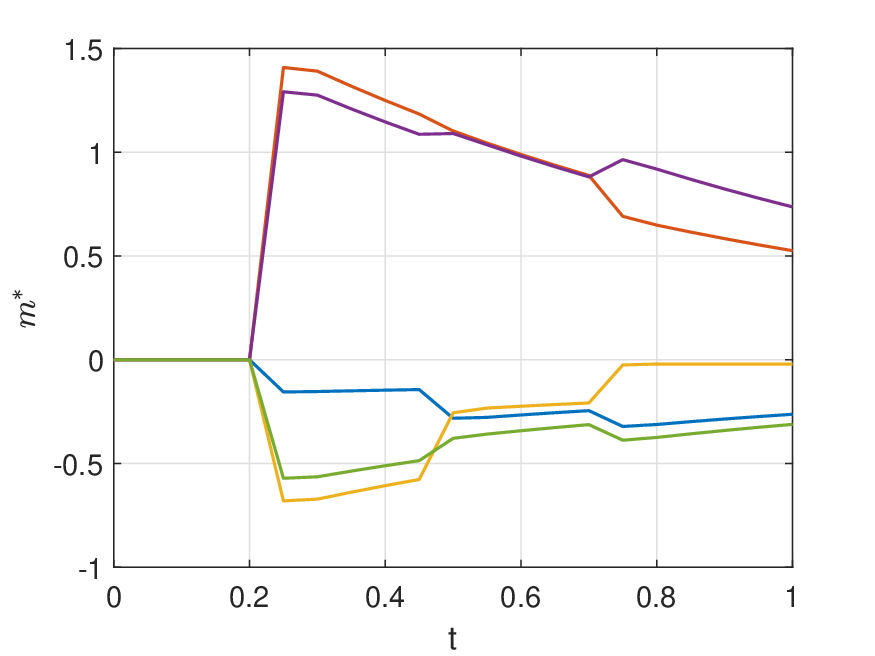}
		\caption{}
		\label{fig:NE_OptPathsBeta_m}
	\end{subfigure}
	\begin{subfigure}{0.48\textwidth}
		\includegraphics[width=1\textwidth]{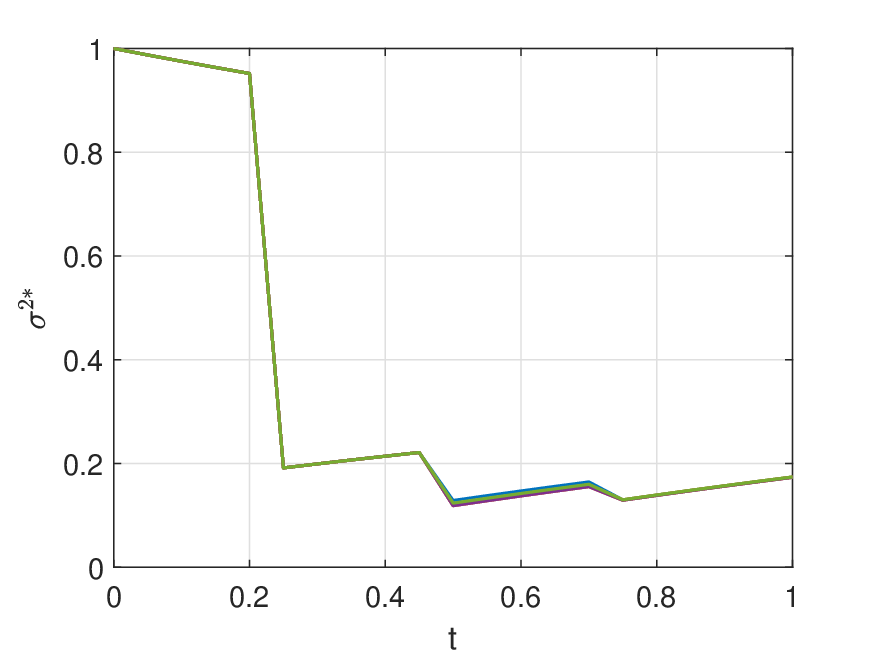}
		\caption{}
		\label{fig:NE_OptPathsBeta_z}
	\end{subfigure}
	\begin{subfigure}{0.48\textwidth}
		\includegraphics[width=1\textwidth]{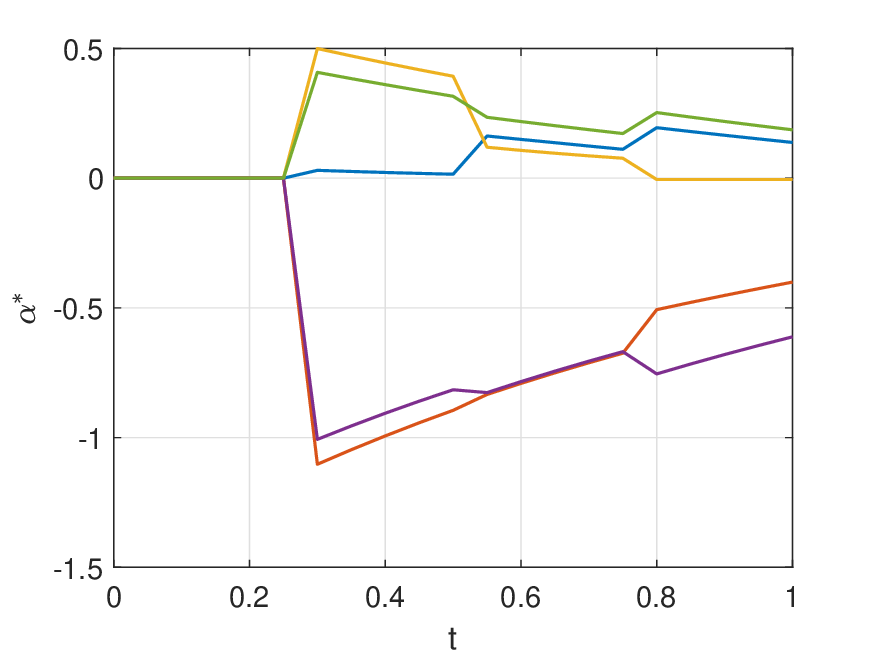}
		\caption{}
		\label{fig:NE_OptPathsBeta_alpha}
	\end{subfigure}
	\begin{subfigure}{0.48\textwidth}
		\includegraphics[width=1\textwidth]{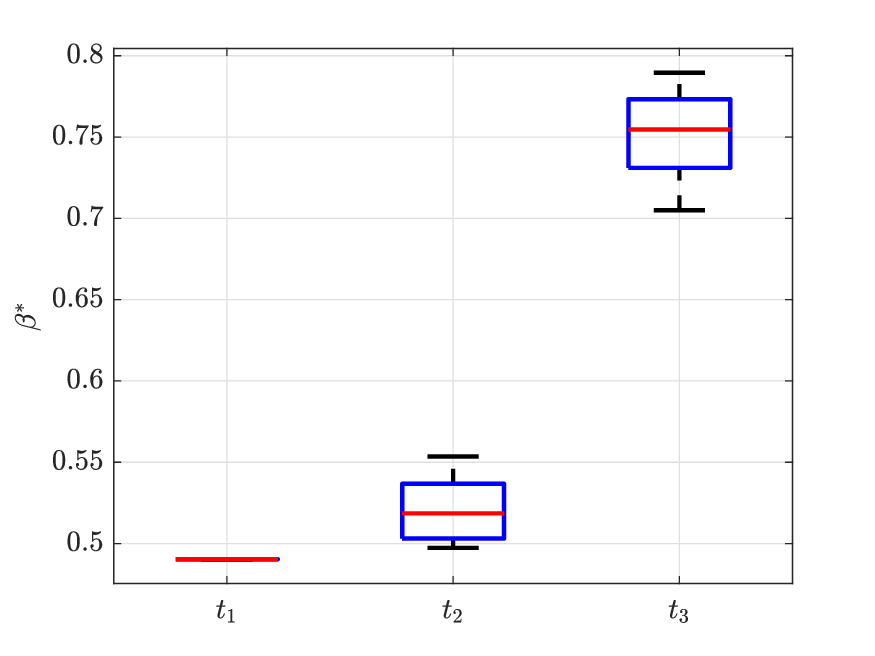}
		\caption{}
		\label{fig:NE_OptPathsBeta_beta}
	\end{subfigure}
	\caption{\label{fig:NE_OptPathsBeta} (A) The optimal path of the mean $m^*$; (B) the optimal path of the variance $z^*$; (C) control $\alpha^*$ over the optimal path; (D) the optimal noise levels $\beta_n^*$ at each observation step for some realizations of the process.}
\end{figure}
\begin{figure}[h!]
	\centering
	\begin{subfigure}{0.32\textwidth}
		\includegraphics[width=1\textwidth]{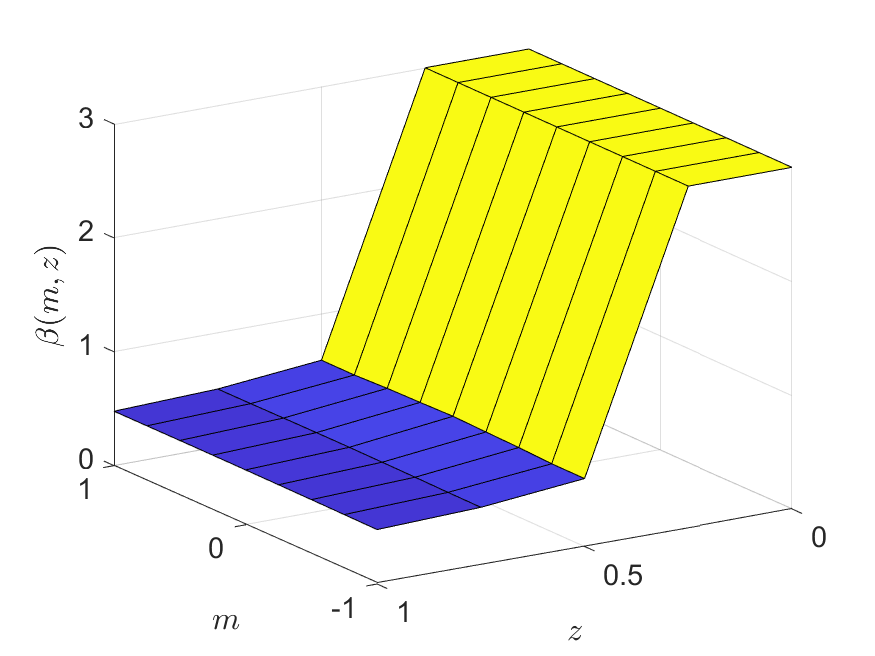}
		\caption{}
		\label{fig:NE_BetaFunction_t1}
	\end{subfigure}
	\begin{subfigure}{0.32\textwidth}
		\includegraphics[width=1\textwidth]{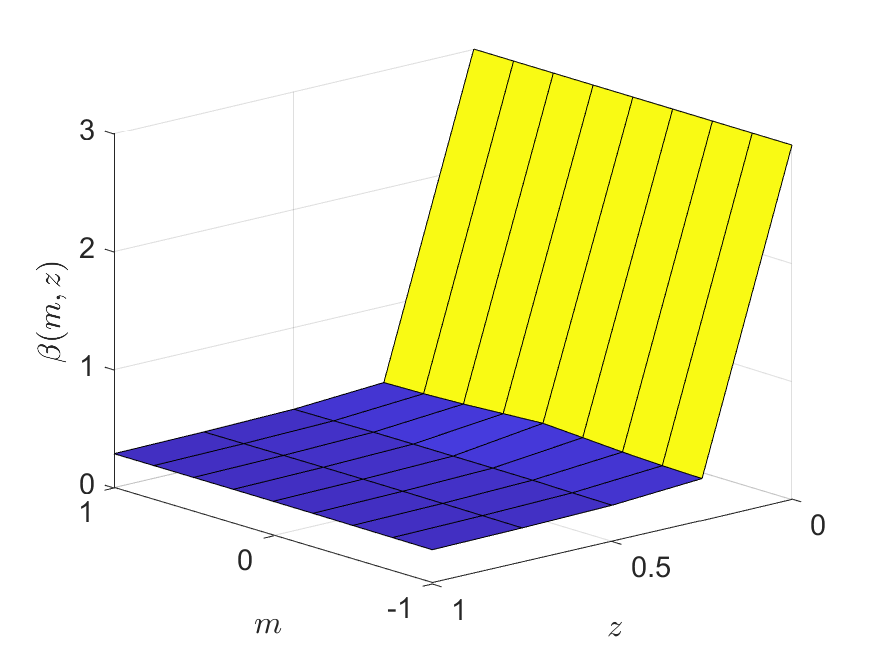}
		\caption{}
		\label{fig:NE_BetaFunction_t2}
	\end{subfigure}
	\begin{subfigure}{0.32\textwidth}
		\includegraphics[width=1\textwidth]{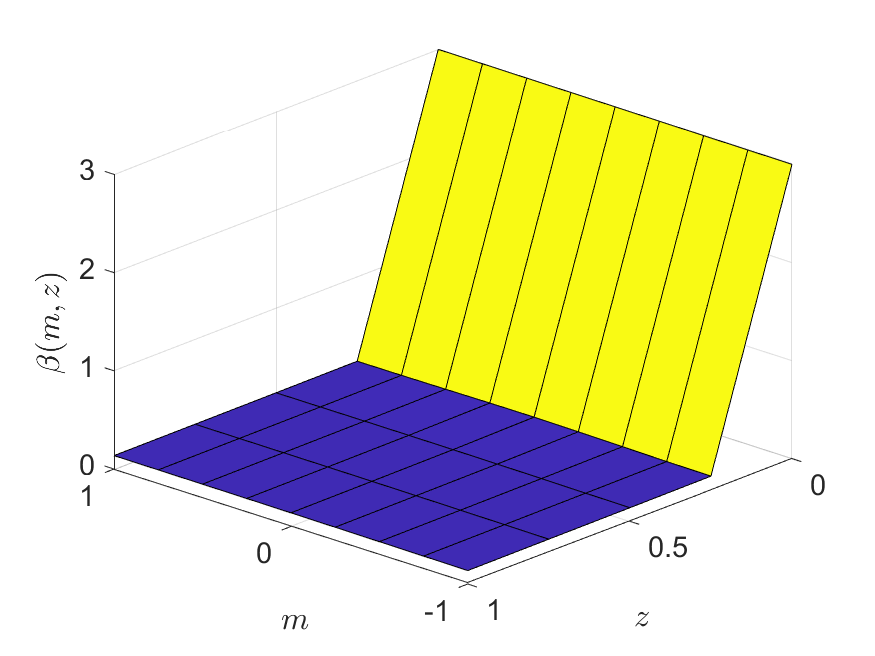}
		\caption{}
		\label{fig:NE_BetaFunction_t3}
	\end{subfigure}
	\caption{\label{fig:NE_BetaFunction} The optimal solution $\pmb{\beta}^*$ as a function of $m$ and $z$ at three observations: (A) $t=0.25$; (B) $t=0.5$; (C) $t=0.75$.}
\end{figure}

\subsection{Example 3: Optimal control with high-penalty region}\label{subsec:NA-Obstacle}

To conclude this section, we illustrate our method on a one‐dimensional Ornstein–Uhlenbeck process subject to a high‐cost “obstacle” region, described in \cite{tottori_kobayashi23}.

\subsubsection{Experiment set up.}  The state $X(t)$ evolves according to a mean-reverted dynamics:
\begin{gather}
	dX(t) = (-\theta (X(t) - p) + \alpha(t,X(t)))dt + b dW(t),~~~\forall t \in [0,T],~~X(0)=x_0.
\end{gather}
We aim to minimize the expected cost
\begin{gather}
	J(\alpha^*) = \underset{\alpha(\cdot)}{\min} ~\mathbb{E} \Big[ \int_0^T Q(t,X(t)) + X^2(t) + \alpha^2(t,X(t)) dt + 10X^2(T)\Big],
\end{gather}
with the obstacle function $Q(t,x)$ given by
\begin{gather}
	Q(t,x) = \begin{cases}
		1000,\quad\text{if }0.3 \leq t \leq 0.6,~0.1 \leq |x| \leq 2.0,\\
		0,\quad\quad\text{otherwise}
	\end{cases}
\end{gather}  

A single noisy measurement is taken at time $t_1=T/2$
\begin{gather}
	Y(t_1) = X(t_1) + \varepsilon Z(t_1),~~~Z \sim \mathcal{N}(0,1).
\end{gather}

It is easy to show that the HJB equation for this problem reads:
\begin{subequations}
	\label{eq:parameterized-HJB-OU_num_obstacle}
	\begin{align}
		\label{eq:parameterized-HJB-OU-1_num}
		&\frac{\pa \widehat{U}}{\pa t}(t, m, z) + \bar{Q}(t,m,z + m^2) + m^2 + z - \theta (m-p) \frac{\pa \widehat{U}}{\pa m}(t,m,z) + \\
		& \quad + (b^2- 2\theta z) \frac{\pa \widehat{U}}{\pa s}(t,m,z) 
		- \f{1}{4} \left( \frac{\pa \widehat{U}}{\pa m}(t,m,z) \right)^2 = 0,~~ t_n \le t < t_{n+1},\quad n=0,...,2\\
		\label{eq:parameterized-HJB-OU-2_num_obstacle}
		&\widehat{U}(t_n^-, m, z) = \int_{\R} \widehat{U}\biggl( t_n, m+ \frac{z}{\sqrt{z + \varepsilon^2}}w, \frac{z \varepsilon^2}{z + \varepsilon^2} \biggr) \phi(w) \dd w,~~n=1\\ 
		\label{eq:parameterized-HJB-OU-3_num_obstacle}
		&\widehat{U}(T,m,z) = 10\Big(m^2 + z \Big),
	\end{align}
\end{subequations}
with optimal control given in closed form:
\begin{gather}
	\alpha^*(t) = -\frac{1}{2} \frac{\pa \widehat{U}}{\pa m}(t,m,z).
\end{gather}

\subsubsection{Simulation and results.} We simulate the process over $t \in [0,1]$ with a single observation at time $t_1=0.5$, and the domain defined by $m \in [-8,8]$, $z \in [0,1]$, with step sizes $\Delta t = 5 \cdot 10^{-4}$, $\Delta m = \Delta z = 0.0625$. The parameters of the OU process are $p=2$, $b=0.5$, and observation noise $\varepsilon = 0.2$. 

Figure \ref{fig:NE_ValFunc_obstacle} shows the slices of the value function. As seen in Figure \ref{fig:NE_ValFunc_obstacle_A}, at time $t_0=0$ the value function has a singularity at $m=0$ and is asymmetric about the origin. The obstacle region $\{0.3 \leq t \leq 0.6,~-2 \leq m \leq 2\}$ appears as a clearly visible elevation in the surface of $(t,m)$-slice at $z=1$ (Figure \ref{fig:NE_ValFunc_obstacle_B}). Although there is a narrow gap in the obstacle for $-0.1 \leq x \leq 0.1$, its width is so small relative to the process variance that paths will very likely intersect the high‐cost region — so the control must still push the paths around the obstacle. Figure \ref{fig:NE_ValFunc_obstacle_C} shows $(t,z)$-slice at $m=1$. Being at $m=1$ during $0.3 \leq t \leq 0.6$ incurs high cost, which appears as an elevated band.

\begin{figure}[h!]
	\centering
	\begin{subfigure}{0.32\textwidth}
		\includegraphics[width=1\textwidth]{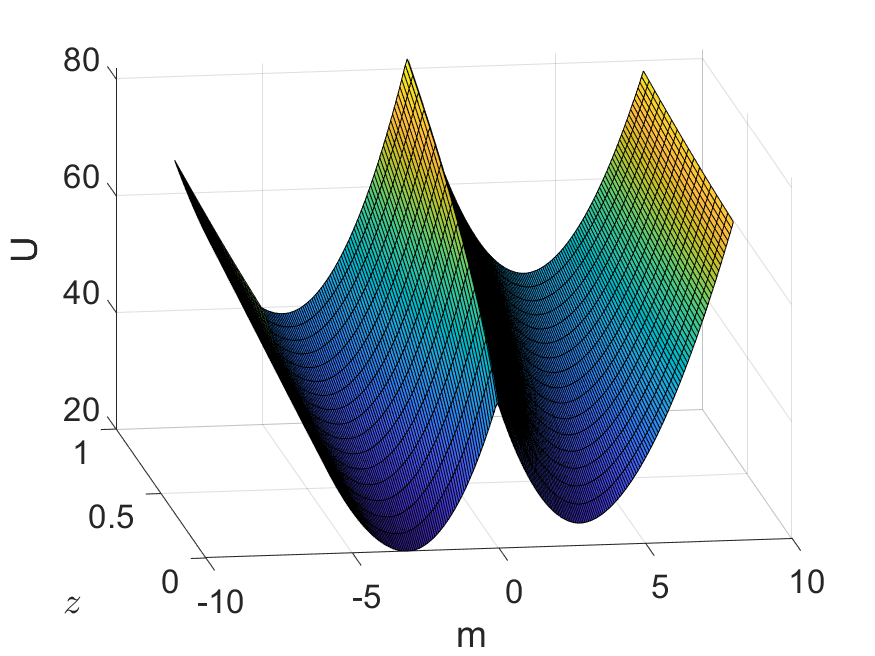}
		\caption{}
		\label{fig:NE_ValFunc_obstacle_A}
	\end{subfigure}
	\begin{subfigure}{0.32\textwidth}
		\includegraphics[width=1\textwidth]{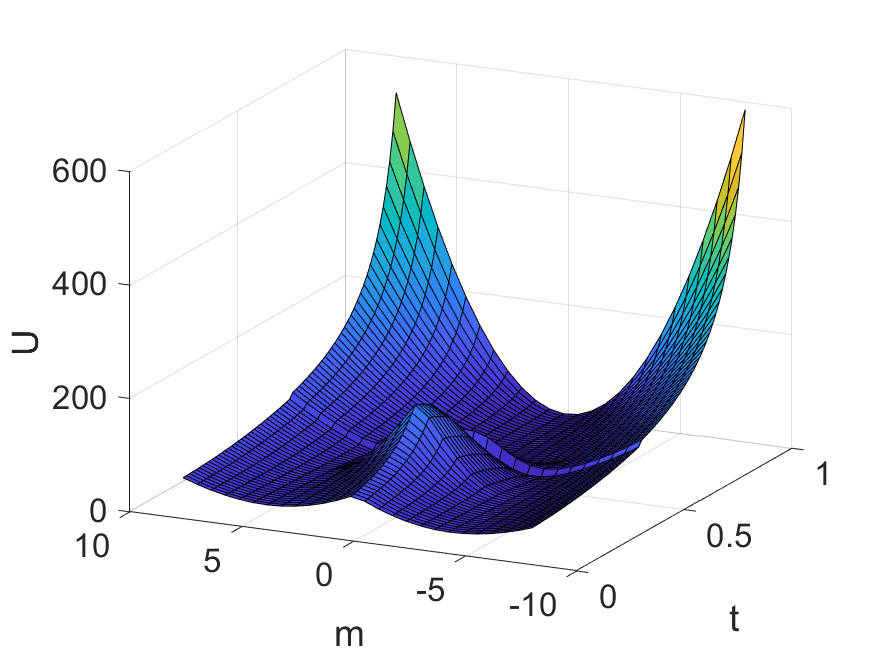}
		\caption{}
		\label{fig:NE_ValFunc_obstacle_B}
	\end{subfigure}
	\begin{subfigure}{0.32\textwidth}
		\includegraphics[width=1\textwidth]{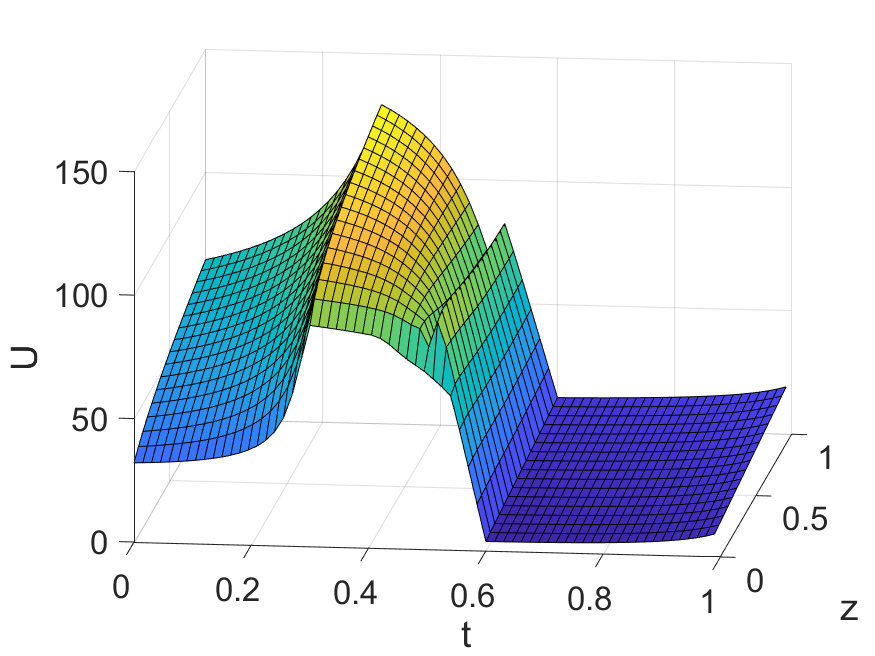}
		\caption{}
		\label{fig:NE_ValFunc_obstacle_C}
	\end{subfigure}
	\caption{\label{fig:NE_ValFunc_obstacle} The slices of the value function for optimal control with high-penalty region. (A) $(m,z)$-slice at time $t_0=0$; (B)  $(t,m)$-slice at fixed $z=1$; (C) $(t,z)$-slice at fixed $m=1$.}
\end{figure} 

Figure \ref{fig:NE_Opt_path_pos} shows the optimal paths of the mean $m^*(t)$, the variance $z^*(t)$, optimal control $\alpha^*(t)$ and cost over the optimal path for initial point $(m_0,z_0) = (0,1)$. As seen in Figure \ref{fig:NE_ValFunc_obstacle_A}, the optimal mean $m^*(t)$ is pushed far enough from the obstacle so that the probability of hitting the obstacle is low. Because of the high terminal cost $10X^2(T)$, the control drives the mean towards zero. After the observation, the variance $z^*(t)$ drops as seen from Figure \ref{fig:NE_ValFunc_obstacle_B}, which allows the trajectories to pass closer to the obstacle with less uncertainty. Figure  \ref{fig:NE_ValFunc_obstacle_C} shows the corresponding optimal control $\alpha^*(t)$. Note in the blue sample the large positive push right after $t=0.5$, needed to avoid the obstacle. The yellow path, starting further out, receives a negative push to drive the mean back towards zero.

\begin{figure}[h!]
	\centering
	\begin{subfigure}{0.48\textwidth}
		\includegraphics[width=1\textwidth]{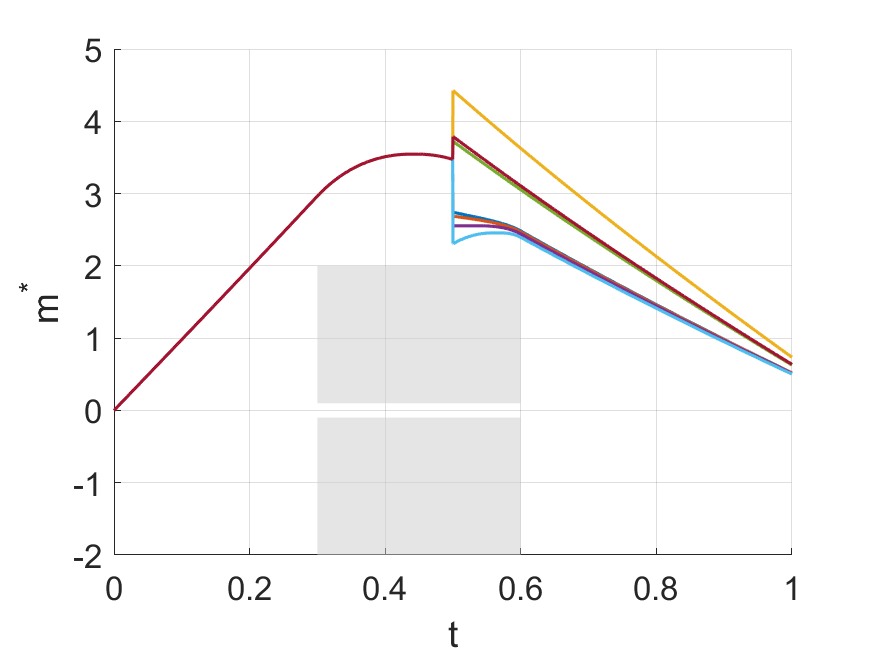}
		\caption{}
		\label{fig:NE_Opt_path_pos_A}
	\end{subfigure}
	\begin{subfigure}{0.48\textwidth}
		\includegraphics[width=1\textwidth]{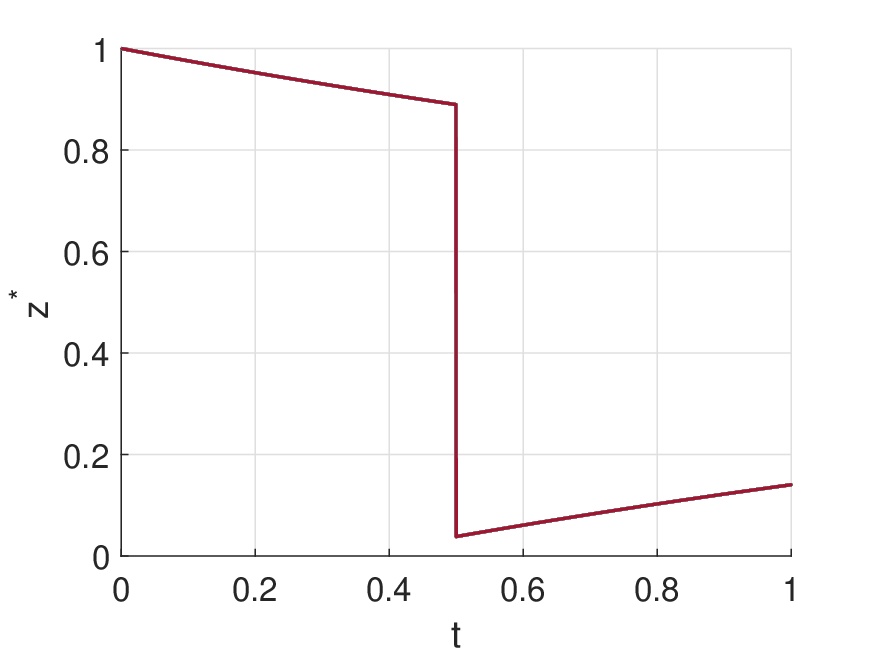}
		\caption{}
		\label{fig:NE_Opt_path_pos_B}
	\end{subfigure}
	\begin{subfigure}{0.48\textwidth}
		\includegraphics[width=1\textwidth]{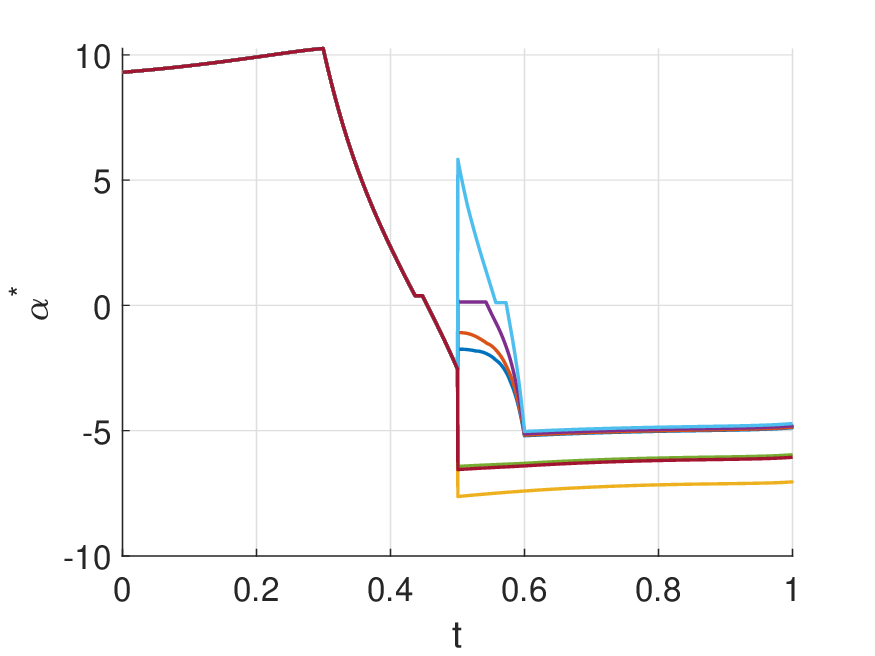}
		\caption{}
		\label{fig:NE_Opt_path_pos_C}
	\end{subfigure}
	\begin{subfigure}{0.48\textwidth}
		\includegraphics[width=1\textwidth]{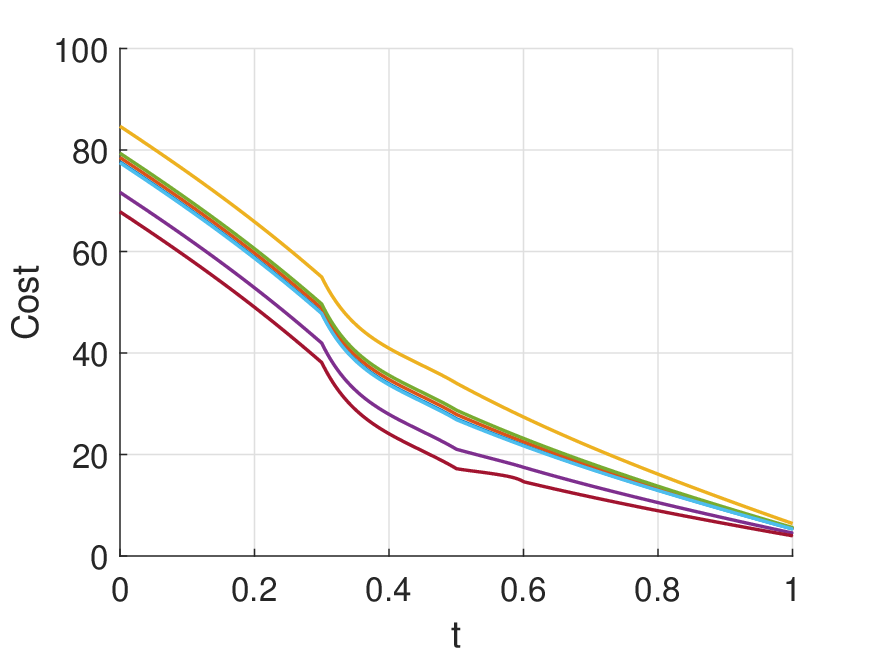}
		\caption{}
		\label{fig:NE_Opt_path_pos_D}
	\end{subfigure}
	\caption{\label{fig:NE_Opt_path_pos} The optimal paths of (A) mean $m^*(t)$, (B) variance $z^*(t)$, (C) controls $\alpha^*(t)$, and (D) cost to go, given the initial point $(m_0,z_0)=(0,1)$. The solution by HJB is $\widehat{U}(t_0,m_0,z_0)=81.0159$, and estimated over optimal paths by Monte Carlo with 100 simulations is $J \approx 79.7914 \pm 2.2261$ (95\% CI: $[77.5653, 82.0175]$).}
\end{figure}

Figure \ref{fig:NE_Opt_path_neg} shows plots of the optimal paths of the mean $m^*$, the variance $z^*$, the optimal control $\alpha^*$ and the cost over the optimal path for a different initial point - $(m_0,z_0) = (-0.125,1)$. Starting from a slightly negative initial mean $m_0=-0.125$, the optimal strategy pushes the mean around the obstacle on the opposite side. Otherwise, the behavior of the optimal paths are analogous to the previous case.
\begin{figure}[h!]
	\centering
	\begin{subfigure}{0.48\textwidth}
		\includegraphics[width=1\textwidth]{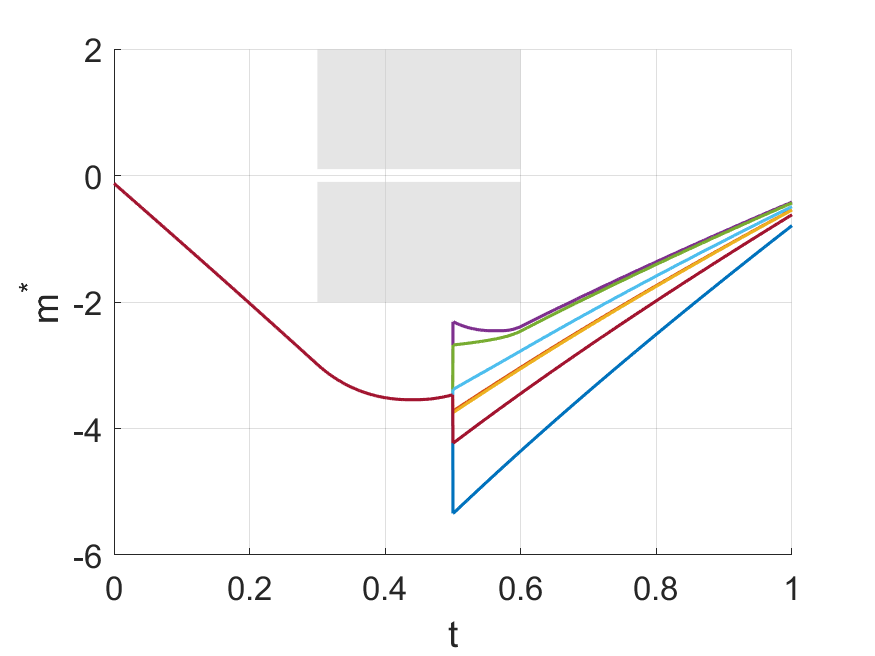}
		\caption{}
		\label{fig:NE_Opt_path_neg_A}
	\end{subfigure}
	\begin{subfigure}{0.48\textwidth}
		\includegraphics[width=1\textwidth]{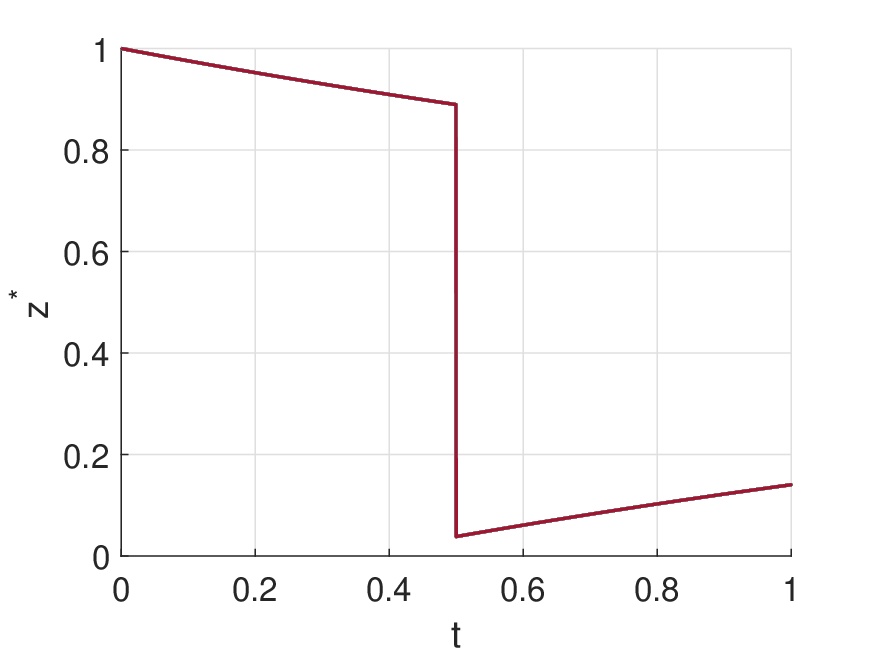}
		\caption{}
		\label{fig:NE_Opt_path_neg_B}
	\end{subfigure}
	\begin{subfigure}{0.48\textwidth}
		\includegraphics[width=1\textwidth]{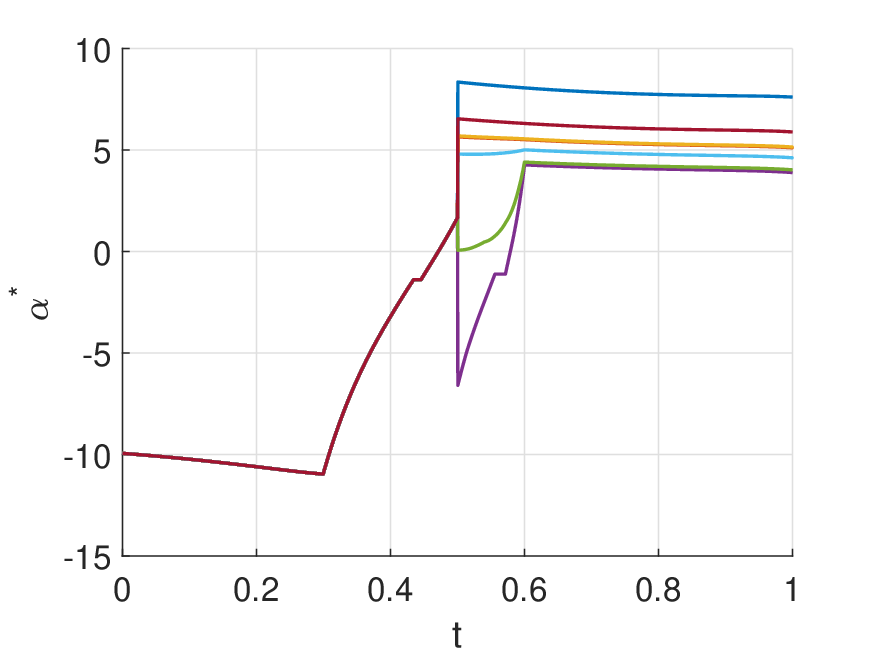}
		\caption{}
		\label{fig:NE_Opt_path_neg_C}
	\end{subfigure}
	\begin{subfigure}{0.48\textwidth}
		\includegraphics[width=1\textwidth]{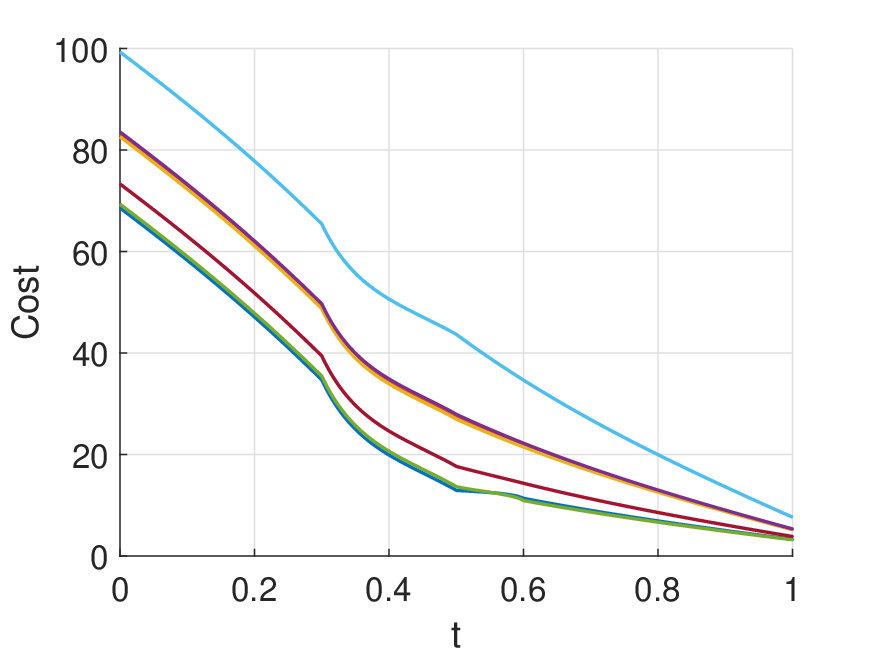}
		\caption{}
		\label{fig:NE_Opt_path_neg_D}
	\end{subfigure}
	\caption{\label{fig:NE_Opt_path_neg} The optimal paths of (A) mean $m^*(t)$, (B) variance $z^*(t)$, (C) controls $\alpha^*(t)$, and (D) cost to go, given the initial point $(m_0,z_0)=(-0.125,1)$. The solution by HJB is $\widehat{U}(t_0,m_0,z_0)=80.6807$, and estimated over optimal paths by Monte Carlo with 100 simulations is $J \approx 79.2191 \pm 2.6703$ (95\% CI: $[76.5488, 81.8895]$).}
\end{figure}

\section{Multidimensional example - Kalman filters with control}
\label{sec:kalman-filters-with}


In this section we generalize the linear quadratic example from Section \ref{sec:optimal-control-an} to the optimal control of a multidimensional Ornstein-Uhlenbeck process under noisy, discrete time Gaussian observations.
While the multidimensional case is practically important, we initially introduced the concepts in a one-dimensional setting for didactical purposes. Here, we demonstrate that the information update step following each observation can be effectively characterized using the Kalman filter.

Consider the multidimensional Ornstein-Uhlenbeck process in $\mathbb{R}^d$, with its controlled generator
\begin{equation}\label{eq:OU-generator}
	\mathcal{G}(u) f(x) = \sum_{i=1}^d\Big((-\theta_i x_i + u_i) \partial_{x_i} f(x) + \sum_{j=i}^{d}\frac{b^2_{ij}}{2} \partial_{x_ix_j} f(x)\Big).
\end{equation}
The sets of summation indices above are $\mathcal{I} = \{1,...,d\}$, $\mathcal{J}_i = \{i,...,d\}$, and we will later use the notation $\mathcal{J}^+_i = \{i+1,...,d\}$ as well. Following Remark \ref{rem:parameterized-HJB},  since we are working with multivariate Gaussian distributions, it is enough to choose $\varphi_{1,i}(x) = x_i$, $\varphi_{2,ij}(x) = x_ix_j$, $i,j=1,...,d$, and then represent the value function
\begin{gather}
	V(t,\mu) = U(t,\{\ip{\mu}{\varphi_{1,i}}\}_{i \in \mathcal{I}},\{\ip{\mu}{\varphi_{2,ij}}\}_{i \in \mathcal{I},~j \in \mathcal{J}_i}) = U(t,m,s),
\end{gather}
with the first moments $m = (m_1,...,m_d) \in \mathbb{R}^d$, the second moments, $s=(\{s_{ij}\}_{i=1,...,d,~~j \in \mathcal{J}_i})$, and their corresponding definitions, namely $m_i = \int_{\mathbb{R}^d}\varphi_{1,i}(x)\mu(dx)$, $s_{ij} = \int_{\mathbb{R}^d}\varphi_{2,ij}(x)\mu(dx)$. Thus, $U:\mathbb{R}^+\times \mathbb{R}^{1 + d + \frac{d(d+1)}{2}} \mapsto \mathbb{R}$ is a  real valued function with a particular domain, since the covariance matrix associated to $\mu$ has to be non negative definite.


Using the representation of the flat derivative \eqref{eq:change-of-variables-moments}, we obtain
\begin{align}\label{eq:rep_FD}
\nonumber	\mathcal{G}(u) \frac{\delta V}{\delta \mu}(t,\mu)(x) &= \sum_{i=1}^d \biggl\{ (-\theta_i x_i + u_i) \biggl( \sum_{l=1}^{d} \biggl[ \frac{\pa U}{\pa m_l}(t,m,s) \pa_{x_i} \varphi_{1,l}(x) + \sum_{k=l}^{d} \frac{\pa U}{\pa s_{lk}}(t,m,s) \pa_{x_i} \varphi_{2,lk}(x) \biggr] \biggl) \\
	&\quad + \sum_{j=i}^{d} \frac{b_{ij}^2}{2} \biggl( \sum_{l=1}^{d} \biggl[ \frac{\pa U}{\pa m_l} (t,m,s) \pa_{x_ix_j} \varphi_{1,l}(x) + \sum_{k=l}^{d} \frac{\pa U}{\pa s_{lk}}(t,m,s) \pa_{x_ix_j} \varphi_{2,lk}(x) \biggr] \biggr] \biggr\}.
\end{align}
Now we need to substitute the particular values of the partial derivatives of the $\varphi$ functions. Indeed, noting right away that $\pa_{x_ix_j} \varphi_{1,l} \equiv 0$, each of the terms in \eqref{eq:rep_FD} yield
\[
	\sum_{i=1}^d (-\theta_i x_i + u_i) \sum_{l =1}^d
	\frac{\pa U}{\pa m_l}(t,m,s) \pa_{x_i} \varphi_{1,l}(x) = \sum_{i=1}^d (-\theta_i x_i + u_i) \frac{\pa U}{\pa m_i}(t,m,s),
\]
\begin{multline*}
	\sum_{i=1}^d  (-\theta_i x_i + u_i) \sum_{l=1}^{d} \sum_{k=l}^{d} \frac{\pa U}{\pa s_{lk}}(t,m,s) \pa_{x_i} \varphi_{2,lk}(x) = \\
	\sum_{i=1}^d  (-\theta_i x_i + u_i) \biggl( 2 \frac{\pa U}{\pa s_{ii}}(t,m,s) \varphi_{1,i}(x) + \sum_{k =i+1}^d
	\frac{\pa U}{\pa s_{ik}}(t,m,s) \varphi_{1,k}(x) 
	+ \sum_{1 \le l <i} \frac{\pa U}{\pa s_{li}}(t,m,s) \varphi_{1,l}(x)
	\biggr),
\end{multline*}
and finally
\begin{multline*}
	\sum_{i=1}^d \sum_{j=i}^{d} \frac{b_{ij}^2}{2} \sum_{l=1}^{d} \sum_{k=l}^{d} \frac{\pa U}{\pa s_{lk}}(t,m,s) \pa_{x_ix_j} \varphi_{2,lk}(x) =\\
	\sum_{i=1}^d \biggl( b_{ii}^2 \frac{\pa U}{\pa s_{ii}}(t,m,s) 
	+ \sum_{j=i+1}^{d} \frac{b_{ij}^2}{2} \frac{\pa U}{\pa s_{ij}}(t,m,s)
	\biggr).
\end{multline*}



To compute the resulting Hamiltonian following \eqref{eq:Hamiltonian}, we need to use the value of $\mathcal{G}(u)\frac{\delta V}{\delta \mu}(t,\mu)$ from \eqref{eq:rep_FD},  
 which yields
\begin{align*}
	&\mathcal{H}\Bigg(t, \mu, \frac{\delta V}{\delta \mu}(t,\mu)\Bigg) 
	= \inf_{u \in \R^d} \left\{ \ip{\mu}{\mathcal{G}(u)\frac{\delta V}{\delta \mu}(t,\mu)} + \sum_{i=1}^d\ip{\mu}{\varphi_{2,ii}} + \sum_{i=1}^d C_iu^2_i \right\},\\
	&= \inf_{u \in \R^d}\biggl\{ \int_{\R^d} \Bigg[ \sum_{i=1}^d\frac{\pa U}{\pa m_i}(t,m,s) (-\theta_i x_i + u_i)   \\
	&\quad\quad+ \sum_{i=1}^d\Big(2x_i \frac{\pa U}{\pa s_{ii}}(t,m,s) + \sum_{j=i+1}^{d} x_j \frac{\pa U}{\pa s_{ij}}(t,m,s)+ 
	\sum_{1 \leq l < i} x_l \frac{\partial U}{\partial s_{li}}(t,m,s)
	\Big)(-\theta_i x_i + u_i) \\
	&\quad\quad+ \sum_{i=1}^d \Big(b^2_{ii} \frac{\pa U}{\pa s_{ii}}(t,m,s)  + 
	\sum_{j=i+1}^{d} \frac{b^2_{ij}}{2} \frac{\pa U}{\pa s_{ij}}(t,m,s)
	\Big)\Bigg] \mu(\dd x) + \int_{\R^d} \sum_{i=1}^d x^2_i \mu(\dd x) + \sum_{i=1}^d C_i u^2_i \biggr\}\\
	&=  \sum_{i=1}^d \Bigg[ s_{ii} - \theta_i m_i \frac{\pa U}{\pa m_i}(t,m,s)  + (b_{ii}^2- 2\theta_i s_{ii}) \frac{\pa U}{\pa s_{ii}}(t,m,s) \\
	&+\sum_{j=i+1}^{d}\Big(
	\frac{b^2_{ij}}{2}
	- \theta_i s_{ij} \Big)\frac{\pa U}{\pa s_{ij}}(t,m,s)-
	\theta_i  \sum_{1 \leq l < i}s_{li} \frac{\pa U}{\pa s_{li}}(t,m,s)
	\Bigg] \\
	&\quad\quad- \sum_{i=1}^d \f{1}{4C_i} \left( \frac{\pa U}{\pa m_i}(t,m,s) + 2 m_i \frac{\pa U}{\pa s_{ii}}(t,m,s) +\sum_{j=i+1}^{d} m_j \frac{\pa U}{\pa s_{ij}}(t,m,s) +
	\sum_{1 \leq l < i} m_l \frac{\partial U}{\partial s_{li}}(t,m,s)
	\right)^2. \label{eq:U_multiD}
\end{align*}

In the Hamiltonian above, 
$U$ is a function of the first and the second moments. We can transform the input second moment variables of $U$ into the corresponding entries of  the covariance matrix of $X,$ $\Sigma$, yielding a new function $\widehat{U}(t,m,z)$ as follows. Let $z_{ii} = s_{ii} - m_i^2$ and $z_{ij} = s_{ij} - m_im_j$, so that $\Sigma=\Sigma(z)$ satisfies $\Sigma_{ij} = z_{ij}$ if $i\le j$ and $\Sigma_{ij} = z_{ji}$ otherwise.
With this construction, $\Sigma$ is only symmetric and we will have to ensure later that it is nonnegative definite by properly defining the state domain.
The derivative $\frac{\partial U}{\partial m_i}(t,m,s)$, for $ i \in \mathcal{I}$, can be written using the new variables, namely:
\begin{gather}
	\frac{\partial U}{\partial m_i}(t,m,s) = \frac{\partial \widehat{U}}{\partial m_i}(t,m,z) - 2m_i \frac{\partial \widehat{U}}{\partial z_{ii}} (t,m,z) - \sum_{j=i+1}^d
	m_j \frac{\partial \widehat{U}}{\partial z_{ij}}(t,m,z).
\end{gather}
Similarly,  the partial derivatives of $U$ w.r.t. $s_{ii}$, $s_{ij}$ are replaced by the corresponding derivatives of $\widehat{U}$ w.r.t. $z_{ii}$, $z_{ij}$.
The Hamiltonian for this problem expressed in terms of $\widehat{U}$ and the new input variables then reads:
  \begin{align}
	& \mathcal{H}\Bigg(t, \mu, \frac{\delta V}{\delta \mu}(t,\mu)\Bigg) =
	 \hat{\mathcal{H}}\Bigg(t, m,z, D \widehat{U}(t, m,z)\Bigg)  \nonumber \\
%
	& =  \sum_{i=1}^d
	 \Bigg[ z_{ii} + m^2_i - \theta_i m_i \frac{\pa \widehat{U}}{\pa m_i}(t,m,z)  + (b_{ii}^2- 2\theta_i z_{ii}) \frac{\pa \widehat{U}}{\pa z_{ii}}(t,m,z) \\
	&- 
	\theta_i  \sum_{1 \leq l < i}z_{li} \frac{\pa \widehat{U}}{\pa z_{li}}(t,m,z)
	- 
	\theta_i  \sum_{1 \leq l < i} m_lm_i \frac{\pa \widehat{U}}{\pa z_{li}}(t,m,z)
	\\
	& + \sum_{j=i+1}^d
	 \Big( \frac{b^2_{ij}}{2}- \theta_i z_{ij}\Big) \frac{\pa \widehat{U}}{\pa z_{ij}}(t,m,z) -  \f{1}{4C_i} \left( \frac{\pa \widehat{U}}{\pa m_i}(t,m,z) +
	 \sum_{1 \leq l < i} m_l \frac{\partial \widehat{U}}{\partial z_{li}}(t,m,z)
	 \right)^2\Bigg]. \nonumber
	\label{Hamiltonian_multiD}
\end{align}
This Hamiltonian determines the evolution of the value function $\widehat{U}$ in between observations according to the HJB equation in its natural domain, namely $m$ in the same domain as $x$ and $z$ such that $\Sigma(z)$ is non negative definite.

\textbf{Update after each observation:}

At the observation time $t_i$ we gather the datum $Y^\alpha_i \coloneqq H\widehat{X}^\alpha_i + \varepsilon Z_i$ with
$\widehat{X}_i^\alpha \in \mathbb{R}^d$, and $Y_i^\alpha$, $Z_i$ $\in \mathbb{R}^n$. Here we have $\widehat{X}^\alpha_i \sim \mu_{t^-_i} = \mathcal{N}(m_{t_i^-}, \Sigma_{t_i^-})$ and $Z_i \sim \mathcal{N}(0,1)$, mutually independent. After observation, the conditional distribution of $\widehat{X}^\alpha_i|_{Y^\alpha_i=y}$ is Gaussian and can be computed using the Kalman filter as follows:
%
\begin{gather}
	\widehat{X}^\alpha_i|_{Y^\alpha_i=y} \sim \mathcal{N}\Big(m_{t_i}, \Sigma_{t_i}\Big) \\
	\text{with } m_{t_i}=m_{t^-_i} + K_{t_i}(y - Hm_{t^-_i}),\\
	\Sigma_{t_i} = (I-K_{t_i}H)\Sigma_{t^-_i}, \\
	\text{and the Kalman gain matrix } K_{t_i} = \Sigma_{t^-_i}H^T\Big(H\Sigma_{t^-_i}H^T + \epsilon^2I\Big)^{-1}
\end{gather}
%
We now introduce the auxiliary function $\tilde z: \R^{d\times d}\to \R^{d+d(d+1)/2}$ 
with $\tilde z(\Sigma)_{ij} = \Sigma_{ij}$ for $i\le j.$
Thus, according to \eqref{eq:HJB-observations} we update $\widehat{U}(t_i^-, m_{t^-_i}, \tilde z(\Sigma_{t^-_i}))$ as follows:
\begin{gather}\label{update_multiD}
	\widehat{U}(t_i^-, m_{t^-_i}, \tilde z(\Sigma_{t^-_i})) = \int_{\R^n} \widehat{U}\biggl( t_i, m_{t^-_i} +  K_{t_i}Lw, \tilde z((I-K_{t_i}H)\Sigma_{t^-_i}) \biggr) \phi(w) \dd w,
\end{gather}
where the matrix $L$ is such that $LL^T = H\Sigma_{t^-_{t_i}}H^T + \epsilon^2I$ and $\phi(w)$ is a standard Gaussian density in $\R^n$. The matrix $L$ can be computed using, for instance, a Cholesky decomposition.
\begin{remark}[Observation costs]
We can have additional, controllable costs, from the observation as in \eqref{eq:HJB-observations}. In that case, \eqref{update_multiD} preserves the same structure, and only an additional optimization step is necessary to find the optimal data acquisition setup $\beta.$
\end{remark}
Summing up, in this section we have characterized the different ingredients to set up HJB equation in our multivariate Gaussian process case. The Hamiltonian defined by (\ref{Hamiltonian_multiD}) determines the evolution of the value function between observation times $t_{i-1}$, $t_{i}$, $i=1,...,n$, and the conditional expectation at each of the observation times $t_{i}$ (\ref{update_multiD}) using the Kalman filter equations.
 It is important to emphasize at this point that the state dimension of the time dependent HJB PDE is $d+\frac{d(d+1)}{2}$, which is relatively large, even for $d=2$. This means that one should be particularly careful when choosing appropriate discretization tools, to address the ensuing curse of dimensionality. 

\textbf{Acknowledgments}
C.~Bayer acknowledges seed support for DFG CRC/TRR 388 ``Rough Analysis, Stochastic Dynamics and Related Topics. This work was supported by the KAUST Office of
Sponsored Research (OSR) under Award No. URF/1/2584-01-01 and the
Alexander von Humboldt Foundation. E.~Rezvanova, and R.~Tempone are
members of the KAUST SRI Center for Uncertainty Quantification in
Computational Science and Engineering.

\appendix

\section{Details on the characteristics}\label{sec:Details on characteristic}

As discussed in Section \ref{subsubsec:NA-LQ_control_domain}, it is necessary to truncate the originally unbounded domain to obtain a numerical solution for the proposed PDE (\ref{eq:parameterized-HJB-OU-1_num}). To effectively truncate the domain and construct a convergent numerical scheme, it is crucial to understand the propagation of information within the PDE. To achieve this, we compute the characteristics of equation (\ref{eq:parameterized-HJB-OU-1_num})  by following the methodology outlined in Section 3.2 of  \cite{evans@2010}. These characteristics provide the curves along which information flows from initial points throughout the domain in first-order PDEs. By knowing these curves, we can appropriately select a domain where the characteristic lines flow from inside the domain to the outside, ensuring that the numerical scheme is well-posed and convergent.

Now we compute the characteristics of \eqref{eq:parameterized-HJB-OU-1_num}. 
Letting $p=\frac{\partial U}{\partial m}$, $q=\frac{\partial U}{\partial z}$,
the characteristics are given by the  following system of ordinary differential equations (ODEs):
\begin{gather}
\begin{cases}
\frac{dm}{d\tau} = \theta m(\tau) + \frac{1}{2C} p(\tau), \\
\frac{dz}{d\tau} =  -(b^2 - 2 \theta z(\tau)), \\
\frac{dp}{d\tau} = 2m(\tau) - \theta p(\tau),\\
\frac{dq}{d\tau} = 1 - 2\theta q(\tau).
\end{cases}
\end{gather}
The solution to this system is given by:
\begin{gather}\label{eq:char_eq}
\begin{cases}
z(\tau) = \frac{b^2}{2 \theta} + C_1(m_0,z_0) e^{2 \theta \tau}, \\
q(\tau) = \frac{1}{2 \theta} + C_2(m_0,z_0) e^{-2 \theta \tau}, \\
p(\tau) = C_3(m_0,z_0) e^{\sqrt{\theta^2 + 1/C}\tau} + C_4(m_0,z_0) e^{-\sqrt{\theta^2 + 1/C}\tau}, \\
m(\tau) = (\theta + \sqrt{\theta^2 + 1/C}) \frac{C_3(m_0,z_0)}{2} e^{\sqrt{\theta^2 + 1/C}\tau} +...\\
 \quad\quad\quad+(\theta - \sqrt{\theta^2 + 1/C}) \frac{C_4(m_0,z_0)}{2} e^{-\sqrt{\theta^2 + 1/C}\tau}
\end{cases}
\end{gather}
where the constants $C_1$, $C_2$, $C_3$ and $C_4$ depend on the root of the characteristic line - the initial point $(m_0,z_0)$, and can be determined from equations for $m(0)$, $z(0)$ and $p(m(0),z(0))$, $q(m(0),z(0))$. The explicit values of the constants are the following:
\begin{gather}\label{eq:char_const}
\begin{cases}
C_1(m_0,z_0) = z_0 - \frac{b^2}{2\theta}, \\
C_2(m_0,z_0) = q_0(m_0, z_0) - \frac{1}{2\theta}, \\
C_3(m_0,z_0) = p(m_0, z_0) - C_4(m_0, z_0),
 \\
C_4(m_0,z_0) = \Big(\frac{\theta + \sqrt{\theta^2 + 1/C}}{2} p(m_0, z_0)-m_0\Big)\frac{1}{\sqrt{\theta^2+1/C}}.
\end{cases}
\end{gather}

Thus, the characteristic curves are given by (\ref{eq:char_eq}) - (\ref{eq:char_const}). Figure \ref{Char} in Section \ref{subsec:NA-LQ_control} shows an example of the characteristics for $m \in [-1,1]$ and $z \in [0,1]$. Choosing a domain which contains $m=0$ and $z=0.5$ ensures that the information flows across the boundary from inside of the domain to the outside.
A similar discussion applies to the multidimensional case.

\section{Derivation of the monotonicity condition}\label{sec:Derivation of the monotonicity condition}
In Section \ref{subsubsec:NA-LQ_control_num_scheme} we construct the numerical scheme which is guaranteed to converge to the true solution, provided that it is consistent and monotone. In this Appendix we derive the condition which ensures that our scheme is monotone. 
To do that, we check directly under which conditions the function $G$ is monotone in each of its arguments $U^{n}_{i-1,j}$, $U^{n}_{i,j}$, $U^{n}_{i+1,j}$, $U^{n}_{i,j-1}$, $U^{n}_{i,j+1}$. 

Let's show the monotonicity w.r.t. $U^{n}_{i,j}$ as an example. For a grid point $(m_i,z_j)$, let $\alpha_0$ be the point where the Hamiltonian $H_1$ changes sign, so that $\frac{\pa H_1}{\pa (\pa \hat{U}/\pa m)}(m_i,z_j,\alpha)(\alpha-\alpha_0) \geq 0$. Let us also use the following notation for convenience:
\begin{gather}
\alpha = \frac{U^n_{i,j}-U^n_{i-1,j}}{\Delta m},~~~\beta = \frac{U^n_{i+1,j}-U^n_{i,j}}{\Delta m}, \\
\gamma = \frac{U^n_{i,j+1} - U^n_{i,j}}{\Delta z},~~~\delta = \frac{U^n_{i,j} - U^n_{i,j-1}}{\Delta z}
\end{gather}

We write out the scheme $G$ explicitly as follows:
\begin{align}
U^{n+1}_{i,j} = &~U^n_{i,j} - \Delta t \Big[\indic{\{\beta \leq \alpha_0\}}\Big(H_1\Big(m_i,z_j,\frac{U^n_{i+1,j}-U^n_{i,j}}{\Delta m}\Big) - H_1(m_i,z_j,\alpha_0)\Big) \\
&+ (1-\indic{\{\alpha\leq\alpha_0\}})\Big(H_1\Big(m_i,z_j,\frac{U^n_{i,j}-U^n_{i-1,j}}{\Delta m}\Big)-H_1(m_i,z_j,\alpha_0)\Big) + H_1(m_i,z_j,\alpha_0) \\
& +\indic{\{b^2 - 2\theta z_j \geq 0\}}H_2\Big(m_i,z_j,\frac{U^n_{i,j+1} - U^n_{i,j}}{\Delta z}\Big) + \indic{\{b^2 - 2\theta z_j < 0\}}H_2\Big(m_i,z_j,\frac{U^n_{i,j} - U^n_{i,j-1}}{\Delta z}\Big)\Big]
\end{align}
The derivative of $U^{n+1}_{i,j}$ w.r.t. $U^n_{i,j}$ is given by the following:
\begin{align}
\frac{\partial U^{n+1}_{i,j}}{\partial U^{n}_{i,j}} = &1 - \Delta t \Big[\underset{:=h_1(m_i,z_j,\alpha,\beta)}{\underbrace{\indic{\{\beta \leq \alpha_0\}} \frac{\pa H_1}{\pa (\pa \hat{U}/\pa m)}(m_i,z_j,\beta)\Big(-\frac{1}{\Delta m}\Big) + (1-\indic{\{\alpha\leq\alpha_0\}}) \frac{\pa H_1}{\pa (\pa \hat{U}/\pa m)}(m_i,z_j,\alpha)\frac{1}{\Delta m}}} \\
+ &\underset{:=h_2(m_i,z_j,\gamma,\delta)}{\underbrace{\indic{\{b^2 - 2\theta z_j \geq 0\}} \frac{\pa H_2}{\pa (\pa \hat{U}/\pa z)}(m_i,z_j,\gamma)\Big(-\frac{1}{\Delta z}\Big) + \indic{\{b^2 - 2\theta z_j < 0\}} \frac{\pa H_2}{\pa (\pa \hat{U}/\pa z)}(m_i,z_j,\delta)\frac{1}{\Delta z}}}\Big]
\end{align}
Since $U^{n+1}_{i,j}$ is an increasing function, so $\frac{\partial U^{n+1}_{i,j}}{\partial U^{n}_{i,j}} \geq 0$. Recalling that $\frac{\pa H_1}{\pa (\pa \hat{U}/\pa m)}(m_i,z_j,\alpha)(\alpha-\alpha_0) \geq 0$ and expanding on $h_1(m_i,z_j,\alpha, \beta)$ we get 
\begin{align}
h_1(m_i,z_j,\alpha, \beta) &= 
\begin{cases}
-\frac{1}{\Delta m} \frac{\pa H_1}{\pa (\pa \hat{U}/\pa m)}(m_i,z_j,\beta),~~~~\text{if}~\beta,\alpha \leq \alpha_0, \\
-\frac{1}{\Delta m} \Big(\frac{\pa H_1}{\pa (\pa \hat{U}/\pa m)}(m_i,z_j,\beta)-\frac{\pa H_1}{\pa (\pa \hat{U}/\pa m)}(m_i,z_j,\alpha)\Big),~~~~\text{if}~\beta \leq \alpha_0,~\alpha \geq \alpha_0,\\
0,~~~~\text{if}~\beta \geq \alpha_0,~\alpha \leq \alpha_0,\\
\frac{1}{\Delta m} \frac{\pa H_1}{\pa (\pa \hat{U}/\pa m)}(m_i,z_j,\alpha),~~~~\text{if}~\beta,\alpha \geq \alpha_0
\end{cases}\\
&=
\begin{cases}
\frac{1}{\Delta m} \Big|\frac{\pa H_1}{\pa (\pa \hat{U}/\pa m)}(m_i,z_j,\beta)\Big|,~~~~\text{if}~\beta,\alpha \leq \alpha_0, \\
\frac{1}{\Delta m} \Big(\Big|\frac{\pa H_1}{\pa (\pa \hat{U}/\pa m)}(m_i,z_j,\beta)\Big|+\Big|\frac{\pa H_1}{\pa (\pa \hat{U}/\pa m)}(m_i,z_j,\alpha)\Big|\Big),~~~~\text{if}~\beta \leq \alpha_0,~\alpha \geq \alpha_0,\\
0,~~~~\text{if}~\beta \geq \alpha_0,~\alpha \leq \alpha_0,\\
\frac{1}{\Delta m}\Big|\frac{\pa H_1}{\pa (\pa \hat{U}/\pa m)}(m_i,z_j,\alpha)\Big|,~~~~\text{if}~\beta,\alpha \geq \alpha_0
\end{cases}
\end{align}
W.l.o.g., suppose that $\Big|\frac{\pa H_1}{\pa (\pa \hat{U}/\pa m)}(m_i,z_j,\alpha)\Big| \geq \Big|\frac{\pa H_1}{\pa (\pa \hat{U}/\pa m)}(m_i,z_j,\beta)\Big|$. Then 
\begin{gather}\label{cond_h1}
h_1(m_i,z_j,\alpha, \beta) \leq  \frac{2}{\Delta m}\Big|\frac{\pa H_1}{\pa (\pa \hat{U}/\pa m)}(m_i,z_j,\alpha)\Big|
\end{gather}
Expanding on $h_2(m_i,z_j, \gamma, \delta)$ we get
\begin{gather}
h_2(m_i,z_j,\gamma, \delta) = \underset{ \frac{\pa H_2}{\pa (\pa \hat{U}/\pa z)}(m_i,z_j,\gamma) \leq 0}{\underbrace{\indic{\{b^2 - 2\theta z_j \geq 0\}} \frac{\pa H_2}{\pa (\pa \hat{U}/\pa z)}(m_i,z_j,\gamma)\Big(-\frac{1}{\Delta z}\Big)}} + \underset{ \frac{\pa H_2}{\pa (\pa \hat{U}/\pa z)}(m_i,z_j,\delta) \geq 0}{\underbrace{\indic{\{b^2 - 2\theta z_j < 0\}} \frac{\pa H_2}{\pa (\pa \hat{U}/\pa z)}(z_j,\delta)\frac{1}{\Delta z}}} \\
=\indic{\{b^2 - 2\theta z_j \geq 0\}}\frac{1}{\Delta z}\Big|\frac{\pa H_2}{\pa (\pa \hat{U}/\pa z)}(m_i,z_j,\gamma)\Big| + \indic{\{b^2 - 2\theta z_j < 0\}}\frac{1}{\Delta z}\Big|\frac{\pa H_2}{\pa (\pa \hat{U}/\pa z)}(m_i,z_j,\delta)\Big|
\end{gather}
Suppose, w.l.o.g. that $\Big|\frac{\pa H_2}{\pa (\pa \hat{U}/\pa z)}(m_i,z_j,\gamma)\Big| \geq \Big|\frac{\pa H_2}{\pa (\pa \hat{U}/\pa z)}(m_i,z_j,\delta)\Big|$, then we have
\begin{gather}\label{cond_h2}
h_2(m_i,z_j, \gamma, \delta) \leq  \frac{1}{\Delta z}\Big|\frac{\pa H_2}{\pa (\pa \hat{U}/\pa z)}(m_i,z_j,\gamma)\Big|
\end{gather}
Combining (\ref{cond_h1}) and (\ref{cond_h2}), we get the condition which guarantees the monotonicity of the scheme $G$:
\begin{gather}\label{condition}
1 - 2\frac{\Delta t}{\Delta m}\Big|\frac{\pa H_1}{\pa (\pa \hat{U}/\pa m)}(m_i,z_j,\alpha)\Big| - \frac{\Delta t}{\Delta z}\Big|\frac{\pa H_2}{\pa (\pa \hat{U}/\pa z)}(m_i,z_j,\gamma)\Big| \geq 0
\end{gather}

To compute the relations $\Delta t/\Delta m$, $\Delta t/\Delta z$, which satisfy condition (\ref{condition}), we must estimate $\Big|\frac{\pa H_1}{\pa (\pa \hat{U}/\pa m)}(m_i,z_j,\alpha)\Big|$ and $\Big|\frac{\pa H_2}{\pa (\pa \hat{U}/\pa z)}(m_i,z_j,\gamma)\Big|$. The derivatives are given by the following
\begin{gather}
\frac{\pa H_1}{\pa (\pa \hat{U}/\pa m)}(m_i,z_j,\alpha) = -\theta m_i + \frac{\alpha}{2C}, \\
\frac{\pa H_2}{\pa (\pa \hat{U}/\pa z)}(m_i,z_j,\gamma) = 2\theta z_j - b^2
\end{gather}
On a domain defined by $[m_0,m_I] \times [z_0,z_J]$ with $m_0<0$, we have the following bounds
\begin{gather}
\Big|\frac{\pa H_1}{\pa (\pa \hat{U}/\pa m)}(m_i,z_j,\alpha)\Big| \leq \theta |m_0| + \frac{|\alpha|}{2C}, \\
\Big|\frac{\pa H_2}{\pa (\pa \hat{U}/\pa z)}(m_i,z_j,\gamma)\Big| \leq 2\theta z_J-b^2
\end{gather}
Obtaining $\Big|\frac{\pa H_2}{\pa (\pa \hat{U}/\pa z)}(m_i,z_j,\gamma)\Big|$ is straightforward, while $\Big|\frac{\pa H_1}{\pa (\pa \hat{U}/\pa m)}(m_i,z_j,\alpha)\Big|$ requires an a priori upper bound on $|\alpha|=|\partial \widehat{U}/\partial m|$. For example, the Lipschitz constant can be used if known. Otherwise, condition (\ref{condition}) could be verified to hold at every point of the grid $(m_i,z_j)$ during numerical simulation.
It is easy to see that the requirement on the derivatives of $U^{n+1}_{i,j}$ (or $G$) w.r.t. the rest of the variables $U^{n}_{i-1,j}$, $U^{n}_{i+1,j}$, $U^{n}_{i,j-1}$, $U^{n}_{i,j+1}$ to be positive is also satisfied by (\ref{condition}).



\printbibliography

\end{document}